\documentclass[11pt,a4paper,reqno]{amsart}
\usepackage[T1]{fontenc}
\usepackage{amsmath}
\usepackage{amsthm}
\usepackage{amsfonts}
\usepackage{amssymb}
\usepackage{graphicx}
\usepackage{amsbsy}
\usepackage{mathrsfs}
\usepackage{bbm}
\def\b{\beta}
\def\pa{\partial}
\def\e{\epsilon}
\newtheorem{theorem}{Theorem}[section]
\newtheorem{lemma}[theorem]{Lemma}

\newtheorem{proposition}[theorem]{Proposition}

\newtheorem{remark}[theorem]{Remark}

\numberwithin{equation}{section}
\catcode`@=11
\def\section{\@startsection{section}{1}%
  \z@{1.5\linespacing\@plus\linespacing}{.5\linespacing}%
  {\normalfont\bfseries\large\centering}}
\catcode`@=12
\newcommand{\be}{\begin{equation}}
\newcommand{\ee}{\end{equation}}
\newcommand{\bea}{\begin{eqnarray}}
\newcommand{\eea}{\end{eqnarray}}
\newcommand{\bee}{\begin{eqnarray*}}
\newcommand{\eee}{\end{eqnarray*}}

\def\pa{\partial}

\def\RR{\mathbb{R}}

\def\fref#1{{\rm (\ref{#1})}}

\catcode`@=11
\def\supess{\mathop{\operator@font Sup\,ess}}
\catcode`@=12

\def\bt{\tilde{b}}

\def\RR{\mathbb{R}}

\def\e{\varepsilon}

\def\fref#1{{\rm (\ref{#1})}}

\def\R2+{\RR ^2_+}

\def\htl{\tilde{H}_{\l}}
\def\htm{\tilde{H_{\mu}}}

 \def\bG{{\bf \Gamma}}

\def\A{\mathcal A}

\def\lsl{\frac{\lambda_s}{\lambda}}

\def\pa{\partial}

\def\lim{\mathop{\rm lim}}

\def\e{\varepsilon}

\def\l{\lambda}

\def\log{{\rm log}}

\def\et{e_\tau}

\def\zl{Z_{\lambda}}
\def\lsl{\frac{\lambda_s}{\lambda}}

\def\alphat{\tilde{\alpha}}
\def\betat{\tilde{\beta}}
\def\tgamma{{\tilde{\gamma}}}

\def\vul{V_{\lambda}^{(1)}}

\def\gammat{\tilde{\gamma}}

\def\G{\Gamma}

\def\S{\mathcal S}
\def\talpha{\tilde{\alpha}}
\def\tbeta{\tilde{\beta}}
\def\alphah{\hat{\alpha}}

\def\betah{\hat{\beta}}

\def\gammah{\hat{\gamma}}

\def\pa{\partial}

\def\tt{\tilde{T}}

\def\pa{\partial}

\def\Bd{B_{\delta}}

\def\S{\mathcal S}
\def\bW{{\bf W}}
\def\hbW{\widehat{\bW}}
\def\bw{{\bf w}}
\def\hbw{\widehat{\bw}}

\def\hJ{\widehat{J}}
\def\H{\Bbb H}
\def\nab{\nabla}
\def\bA{{\bf A}}
\def\tbw{\widetilde{\bw}}

\def\bF{{\bf F}}
\def\A{\Bbb A}
\def\mu{\l}
\def\bPsi{{\bf \Psi}}
\def\btPsi{\tilde{\bPsi}}
\def\tPsi{\tilde{\Psi}}
\def\bE{{\bf E}}
\def\Mod{\mbox{Mod}}
\def\matchal{\mathcal}
\title[]{Blow up dynamics for smooth equivariant solutions to the energy critical Schr\"odinger map}
\author[F. Merle]{Frank Merle}
\address{Universit\'e de Cergy Pontoise and IHES, France}
\email{merle@math.u-cergy.fr}
\author[P. Rapha\"el]{Pierre Rapha\"el}
\address{IMT, Universit\'e Toulouse III, France}
\email{pierre.raphael@math.univ-toulouse.fr}
\author[I. Rodnianski]{Igor Rodnianski}
\address{Mathematics Department, Princeton University, USA}
\email{irod@math.princeton.edu}

\begin{document}
\maketitle

\begin{abstract}
We consider the energy critical Schr\"odinger map problem with the 2-sphere target for equivariant initial data of homotopy index $k=1$. We show the existence of a codimension one set of smooth well localized initial data arbitrarily close to the ground state harmonic map 
in the energy critical norm, which generates finite time blow up solutions. We give a sharp description of the corresponding singularity formation which occurs by concentration of a universal bubble of energy.  
\end{abstract}


\section{Introduction}


\subsection{Setting of the problem} 

We consider in this paper the energy critical\\
Schr\"odinger map 
\be
\label{nlsmap}
\left\{\begin{array}{ll}\pa_tu=u\wedge \Delta u,\\u_{|t=0}=u_0 \in \dot{H}^1 \end{array}\right . \ \ (t,x)\in \Bbb R\times \Bbb R^2, \  \ u(t,x)\in \Bbb S^2.
\ee
This equation is related to the Landau-Lifschitz equation in ferromagnetism and it is a special  case of the Schr\"odinger flow for maps from a Riemannian manifold into a K\"ahler manifold, see \cite{Grillakis}, \cite{ding}. It 
belongs to a class of geometric evolution equations \cite{struweheatflow}, \cite{Q}, \cite{QT},  \cite{PT}, \cite{heatflow}, including wave  maps and the harmonic heat flow, which have attracted a considerable attention in the past and more recently. 
The Hamiltonian structure of the problem implies conservation of the Dirichlet energy 
\be
\label{dircehit}
E(u(t))=\int_{\Bbb R^2}|\nabla u(t,x)|^2dx=E(u_0)
\ee
which is moreover invariant under the action of symmetric transformations 
\be\label{symmetry}
u(t,x)\mapsto u_{\lambda,O}(t,x)=Ou(\frac t{\lambda^2},\frac x{\lambda}),\quad (\lambda,O)\in {\Bbb R}_+^*\times \mathcal O(\Bbb R^3).
\ee

The problem of global existence of large data solutions or, on the contrary, the possibility of a finite blow up and singularity formation corresponding to a concentration of energy has been addressed recently in detail for the wave map problem -- the wave analogue of \fref{nlsmap} -- and the Yang-Mills equations, see \cite{T}, \cite{ST}, \cite{KS1} for the large data wave map global regularity problem; 
\cite{RaphRod} and references therein, \cite{RodSter}, \cite{KST} (see also \cite{struweheatflow}, \cite{Q} \cite{QT}  \cite{PT}, \cite{GS}, \cite{heatflow} for the heat flow), and has been until now open for the Schr\"odinger map problem.\\

A specific class of solutions with additional symmetry preserved by the Schr\"odinger flow is given by k-equivariant maps, which take the form 
\be
\label{defR}
u(t,x)=e^{k\theta R}\left|\begin{array}{lll} u_1(t,r)\\ u_2(t,r)\\u_3(t,r),\end{array}\right .\ \ R=\left(\begin{array}{lll} 0 & -1 & 0 \\ 1 & 0 &0\\ 0 & 0& 0
\end{array}
\right ),\ee
where $(r,\theta)$ are the polar coordinates on $\Bbb R^2$, and $k\in \Bbb Z^*$ is the homotopy degree of the map,
explicitely: $$k=\frac{1}{4\pi}\int_{\Bbb R^2}(\pa_1u\wedge\pa_2u)\cdot u.$$

The k-equivariant harmonic map 
\be
\label{defqk}
Q_k=\left|\begin{array}{lll} \frac{2r^k}{1+r^{2k}}\cos(k\theta)\\  \frac{2r^k}{1+r^{2k}}\sin(k\theta)\\ \frac{1-r^{2k}}{1+r^{2k}}\end{array} \right., \ \ k\in \Bbb Z
\ee
is a stationary solution of the Schr\"odinger map equation. In a given homotopy class, $Q_k$ minimizes 
the Dirichlet energy \fref{dircehit} with 
$$
E(Q_k)=4\pi |k|.
$$
Moreover,  a classical consequence of the Bogomol'nyi's factorization \cite{Bog} is that, up to symmetries, $Q_k$ is the 
global minimizer of $E$ in a given $k$-equivariant homotopy class.

In the general case without $k$-equivariant symmetry 
local existence and uniqueness of smooth solutions goes back to \cite{bardosulem},
the small energy data global existence result is shown in (\cite{BT2}) and a conditional global result for solutions with energy below that of the ground state
$Q_1$ is given in \cite{PS}.

For the $k$-equivariant problem, the Cauchy problem is well-posed in $\dot{H}^1$ if the energy $E$ is sufficiently small, 
\cite{chang} or, more 
generally, if the energy $E$ is sufficiently close to the minimum in a given homotopy class $k$, realized on a harmonic 
map $Q_k: \Bbb R^2\to \Bbb S^2$, \cite{T3} \cite{T4}.
For large degree $k\geq 3$, this solution is stable, in fact, asymptotically stable by the result of Gustaffson, Nakanishi and Tsai \cite{T2}. For $k=1$ which corresponds to least energy maps, Bejenaru and Tataru \cite{BT} exhibit some instability mechanism of $Q\equiv Q_1$ in the scale invariant space $\dot{H}^1$.\\


\subsection{On energy critical geometric equations}


The Schr\"odinger map problem \fref{nlsmap} can be rewritten by application of $u\wedge $ as
$$u\wedge \pa_tu=-\Delta u-|\nabla u|^2u,$$ and in this form it appears  as the Schr\"odinger version of other energy critical geometric equations: the parabolic harmonic heat flow from crystal physics and ferromagnetism (see e. g. \cite{heatflow}, \cite{matano} for an introduction to these class of problems):
\be
\label{heatflow}
(\mbox{Heat flow}) \ \ \left\{\begin{array}{ll}\pa_tu=\Delta u+|\nabla u|^2u\\u_{t=0}=u_0\end{array} \right.\ \ (t,x)\in \Bbb R_+\times \Bbb R^2, \ \ u(t,x)\in \mathcal S^2.
\ee
and the wave map problem (see eg \cite{MS}, \cite{RaphRod}):
\be
\label{wavemap}
(\mbox{Wave map}) \ \ \left\{\begin{array}{ll}\pa_{tt}u- \Delta u=(|\nabla u|^2-|\pa_tu|^2)u\\u_{t=0}=u_0, \ \ \pa_tu_{|t=0}=u_1\end{array} \right.\ \ (t,x)\in \Bbb \times \Bbb R^2, \ \ u(t,x)\in \mathcal S^2.
\ee
For equations \fref{heatflow}, \fref{wavemap} a special class of k-equivariant solutions arises from the co-rotational symmetry of degree k, where $u$ takes the form
$$u(t,x)=\left|\begin{array}{lll}\sin(\phi(t,r))\cos(k\theta),\\ \sin(\phi(t,r))\sin(k\theta)\\\cos(\phi(t,r)\end{array}\right. \ \ \mbox{i.e.} \ \ u_2(t,r)\equiv 0.$$ In this class the full problem reduces to a radially symmetric semilinear equation for the Euler angle $\phi(t,r)$. 
In a given homotopy class the harmonic map $Q_k$ is also the least energy stationary solution. 

and a general problem which has attracted a considerable attention for the past ten years is:\\

{\it Describe the flow for initial data near $Q_k$.}\\
 
 In particular, the question of existence of singular dynamics for non trivial topology $|k|\geq 1$ has been the heart of important numerical, formal and rigorous works, see e.g. \cite{B1}, \cite{B2}, \cite{heatflow}, \cite{struweheatflow}, \cite{Struwe}.  For all three problems, singularity formation, if occurs, is expected to lead to the bubbling off of a non trivial harmonic map. 
 It is important to stress the significance of the $|k|=1$ case which, due to the inner structure of $Q_k$, is expected to be the only case where {\it stable} singular dynamics could give insight into a mechanism of singularity formation for generic data near 
 $Q_k$ {\it without}
 symmetry.\\
 
 In the parabolic case, the existence of blow up solutions for $k=1$ is known, we refer to \cite{heatflow}, \cite{T2}, \cite{matano} for an introduction to the parabolic problem. Near $Q_k$, blow up is {\it ruled out} in \cite{T2} for $k\geq 3$ where the harmonic map is proved to be asymptotically stable, and an infinite time blow up is shown to exist for $k=2$ emerging from slowing decaying at infinity initial data. The sharp description of the singularity formation for $k=1$ is still mostly open.\\
  
  For the (WM) problem, important progress have been made on the existence and description of the singularity formation in 
 \cite{RodSter}, \cite{KST}, \cite{RaphRod} where $Q_k$ is shown to be unstable by blow up for all $k\geq 1$. In particular,  \cite{RaphRod} obtained a complete description of the {\it stable} blow up regime for the Wave Map problem \fref{wavemap} emerging from smooth well localized data with co-rotational symmetry, and for all homotopy numbers $k\geq 1$. The exact blow up speed in this regime is derived, and a conceptual connection is made with the singularity formation problem for the $L^2$ critical Nonlinear Schr\"odinger equation as studied by Merle and Rapha\"el, \cite{MR1}, \cite{MR2}, \cite{MR3}, \cite{MR4}, \cite{MR5}, see also Perelman \cite{P}.\\
  
 For the Schr\"odinger map problem \fref{nlsmap} under equivariant symmetry, blow up has been {\it ruled out} again for $k\geq 3$ in  \cite{T2}, \cite{T3}, \cite{T4}, where $Q_k$ is shown to be asymptotically stable. For $k=1$, Bejenaru and Tataru showed 
 that $Q_1$ is stable under smooth well localized perturbation and also established instability of $Q_1$ in the $\dot{H}^1$ topology. This still leaves open the question of existence and stability of  a singular dynamics.


\subsection{Statement of the result}


Continuing the line of investigation started in \cite{MR4}, \cite{RaphRod} for the nonlinear Schr\"odinger and wave map problems, we establish the existence of a finite time blow up regime for the $k=1$ equivariant Schr\"odinger map  problem together with sharp asymptotics on the singularity formation.

\begin{theorem}[Existence and description of the blow up Schr\"odinger map dynamics for $k=1$]
\label{thmmain} There exists a set of smooth well localized 1-equivariant initial data with elements arbitrarily close to $Q\equiv Q_1$ in the $\dot{H}^1$ topology such that for all initial data in this set, the corresponding solution to \fref{nlsmap} blows up in finite time. The singularity formation corresponds to the concentration of a universal bubble of energy in the scale invariant energy space:
\be
\label{strongvonegrebe}
u(t,x)-e^{\Theta(t)R}Q\left(\frac{x}{\lambda(t)}\right)\to u^*\ \ \mbox{in} \ \ \dot{H^1} \ \ \mbox{as}\ \ t\to T
\ee 
for some parameters $(\lambda(t),\Theta(t))\in \mathcal C^1([0,T),\Bbb R^*_+\times \Bbb R)$ with the asymptotic behavior near blow up time:
\be
\label{speed}
\lambda(t)=\kappa(u)\frac{T-t}{|\log (T-t)|^{2}}\left(1+o_{t\to T}(1)\right),\ \ \kappa(u)>0,
\ee
\be
\label{convergenetheta}
\Theta(t)=\Theta(u)(1+o(1)), \ \ \Theta(u)\in \Bbb R.
\ee
Moreover, there holds the propagation of regularity:
\be
\label{htworegulairity}
\Delta u^*\in L^2.
\ee
\end{theorem}

{\it Comments on the result}\\

{\it 1. Stability vs blow near $Q$}: In \cite{BT}, the authors in particular obtained the following statements. First, $\forall \e>0$, $\exists \delta>0$ such that $\|u_0-Q\|_{H^1}<\delta$ implies that the corresponding solution to \fref{nlsmap} is globally defined and $$\forall t\in \Bbb R, \ \ \|u(t)-Q\|_{\dot{H}^1}<\e.$$ Second, there exists $\e_0>0$ such that $\forall \delta>0$, there exists an initial data $u_{0,\delta}$ such that $$\|u_{0,\delta}-Q\|_{\dot{H^1}}<\delta \ \ \mbox{and} \ \ \exists t_\delta\ \ \mbox{with} \ \ \|u_\delta(t_{\delta})-Q\|_{\dot{H}^1}=\e_0.$$ 
In contrast to the first part of the statement above, which has the appearance of a result on stability of $Q_1$,  the initial data considered in Theorem \ref{thmmain} are smooth, {\it large} in the inhomogneneous $H^1$ topology but arbitrarily close to $Q$ in the homogeneous $\dot{H}^1$ topology, that is: $$\|\nabla u_0-\nabla Q\|_{L^2}\ll1\ \ \mbox{but}\ \ \|u_0-Q\|_{L^2}\gtrsim 1.$$

{\it 2. On the instability on blow up by rotation}: The blow up speed \fref{speed} is conjectured in \cite{heatflow} to
accompany a  {\it generic stable} singularity formation for the harmonic heat flow \fref{heatflow}. The main result of this paper given by Theorem \ref{thmmain} shows that such a singularity formation regime
 also exists for the Schr\"odinger map flow, but it is no longer generic {\it due to a completely new instability mechanism generated by the coupling between the dynamics of the scaling and phase parameters}, see the strategy of the proof below. The rigorous derivation of {\it stable} blow up dynamics for $k=1$ co-rotational data for the harmonic heat flow with the blow up speed \fref{speed} is given in the forthcoming paper \cite{RSstudent}. Moreover, while the blow up dynamics exhibited in \cite{RaphRod}, \cite{RSstudent} for respectively the Wave Map and the Harmonic Heat flow are stable within the more restricted class of corotational symmetry, we expect the same instability mechanism by rotation freedom to occur for generic equivariant perturbations.\\
 
{\it 3. On the codimension one instability}: The initial data in Theorem \ref{thmmain} are constructed to to form a set of codimension one in some weak sense to arrest the intrinsic instability of blow up dynamics induced by the rotation symmetry. Namely, given $0<b_0\ll1$ and a smooth well localized and small enough \footnote{with respect to $b_0$} initial data $v_0$, we can find a parameter $a_0(b_0,v_0)$ such that the initial data $$Q_{a_0(b_0,v_0),b_0}+v_0$$ generates a finite time blow up solution in the regime described by Theorem \ref{thmmain}, where $Q_{a_0,b_0}$ is a suitable small two parameters deformation\footnote{see the strategy of the proof below} of $Q_1$. This relates to the construction of manifold of unstable blow up solutions performed in \cite{BW}, \cite{KS}, \cite{MRS} for the $L^2$ critical NLS and \cite{HillRaph} for the energy critical wave equation.\\

{\it 4. Propagation of regularity}: The propagation of regularity \fref{htworegulairity} shows a major difference with the wave map problem \cite{RaphRod} or the mass critical NLS problem \cite{MR5}, where the remainder is proved to barely belong to the scale invariant space, while a full derivative is propagated here. Using the estimates in this paper one could show  that $$\forall s>0, \ \ \Delta ^{1+s}u^*\notin L^2$$ and in this sense \fref{htworegulairity} is sharp, see again \cite{MR5} for related properties. The regularity of the remainder  is deeply connected to the blow up speed, see \cite{MR5}, \cite{RS} for a further discussion of a similar issue.\\

{\bf Notations}: $(r,\theta)$ and $(y,\theta)$  with $r=\frac y\lambda$ and $\lambda$ -- the scaling parameter --
will denote the polar coordinates on $\Bbb R^2$. 
We set  $$\pa_{\tau}=\frac{1}{r}\pa_{\theta}, \ \ 
\Lambda f=y\cdot\nabla_y f, \ \ Df=f+y\cdot\nabla  f.$$ For  a given parameter $b>0$ we introduce the scales
\be
\label{defbzerobone}
B_0=\frac{1}{\sqrt{b}}, \ \ B_1=\frac{{|\log b|}}{\sqrt{b}}.
\ee

 \subsection{Strategy of the proof} In what follows we detail the main steps of the proof of Theorem \ref{thmmain}.\\ 
 
 {\bf step 1} Renormalization and choice of gauge.\\
 
Let $u(t,x)$ be a $k=1$ equivariant solution of the Schr\"odinger map flow, close to $Q$ in the energy topology and
potentially blowing up at $t=T$.
By the local Cauchy theory and the variational characterization of $Q$
it can be written in the form: $$u(t,x)=e^{\Theta(t) R}(Q+\e)\left(t,\frac{r}{\lambda(t)}\right)$$ with$$ \|\e(t)\|_{\dot{H}^1}\ll 1\ \ \mbox{and} \ \ \lambda(t)\to 0 \ \mbox{as} \ \ t\to T.$$
 It is therefore natural to pass to a renormalized function $v(s,y)$: 
 \be
 \label{selsimilar}
 u(t,x)=e^{\Theta R}v(s,y), \ \ \frac{ds}{dt}=\frac{1}{\lambda^2}, \ \ y=\frac{x}{\lambda}
 \ee
determined up to  the unknown modulation variables $(\lambda(t),\Theta(t))\in \Bbb R^*_+\times \Bbb R$. 
This maps \fref{nlsmap} into the corresponding equation for $v$:
 \be
 \label{cbieeife}
 \pa_s v+\Theta_s Rv-\lsl\Lambda v=v\wedge \Delta v, \ \ (s,y)\in \Bbb R_+\times \Bbb R^2.
 \ee
 In order to understand the renormalized equation \fref{cbieeife} in the vicinity of the harmonic map $Q$, we need to chose a gauge to describe $v$. A suitable gauge is provided by the Frenet basis associated to $Q$: $$e_r=\frac{\pa_yQ}{|\pa_yQ|}, \ \ e_{\tau}=\frac{\pa_\tau Q}{|\pa_\tau Q|}, \ \ Q(y,\theta)=e^{R\theta}\left|\begin{array}{lll}\Lambda\phi\\0\\Z\end{array}\right.$$ where we introduced the explicit functions describing the ground state: 
 \be
 \label{cnkfehohef}
\phi(y)=2\tan^{-1}(y), \ \  \Lambda\phi=y\pa_y\phi=\frac{2y}{1+y^2}, \ \ Z(y)=\frac{1-y^2}{1+y^2}.
 \ee 
At each point $(y,\theta)$ the triple $(e_r,e_\tau,Q)$ ia  an orthonormal basis in ${\Bbb R}^3$ with $(e_r,e_\tau)$ 
spanning $T_{Q(y)}\Bbb S^2$. The equivariant symmetry assumption is equivalent to the statement that the expansion of $v$
relative to the Frenet basis  is given by the  {\it spherically symmetric} coordinate functions: 
$$v(s,y)=\alphah(s,y)e_r+\betah(s,y)e_\tau+(1+\gammah(s,y))Q, \ \ \alphah^2+\betah^2+(1+\gammah)^2=1.$$ 
 
 {\bf step 2} A slowly modulated approximate solution.\\
 
 We now construct a slowly modulated ansatz for the equation \eqref{cbieeife}. 
 This procedure is similar in spirit to the constructions in \cite{P}, \cite{MR2}, \cite{MR4}, \cite{KMR}, \cite{RS}. 
 More precisely, we impose the modulation equations 
 \be
 \label{scalingandphase}
 -\lsl=b, \ \ -\Theta_s=a
 \ee
  and look for an approximate solution of the form $$v(s,y)=v_{a(s),b(s)}(y)$$ for the unknown maps: $$(a,b)\mapsto v_{a,b}(y), \ \ s\mapsto(a(s),b(s)).$$  After linearization near $Q$, the equation \fref{cbieeife} takes the form of a 
  nonlinear (in fact, quasilinear) Schr\"odinger system driven by the $(a,b)$ parameters:
  \be
  \label{aprpozo}
 \left\{\begin{array}{ll} \pa_s\alphah-H\betah=-b\Lambda\phi-b\Lambda \alphah-a\betah Z+N_1(\alphah,\betah)\\\pa_s\betah+H\alphah=a\Lambda\phi-b\Lambda \betah+a\alphah Z++N_2(\alphah,\betah)\end{array}\right.
  \ee 
  where $H$ is the linearized Hamiltonian $$H=-\Delta +\frac{V(y)}{y^2},  \ \ V(y)=\frac{y^4-6y^2+1}{(1+y^2)^2}.$$ The general strategy is then to build an approximate solution via an asymptotic expansion relative to the small parameters $(a,b)$:
  \be
  \label{explaissnos}
  \alphah=aT_{1,0}(y)+b^2T_{0,2}(y)+a^2T_{2,0}(y), \ \ \betah=bT_{0,1}(y)+abT_{1,1}+\mbox{lot}
  \ee
  and to chose the law for the parameters 
  \be
  \label{cnoeoeheho}
  a_s=-c_{1,1}ab+\dots, \ \ b_s=-c_{0,2}b^2-c_{2,0}a^2+\dots...
  \ee 
  which yields solutions $T_{i,j}$ of the generated elliptic equations 
  $$HT_{i,j}=N_{i,j}(\Lambda\phi, (T_{k,l})_{0\leq k<i, 0\leq l<j})$$
  with  {\it least possible growth} in $y$.  A general non trivial growth of solutions to the inhomogeneous problem $Hu=f$ is 
  induced by a {\it resonance} $\Lambda\phi$ for the Schr\"odinger operator $H$: $$H(\Lambda \phi)=0,$$
 generated by the symmetry group \fref{symmetry}.
  At the order $b^1$ the expansion \fref{explaissnos} gives the equation $$HT_{1,0}=HT_{0,1}=\Lambda \phi$$ which admits an explicit solution $T_{0,1}=T_{1,0}=T_1$ with growth at infinity: 
  \be
  \label{vnoeohoehri}
  T_1(y)= -y\log y+-y +O(1), \ \ \Lambda T_1=-y\log y+O(1)\ \ \mbox{as} \ \ y\to +\infty.
  \ee 
 Examining the $b^2$ order terms in the second equation of \fref{aprpozo} we obtain
 $$HT_{0,2}=-\Lambda T_1+c_{0,2}T_1+N(\Lambda\phi,T_1).
 $$ 
 The $y\log y$ growth of $\Lambda T_1$ can be compensated in \fref{vnoeohoehri} by the similar growth of $T_1$ and the choice 
 of $$c_{0,2}=1.$$ 
Furthermore, similar to the construction for the wave map problem in \cite{RaphRod}, an additional explicit correction can  
be introduced to eliminate the remaining $y$ growth in \fref{vnoeohoehri}.  
This correction leads to a non trivial {\it flux computation} in the parabolic region $y\leq\frac{1}{\sqrt{b}}$ and 
gives a leading  order ODE for the parameter $b$: 
$$
b_s=-b^2\left(1+\frac{2}{|\log b|}\right)+\mbox{lot}.
$$ 
Similarily, the $ab$ order term in the first equation of \fref{aprpozo} leads to 
$$
HT_{1,1}=-(\Lambda T_1-T_1)+c_{1,1}T_1+N(\Lambda\phi,T_1)$$ where we used from \fref{cnkfehohef}: $$Z(y)=-1+O\left(\frac{1}{y^2}\right)\ \ \mbox{as} \ \ y\to+\infty.$$ 
This choice $c_{1,1}=0$ ensures that $T_{1,1}$ has the least growth as $y\to \infty$. After an additional correction
coupled with the flux computation we obtain the leading order ODE for $a$: 
$$
a_s=-2\frac{ab}{|\log b|}+\mbox{lot}.
$$ 
The expansion \eqref{explaissnos}, \fref{cnoeoeheho} results in the construction of {\it moderately} growing smooth profiles $T_{i,j}$ such that the ansatz \fref{explaissnos} gives a high order approximate solution to the equation \fref{cbieeife} {\it for the choice of the modulation parameters} driven by the system of ODE's \fref{scalingandphase} 
$$
\lambda_s=-b\lambda,\ \ \Theta_s=-a
$$
coupled with
  \be
  \label{systemodecomplete}
 b_s=-b^2\left(1+\frac{2}{|\log b|}\right)-a^2+O\left(\frac{b^2}{|\log b|^{1+\delta}}\right), \ \ a_s=-\frac{2ab}{|\log b|}+O\left(\frac{b^2}{|\log b|^{1+\delta}}\right)
  \ee
  A spectacular feature of this system is that, generically in the regime of small $a,b$ with positive initial value of $b$, 
  the phase speed $a$ dominates the concentration velocity $b$ and, through the term $-a^2$, turns $b$ negative, thus 
  arresting the concentration behavior where $\lambda(s)\to 0$.
 Nonetheless, for a given initial $b_0>0$, one can find a locally unique $a_0(b_0)$ such that the corresponding 
 solution of the $a$ equation 
 obeys the bound  
 $$
  |a|\ll \frac{b}{|\log b|^{1+\delta}}.
  $$
 In this non-generic case, the remaining ODE's reduce to the system
  $$
  -\lsl=b, \ \-\Theta_s=a, \ \ b_s+b^2=-\frac{2b^2}{|\log b|}.
  $$ 
  Integrating and using the scaling \fref{selsimilar} now easily leads to finite time blow up with the asymptotics \fref{speed}.\\
  
  {\bf step 3} Controlling radiation: the mixed energy/Morawetz Lyapunov functional.\\
  
 Let $Q+\tbw_{a,b}$ denote the approximate solution constructed in step 2. We now decompose the solution $u$ relative to 
 the Frenet basis:
 $$u(t,x)=e^{\Theta(t)R}(Q+\tbw_{a,b}+\bw)(s,y)$$ where the uniqueness of the decomposition is ensured through the choice of {\it four} suitable orthogonality conditions on $\bw$, associated with the modulation parameters $(\lambda,\Theta,a,b)$. 
 The rigorous derivation of the system of modulation equations \fref{systemodecomplete} requires a careful control of the remainder 
 radiation term $\bw$ and relies on the estimate:
 \be
 \label{cneoheoheoh}
 \int_{{\Bbb R}^2}\frac{|\bw(s)|^2}{1+y^8}\lesssim \frac{b^4(s)}{|\log b(s)|^2},
 \ee 
 which in view of the expected behavior $b(s)\to 0$ as $s\to \infty$ has the appearance of a dispersive estimate for a solution
 of the Schr\"odinger equation.
The $y^8$ weight is dictated by the slow decay of the ground state \fref{cnkfehohef}. From our choice of {\it four} orthogonality conditions and the explicit knowledge of the kernel of $H$, \fref{cneoheoheoh} is implied by the energy bound:
 \be 
 \label{enegybound}
 \mathcal E_4= \|H^2\bw\|_{L^2}^2\lesssim \frac{b^4}{|\log b|^2}.
 \ee
 Following the strategy developed in \cite{RaphRod}, \cite{RS}, we return to the original variables $(t,r)$, in which 
 $$\bW(t,r):=\bw(s,y)$$ satisfies the equation:
 $$i\pa_t \bW+H_\l\bW=\frac{1}{\l^2}\Psi(s,y)+\mbox{lot},$$ where $\Psi$ is the error in the construction of the 
 approximate solution $\tbw_{a,b}$, and where $$H_\l=-\Delta +\frac{V_\l}{r^2},  \ \ V_\lambda(r)=V(y)$$ is the renormalized linearized operator with a potential $V_\l$. We compute the appropriately constructed energy identity providing control of the
 $4$-th derivative of $\bW$. Time dependence
 of the Hamiltonian $H_\l$ leads to the appearance of quadratic terms which require the use of an additional Morawetz type identity. Here an important simplification of the analysis is provided by the {\it factorization} properties of $H$, see \fref{factojfh}, which are a consequence of the Bogomol'nyi's factorization or, equivalently, certain {\it implicit} repulsive properties of $H_\l$. The result is  a mixed energy/Morawetz Lyapunov control of the form 
 \be
 \label{vneoheoheor}
 \frac{d}{dt}\left\{ \frac{\mathcal E_4}{\l^6}\right\}\lesssim \frac{b}{\l^8}\frac{b^4}{|\log b|^2}
 \ee where the size of the RHS is dictated by the error in the construction of the approximate solution. Integration of \fref{vneoheoheor} in the expected regime,  $$\lambda\sim \frac{b}{|\log b|^2},$$ yields the bound \fref{enegybound}.\\
 The strategy, as described, would present insurmountable difficulties due to the specific quasilinear structure of the Schr\"odinger map problem. Indeed, in the chosen gauge, the nonlinear terms in \fref{aprpozo} contain expressions of the type $$N_1(\alphah,\betah)\sim \alphah\Delta\betah, \betah\Delta\alphah,\dots$$ which produce a {\it loss of two derivatives} in the computation of the energy identity. This is well a known problem, and the classical way to overcome this in the equivariant case is to use  the 
 generalized nonlinear  Hasimoto transform, see \cite{chang}, \cite{T2} which, in a suitable nonlinear frame, maps the equivariant flow to a nonlocal cubic nonlinear Schr\"odinger equation. We propose a simpler and more robust approach based on the computation of {\it suitably defined nonlinear Sobolev norms} equivalent to $\mathcal E_4$. The corresponding energy identities 
 {\it do not loose derivatives}. At the level of the original  problem for the Schr\"odinger map $u$, this can be traced to the 
 observation that the quantity $$\frac{d}{dt}\int |u\wedge\Delta\left(u\wedge \Delta u\right)|^2$$ can be controlled in a bootstrap not requiring any assumptions on the derivatives of $u$ higher then the ones appearing in the quantity.  The propagation of this property to the corresponding problem for the radiation $\bW$ requires keeping careful track of the geometric structure of the 
 system written in the Frenet basis, see \fref{equationdefvectorial} and Lemma \ref{gainderivatives}.\\
 
 This paper is organized as follows.\\
 
  In section 2 we introduce the class of 1-equivariant solutions of the Schr\"odinger map problem, define the Frenet basis,
  renormalized variables and the linear Hamiltonian $H$. In section 3 we construct the approximate solution $\bw_0={\bw_0}_{(a,b)}$ and its localized version $\tbw_0$. In section 4 we describe our bootstrap assumptions, set up the orthogonality conditions, define the nonlinear energies 
  $\mathcal E_1$, $\mathcal E_2$, $\mathcal E_4$
  for the remainder radiation term $\bw$ and derive the modulation equations and the mixed energy/Morawetz type identity
  for $\mathcal E_4$. In section 5  we retrieve our bootstrap assumptions for the energies $\mathcal E_2$, $\mathcal E_4$ and the modulation
  parameters $a,b$. In section 6 we conclude the proof of the main result of the paper establishing a finite time blow up and the accompanying asymptotics. In Appendix A we examine coercivity properties of the operators $H$ and $H^2$. In Appendix B we derive various interpolation bounds for the remainder $\bw$ implied by the bounds on the energies
  $\mathcal E_1$, $\mathcal E_2$ and $\mathcal E_4$. These bounds are used extensively throughout the paper in the
  treatment of nonlinear terms. \\
  
 {\bf Aknowledgements} P.R is supported by the Junior ERC-ANR program SWAP. I.R. is partially supported by the NSF grant
 DMS-1001500. This work was completed while P.R. was visited the ETH, Zurich, which he would like to thank for its kind hospitality.


\section{The 1-equivariant flow in the Frenet basis}


This section is devoted to the description of the equivariant flow in the Frenet basis associated to the harmonic map $Q$. We recall in particular the main structure of the linearized operator, and give a vectorial formulation of the flow near $Q$ expressed in coordinates, see \fref{fulleqaito},  which will be important to handle the quasilinear structure of the problem.


\subsection{Ground state and Frenet basis}


Let us provide the geometric and analytic setup for equivariantl solutions $u: {\Bbb R}\times {\Bbb R}^2\to {\Bbb S}^2
\subset {\Bbb R}^3$ of the Schr\"odinger map flow:
\be\label{eq:Schr}
\pa_t u = u\wedge \Delta u.
\ee
Maps with values in ${\Bbb S}^2$ will be treated as maps into ${\Bbb R}^3$ with the image parametrized by the 
Euler angles $(\phi,\theta)$.

The ground state solution of \eqref{eq:Schr}  
\be
\label{explicitformulas}
Q=\left|\begin{array}{lll} \sin(\phi(r))\cos\theta\\  \sin(\phi(r))\sin\theta\\ \cos(\phi(r))\end{array} \right . \ \ \mbox{with} \ \ \phi(r)=2\tan^{-1}(r)
\ee
is a harmonic map ${\Bbb R}^2\to {\Bbb S}^2$ of degree $1$ and satisfies the equation
$$\Delta Q=-|\nabla Q|^2Q.$$
Let us introduce the dilation operator 
$$
\Lambda=r \pa_r
$$
and the functions:
\be
\label{deffunctions}
\sin(\phi(r))=\Lambda \phi(r)=\frac{2r}{1+r^2}, \ \ \cos(\phi(r))=\frac{1-r^2}{1+r^2}=Z(r),
\ee
then
$$Q=\left|\begin{array}{lll} \frac{2r}{1+r^2}\cos\theta\\  \frac{2r}{1+r^2}\sin\theta\\ \frac{1-r^2}{1+r^2}\end{array} \right .=e^{\theta R}\left|\begin{array}{lll} \Lambda \phi\\0\\Z\end{array}\right .
$$
with the rotation generator $R$ given in \fref{defR}. Define $$h=\frac{\sqrt{2}}{|\nabla Q|}=\frac{1+r^2}{2}=\frac{1}{1+Z}$$ and consider the normalized Frenet basis associated to $Q$:
 \be
 \label{defrenet}
 e_r=\frac{\pa_rQ}{|\pa_rQ|}=h\partial_r Q=\left|\begin{array}{lll} \frac{1-r^2}{1+r^2}\cos\theta\\  \frac{1-r^2}{1+r^2}\sin\theta\\\frac{-2r}{1+r^2}\end{array} \right ., \ \  e_{\tau}=\frac{\pa_{\tau}Q}{|\nabla_{\tau}Q|}=h\partial_{\tau} Q=\left|\begin{array}{lll} -\sin\theta\\  \cos\theta\\ 0 \end{array} \right ..
\ee
We compute the action of derivatives and rotations in the moving frame of $Q$:

\begin{lemma}[Derivation and rotation in the Frenet basis]
\label{lemmafrenet}
There holds:\\
(i) {\em Action of derivatives}: 
$$\pa_re_r=-(1+Z) Q, \ \ \Lambda e_r=-\Lambda \phi Q, \ \ \pa_{\tau}e_r=\frac{Z}{r}e_{\tau}, \ \ \Delta e_r=-\frac{1}{r^2}e_r-\frac{2Z(1+Z)}{r}Q,$$
$$\pa_re_{\tau}=0, \ \ \Lambda e_{\tau}=0, \ \ \pa_{\tau}e_{\tau}=-\frac{Z}{r}e_r-(1+Z)Q, \ \ \Delta e_{\tau}=-\frac{1}{r^2}e_{\tau},$$
$$\pa_rQ=(1+Z)e_r, \ \ \Lambda Q=\Lambda \phi e_r, \ \ \pa_{\tau}Q=(1+Z)e_\tau, \ \ \Delta Q=-2(1+Z)^2Q.$$
(ii) {\em Action of the $R$ rotation}:
\be
\label{rotationfrenet}
Re_r=Ze_\tau, \ \ Re_{\tau}=-Ze_r-\Lambda \phi Q,\ \ RQ=\Lambda \phi e_\tau.
\ee
\end{lemma}

Note from the relation 
\be
\label{identityphase}
e^{\Theta R}(u\wedge v)=\left(e^{\Theta R}u\right)\wedge \left(e^{\Theta R}v\right), \  \ \forall \Theta \in \Bbb R,
\ee
that the scaling and rotation symmetries yield the two parameters family of harmonic maps 
 $$Q_{\Theta,\lambda}(r)=e^{\Theta R}Q(\frac r \lambda ), \ \ (\Theta,\lambda)\in \Bbb R\times \Bbb R^*_+, $$ 
with the infinitesimal generators:
 \be
\label{premierzero}
\frac{d}{d\lambda}(Q_{\Theta,\lambda})_{|\lambda=1,\Theta=0}=-\Lambda  Q=-\Lambda \phi e_r,
\ee
\be
\label{secondzero}
\frac{d}{d\Theta}(Q_{\Theta,\lambda})_{|\lambda=1,\Theta=0}=RQ=\Lambda \phi e_{\tau}.
\ee


\subsection{Decomposition of equivariant maps in the Frenet basis}

A degree (index) $1$ equivariant solution $u(t,r,\theta)$ of the Schr\"odinger map problem \eqref{eq:Schr} is defined to 
have the form
$$u(t,r,\theta)=e^{\theta R}\left|\begin{array}{lll} v_1(t,r)\\v_2(t,r)\\v_3(t,r)\end{array} \right . \ \ \mbox{with} \ \ (v_1,v_2,v_3)=v\in \Bbb S^2.$$ 
The 1-equivariant symmetry is preserved by the Schr\"odinger map flow. Any equivariant map has uniquely defined radially symmetric coordinates in the Frenet basis. Indeed, let $$u(t,r,\theta)=e^{\theta R}v(t,r)$$ be an equivariant map and define  a decomposition of $u$ relative to the Frenet basis \fref{defrenet}: $$u=\alpha (t,r)e_r+\beta(t,r) e_{\tau}+\gamma(t,r) Q=e^{\theta R}\left|\begin{array}{ll} \alpha Z+\gamma \Lambda \phi\\ \beta\\-\alpha \Lambda \phi+\gamma Z
\end{array} \right ..$$ This leads to the relation 
$$\left(\begin{array}{lll} Z &0 & \Lambda \phi\\ 0 & 1 & 0\\ -\Lambda \phi& 0 & Z\end{array}\right)\left|\begin{array}{lll}\alpha\\\beta\\\gamma\end{array}\right .=\left|\begin{array}{lll} v_1\\v_2\\v_3\end{array}\right .,$$ which uniquely determines 
$(\alpha,\beta,\gamma): \alpha^2+\beta^2+\gamma^2=1$ from $(v_1,v_2,v_3)$.


\subsection{Linear Hamiltonian $H$ and its factorization}


Recall the definition $$ Z(y)=\frac{1-y^2}{1+y^2}.$$
The function $Z$ can be regarded as a solution of the equation
\be
\label{formjlaz}
\Lambda Z=Z^2-1.
\ee
We introduce the following Schr\"odinger operator $H$ which will appear as the leading order term in the linearization 
of the Schr\"odinger map  equation, written relative to the Frenet basis, around $Q$:
\be
\label{defhmialotn}
H=-\Delta +\frac{V}{y^2}, \qquad y\in [0,\infty),
\ee 
with the potential $$V(y)=Z^2+\Lambda Z=2\Lambda Z+1=\frac{y^4-6y^2+1}{(1+y^2)^2}.$$ 
The Hamiltonian $H$ admits the factorization 
\be
\label{factojfh}
H=A^*A, \ \ A=-\pa_y+\frac{Z}{y}, \ \ A^*=\pa_y+\frac{1+Z}{y},
\ee  
where the conjugation is defined with respect to the 2-dimensional measure $y dy$.
The conjuguate operator is given explicitely by: 
\be
\label{defhtilde}
\tilde{H}=AA^*=-\Delta +\frac{2(1+Z)}{y^2}=-\Delta+\frac{4}{y^2(1+y^2)}.
\ee


\subsection{Schr\"odinger map flow in the renormalized Frenet basis}

We introduce two time-dependent modulation parameters $\Theta(t)$ and $\lambda(t)$ responsible for the 
changes of the phase and scale, respectively, and the associated transformation $\S$:
$$
(\S v)(t,r)=e^{\Theta(t)R}v(t,\frac r\lambda).
$$ 
We define a space-time scaling transformation
$$
 \frac{ds}{dt}=\frac{1}{\lambda^2(t)}, \ \ y=\frac{r}{\lambda(t)}
 $$ 
 relative to which 
 $$
v_{\lambda}(t,r):= v(t,\frac r\lambda)=v(s,y).
 $$
 We compute $$\pa_tv=\frac{\S}{\lambda^2}\left[\pa_sv+\Theta_s R v-\lsl\Lambda v\right]$$ and thus in particular:
\be
\label{erdt}
\pa_t(\S e_r)=\frac{\S}{\l^2}\left\{\Theta_s Z e_{\tau}+\lsl\Lambda\phi Q\right\},
\ee
\be
\label{etaudt}
\pa_t(\S e_{\tau})=\frac{\S}{\l^2}\left\{\Theta_s (-Ze_r-\Lambda\phi Q)\right\}\\
\ee
\be
\label{etuq}
\pa_t(\S Q)=\frac{\S}{\l^2}\left\{\Theta_s\Lambda \phi e_\tau-\lsl\Lambda\phi e_r\right\}.
\ee
Let $u$ now be an equivariant Schr\"odinger map. We decompose $u$ according 
to 
\be
\label{defv}
u=\S(Q+v)
\ee
and expand $\S v$ relative to the Frenet basis associated with $\S Q$:
$$
\ \ \S v=\alphah_\l \S e_r+\betah_\l\S e_\tau+\gammah_\l \S Q.$$ 
The coefficients 
$$
(\alphah(t,r),\betah(t,r),\gammah(t,r))=(\hat\alpha(s,y),\hat\beta(s,y),\hat\gamma(s,y))
$$
obey the constraint
\be \label{eqgammah}
\hat\alpha^2+\hat\beta^2+(1+\hat\gamma)^2=1.
\ee
The map $\S v$ satisfies the equation
\bea
\label{eqssv}
\pa_t (\S v) & = & \S Q\wedge\Delta (\S v)+\S v\wedge\Delta(\S Q)-\pa_t(\S Q)+\S v\wedge\Delta(\S v)\\
\nonumber & = & \frac{1}{\l^2}\S\left\{Q\wedge(\Delta  v+|\nabla Q|^2v)+ v\wedge\Delta v+\lsl \Lambda \phi e_r-\Theta_s\Lambda \phi e_\tau\right\}.
\eea
We now re-express quantities involving $\S v$ in terms of the coefficients $(\hat\alpha_\lambda,\hat\beta_\lambda,\hat\gamma_\lambda)$:
\bee 
\pa_t (\mathcal S v) & =& \left\{\pa_t\alphah_\l-\frac{\betah_\l}{\l^2}\left(\Theta_sZ\right)_{\l}+\frac{\gammah_\l}{\l^2}\left(-\lsl\Lambda \phi\right)_\l\right\}\S e_r\\
& + & \left\{\pa_t\betah_\l+\frac{\alphah_\l}{\l^2}\left(\Theta_sZ\right)_{\l}+\frac{\gammah_\l}{\l^2}\left(\Theta_s\Lambda \phi\right)_\l\right\}\S e_\tau\\
& + & \left\{\pa_t\gammah_\l+\frac{\alphah}{\lambda^2}\left(\lsl\Lambda \phi\right)_\l-\frac{\betah}{\l^2}\left(\Theta_s \Lambda \phi\right)_\l\right\}\S Q.
\eee
We now use the algebra: $$\frac{1}{y^2}-2(1+Z)^2=\frac{y^4-6y^2+1}{y^2(1+y^2)^2}=\frac{V(y)}{y^2}$$ to compute:
\bea
\label{ecomptuaiowquat}
 \Delta v+|\nabla Q|^2v & = & \left(-H\alphah+2(1+Z)\pa_y\gammah\right)e_r+(-H\betah)\et\\
\nonumber& +& \left(-\frac{2Z(1+Z)}{y}\alphah+\Delta \gammah-2(1+Z)\pa_y\alphah\right) Q
\eea where $H$ is the Hamiltonian \fref{defhmialotn}. This gives for the linear term in $v$:
$$ Q\wedge\left(\Delta v+|\nabla Q|^2v\right)= (H\betah)e_r+\left(-H\alphah+2(1+Z)\pa_y\gammah\right)e_{\tau}.$$ 
The nonlinear term produces the expression:
\bee
& & v\wedge\Delta v=v\wedge\left(\Delta v+|\nabla Q|^2 v\right)\\
& = & \left[-\frac{2Z(1+Z)}{r}\alphah\betah+\beta(\Delta \gammah-2(1+Z)\pa_y\alphah)+\gammah H\betah\right]e_r\\
& + & \left[\frac{2Z(1+Z)}{r}\alphah^2-\alphah(\Delta \gammah-2(1+Z)\pa_y\alphah)+\gammah(-H\alphah+2(1+Z)\pa_y\gammah)\right]e_\tau\\
& + & \left[-\alphah H\betah+\betah(H\alphah-2(1+Z)\pa_r\gammah)\right] Q.
\eee
We now project \fref{eqssv} onto $\mbox{Span}\{\S e_r,\S e_\tau\}$ and obtain the equivalent set of renormalized equations:
\begin{align}
\label{eqalpha}
\pa_s\alphah-\lsl\Lambda \alphah  &= H\betah-\frac{2Z(1+Z)}{y}\alphah\betah+\betah(\Delta \gammah-2(1+Z)\pa_y\alphah)+\gammah H\betah\\
 & +  \lsl(1+\gammah)\Lambda \phi+\Theta_s\betah Z,\nonumber \\
\pa_s\betah -\lsl\Lambda \betah& =  -H\alphah + \frac{2Z(1+Z)}{y}\alphah^2+2(1+Z)\pa_y\gammah -\alphah(\Delta \gammah-2(1+Z)\pa_y\alphah)\nonumber\\
&+\gammah(-H\alphah+2(1+Z)\pa_y\gammah)-  \Theta_s(1+\gammah)(\Lambda \phi)-\Theta_s \alphah Z\label{eqbetah}\\
\label{eqgammahbis}
\pa_s\gammah-\lsl\Lambda \gammah & =  -\alphah H\betah+\betah(H\alphah-2(1+Z)\pa_r\gammah)
 -  \alphah \lsl\Lambda \phi+\betah\Theta_s\Lambda \phi.
\end{align}


\subsection{Vectorial formulation}

An essential feature of our analysis is to keep track of the geometric structure of \fref{nlsmap}. On the other hand,
an examination of the equations \eqref{eqalpha}, \eqref{eqbetah} and \eqref{eqgammahbis}, associated with the 
Frenet basis of $\S Q$, reveals the structure of 
a quasilinear Schr\"odinger system. To overcome a potential loss of derivatives in the analysis of the system,
we rewrite the linear system in a vectorial form. Let $\widehat{\bw}$ be the vector of coordinates in the Frenet basis:
\be
\label{eqcoridntesglobale}
\widehat{\bw}=\left|\begin{array}{lll}\alphah\\\betah\\\gammah\end{array}\right .
\ee
and $e_z$ denote 
$$
\ \ e_z=\left|\begin{array}{lll}0\\0\\1\end{array}\right ..
$$ 
The system of equations for $\hat\alpha,\hat\beta,\hat\gamma$ can be then equivalently expressed in the form 
\be
\label{fulleqaito}
\pa_s\widehat{\bw}-\lsl\Lambda \widehat{\bw}+\Theta_sZR\widehat{\bw}=-\hJ \left[\Bbb H\hbw +\Lambda \phi\left|\begin{array}{lll}\Theta_s\\\lsl\\0\end{array}\right .\right]
\ee with
\be
\label{defjhat}
\hJ=(e_z+\widehat{\bw})\wedge, \ \ |e_z+\hbw|^2=1,
\ee
\be
\label{vectorialhamiltonian}
{\Bbb H} \widehat{\bw}=\left|\begin{array}{lll}H\alphah\\H\betah\\-\Delta \gammah\end{array}\right .+2(1+Z)\left|\begin{array}{lll}-\pa_y\gammah\\0\\\pa_y\alphah+\frac{Z}{y}\alphah\end{array}\right..
\ee
 A direct computation shows that the vectorial Hamiltonian $\H$ is (formally) self-adjoint: 
 $$\Bbb H=\Bbb H^*.$$
 We also define the vectorial operators $\A$ and $\A^*$:
 \begin{align}
 \label{eq:defA}
 &{\A} \widehat{\bw}=\left|\begin{array}{lll} A\alphah\\ A\betah\\0\end{array}\right .,\\
  &{\A^*} \widehat{\bw}=\left|\begin{array}{lll} A^*\alphah\\ A^*\betah\\0\end{array}\right .
 \label{eq:defA*}
 \end{align}


\section{Construction of the approximate profile}


The aim of this section is to construct an approximate solution with a {\it controllable growth} in $y$ variable to the modulated nonlinear equation \fref{fulleqaito}. This construction depends on an assumed dynamics of the modulation parameters
$(\Theta(s),\lambda(s))$. In turn, the above dynamics depends on the interaction between the approximate profile, 
with the first term given by $\S Q$, and the remaining radiation part of the solution. We complement $(\Theta,\lambda)$ 
by two additional modulation parameters $a$ and $b$ and assume that  
to a leading order: 
\be
\label{lawapproximate}
-\lsl=b, \ \ -\Theta_s=a, \ \ b_s=-(b^2+a^2), \ \ a_s=0.
\ee
This choice will be justified in our final analysis.

\begin{proposition}[Construction of the approximate profile]
\label{propprofile}
Let $M>0$ be a large universal constant, then there exists a small enough universal constant $b^*=b^*(M)>0$ such that the following holds. Let $0<b<b^*$, $|a|\leq \frac{b}{|\log b|}$ and $B_0,B_1$ be given by \fref{defbzerobone}, then there exist profiles $T_{i,j}$ such that  
\be
\label{defwzero}
\bw_0=\left|\begin{array}{lll}\alpha_0\\\beta_0\\\gamma_0\end{array}\right .
\ee
with
\be
\label{defeab}
\alpha_0=aT_{1,0}+b^2T_{0,2}+a^2T_{2,0}, \ \ \beta_0=bT_{0,1}+abT_{1,1}+b^3T_{0,3},
\ee
\be
\label{defgammaab}
\gamma_{0}=b^2S_{0,2} \ \ \mbox{with} \ \  S_{0,2}=-\frac 12T_{0,1}^2
\ee
is an approximate solution of \eqref{fulleqaito} (in the regime \fref{lawapproximate}) in the sense that the error
\bea
\label{defpisovectorial}
\bPsi_0& = &\left|\begin{array}{lll}\Psi_0^{(1)}\\\Psi_0^{(2)}\\\Psi_0^{(3)}\end{array}\right.=-b\Lambda\bw_0+ZR\bw_0-(e_z+\bw_0) \wedge\left[\Bbb H\bw_0 -\Lambda \phi\left|\begin{array}{lll}a\\ b\\0\end{array}\right.\right]\\
\nonumber & + & \left|\begin{array}{lll}2b(b^2+a^2)T_{0,2}\\ (b^2+a^2)T_{0,1}+a(b^2+a^2)T_{1,1}\\ 2b(b^2+a^2)S_{0,2}\end{array}\right .
\eea
satisfies the bounds:\\
(i) Weighted norm estimates: 
\be
\label{estderivatoveweight}
\int_{y\leq 2B_1} \frac{|\pa_y^{i}\Psi_0^{(1)}|^2}{y^{6-2i}}+\int_{y\leq 2B_1}  \frac{|\pa_y^{i}\Psi_0^{(2)}|^2}{y^{6-2i}}\lesssim \frac{b^4}{|\log b|^2},\ \ 0\leq i\leq 3,
\ee
\be
\label{estderivatoveweightthree}
\int_{y\leq 2B_1} \frac{|\pa_y^{i}\Psi_0^{(3)}|^2}{y^{8-2i}}\lesssim \frac{b^6}{|\log b|^2},\ \ 0\leq i\leq 4,
\ee
(i') Localized estimates:
\be
\label{estderivloc}
\int_{B_1\leq y\leq 2B_1} {|\pa_y^{i}\Psi_0^{(1)}|^2}+\int_{B_1\leq y\leq 2B_1}  {|\pa_y^{i}\Psi_0^{(2)}|^2}
\lesssim \frac{b^4}{|\log b|^2} (bB_1^{4-i})^2,\ \ 0\leq i\leq 4,
\ee
\be
\label{estderivloc3}
\int_{B_1\leq y\leq 2B_1} {|\pa_y^{i}\Psi_0^{(3)}|^2}\lesssim \frac{b^4}{|\log b|^2} (bB_1^{3-i}\log^2 B_1)^2,\ \ 0\leq i\leq 4.
\ee 
(ii) $H^2$, $H^3$ estimates:
\be
\label{htwolossy}
\int_{y\leq 2B_1}|H\Psi_0^{(1)}|^2+|H\Psi_0^{(2)}|^2\lesssim b^4|\log b|^2,
\ee
\be
\label{estderivatoveweighithH}
\int_{y\leq 2B_1} \frac{|\pa_y^{i}H\Psi_0^{(1)}|^2}{y^{2-2i}}+\int_{y\leq 2B_1}  \frac{|\pa_y^{i}H\Psi_0^{(2)}|^2}{y^{2-2i}}\lesssim \frac{b^4}{|\log b|^2},\ \ 0\leq i\leq 1.
\ee
(iii) Sharp $H^4$ estimate:
 \be
\label{estimatecrucial}
\int_{y\leq 2B_1}|H^2\Psi^{(i)}_{0}|^2\lesssim \frac{b^6}{|\log b|^2}, \ \ i=1,2.
\ee
(iv) {\em Flux computation}: Let $\Phi_M$ be given by 
$$
\Phi_M=\chi_M\Lambda\phi-c_MH(\chi_M\Lambda\phi)
$$
with $$c_M=\frac{(\chi_M\Lambda\phi,T_{1,0})}{(H(\chi_M\Lambda\phi),T_{1,0})}=c_{\chi}\frac{M^2}{4}(1+o_{M\to +\infty}(1)),$$
see \fref{defphim}, then:
\be
\label{fluxcomputationone}
\frac{(H(\Psi_0^{(2)}),\Phi_M)}{(\Lambda\phi, \Phi_M)}=-\frac{2b^2}{|\log b|}+O\left(\frac{b^2}{|\log b|^2}\right),
\ee
\be
\label{fluwnumbertwo}
\frac{(H(\Psi_0^{(1)}),\Phi_M)}{(\Lambda\phi, \Phi_M)}=-\frac{2ab}{|\log b|}+O\left(\frac{b^2}{|\log b|^2}\right).
\ee
\end{proposition}

\begin{remark} The flux computations \fref{fluxcomputationone}, \fref{fluwnumbertwo} will be central in the derivation of the correction to the approximate system of modulation equations \fref{lawapproximate}. The improved behavior, relative to the
powers of $b$, in \fref{estderivatoveweightthree}, \fref{estimatecrucial} will be fundamental for the $H^4$ bounds on the remaining radiative part of the solution.
\end{remark}
\begin{remark}
We also record the important behavior of the constructed profiles $T_{i,j}$ for small and large values of $y$: for $y\leq 1$,
$$
T_{0,1}=T_{1,0} = O(y^3),\quad T_{0,2}=T_{2,0}=T_{0,3}=O(y^5),
$$
and for $y\ge 1$:
\be
\label{roughboundone}
T_{0,1}=T_{1,0} = O(y\log y),\quad T_{0,2}=O\left(\frac {y}{b|\log b|}\right), \quad T_{2,0}= O(y|\log y|^3),
\ee
\be
\label{roughboundonebis}
\quad T_{0,3} = O\left(\frac {y^3}{b|\log b|}\right).
\ee
The above relations also hold for the derivatives with the usual convention that each derivative reduces one power of $y$. 
\end{remark}

{\bf Proof of Proposition \ref{propprofile}}\\

We compute the coordinate components of the error $\bPsi_0$ from \fref{defeab}, \fref{defgammaab}, \fref{defpisovectorial}, 
and obtain after collecting the terms with the like powers of $a,b$:
\bea
\label{computationpsiber}
&& \Psi_0^{(1)}  =  b\left\{HT_{0,1}-\Lambda \phi\right\}\\
\nonumber & + & b^3\left\{HT_{0,3}-\frac{2Z(1+Z)}{y}T_{0,2}T_{0,1}+T_{0,1}\Delta S_{0,2}-2(1+Z)T_{0,1}\pa_yT_{0,2}+S_{0,2} HT_{0,1}\right .\\
\nonumber & - & \left .\Lambda T_{0,2}-S_{0,2}\Lambda\phi+2T_{0,2}\right\}\\ 
\nonumber & + & ab\left\{HT_{1,1}-\frac{2Z(1+Z)}{y}T_{1,0}T_{0,1}-2(1+Z)T_{0,1}\pa_yT_{0,1}-\Lambda T_{0,1}-ZT_{0,1}\right\}\\
\nonumber & + & ab^3\left\{T_{1,1}\Delta S_{0,2}-2(1+Z)T_{1,1}\pa_yT_{0,2}-ZT_{0,3}-2(1+Z)T_{0,3}\pa_yT_{1,0}+S_{0,2}HT_{1,1}\right .\\
\nonumber &  -& \left . \frac{2Z(1+Z)}{y}(T_{1,0}T_{0,3}+T_{0,2}T_{1,1})\right\}\\
\nonumber & + & a^2b\left\{-\frac{2Z(1+Z)}{y}T_{1,0}T_{1,1}-2(1+Z)T_{1,1}\pa_yT_{1,0}-ZT_{1,1}\right.\\
\nonumber & - & \left .\Lambda T_{2,0}-\frac{2Z(1+Z)}{y}T_{2,0}T_{0,1}-2(1+Z)T_{0,1}\pa_yT_{2,0}+2T_{0,2}\right\}\\
\nonumber & + & a^3b\left\{-\frac{2Z(1+Z)}{y}T_{2,0}T_{1,1}-2(1+Z)T_{1,1}\pa_yT_{2,0}\right\}\\
\nonumber & + & b^5\left\{-\frac{2Z(1+Z)}{y}T_{0,2}T_{0,3}+T_{0,3}(\Delta S_{0,2}-2(1+Z)\pa_yT_{0,2})+S_{0,2}HT_{0,3}\right\}\\
\nonumber & + & a^2b^3\left\{-\frac{2Z(1+Z)}{y}T_{2,0}T_{3,0}-2(1+Z)T_{0,3}\pa_yT_{2,0}\right\}
\eea
\bea
\label{computationpsibetau}
& &  \Psi_0^{(2)} =  a\left\{-HT_{1,0}+\Lambda \phi\right\}\\ 
 \nonumber & + &  b^2\left\{-HT_{0,2}+2(1+Z)\pa_yS_{0,2}-\Lambda T_{0,1}+T_{0,1}\right\}\\
\nonumber & + &b^4\left\{-T_{0,2}\Delta S_{0,2}+\frac{2Z(1+Z)}{y}T_{0,2}^2+2(1+Z)T_{0,2}\pa_yT_{0,2}-S_{0,2}HT_{0,2}\right .\\
\nonumber & + & \left .2(1+Z)S_{0,2}\pa_yS_{0,2}-\Lambda T_{0,3}\right\}\\
 \nonumber & + &  a^2\left\{-HT_{2,0}+\frac{2Z(1+Z)}{y}T_{1,0}^2+2(1+Z)T_{1,0}\pa_yT_{1,0}+ZT_{1,0}+T_{0,1}\right\}\\
 \nonumber& + & ab^2\left\{\frac{4Z(1+Z)}{y}T_{1,0}T_{0,2}-T_{1,0}\Delta S_{0,2}+2(1+Z)(T_{1,0}\pa_yT_{0,2}+T_{0,2}\pa_yT_{1,0}+T_{1,1})\right.\\
\nonumber & - & \left.S_{0,2}HT_{1,0}-\Lambda T_{1,1}+ZT_{0,2}+S_{0,2}\Lambda \phi\right\}\\
 \nonumber & + & a^3\left\{\frac{4Z(1+Z)}{y}T_{2,0}T_{1,0}+2(1+Z)(\pa_yT_{2,0}T_{1,0}+T_{2,0}\pa_yT_{1,0})+ZT_{2,0}+
 T_{1,1}\right\}\\
 \nonumber & + & a^2b^2\left\{-S_{0,2}HT_{2,0}+\frac{4Z(1+Z)}{y}T_{0,2}T_{2,0}-T_{2,0}\Delta S_{0,2}+2(1+Z)(\pa_yT_{0,2}T_{2,0}+T_{0,2}\pa_yT_{2,0})\right\}\\
 \nonumber & + & a^4\left\{\frac{2Z(1+Z)}{y}T_{2,0}^2+2(1+Z)T_{2,0}\pa_y T_{2,0}\right\},
 \eea 
 \bea
 \label{computationpsibeq}
&& \Psi_0^{(3)}= ab\left\{-T_{1,0}(HT_{0,1}-\Lambda \phi)+T_{0,1}(HT_{1,0}-\Lambda \phi)\right\}\\
\nonumber & + & b^3\left\{-T_{0,2}(HT_{0,1}-\Lambda \phi)+T_{0,1}HT_{0,2}-2(1+Z)T_{0,1}\pa_yS_{0,2}-\Lambda S_{0,2}+2S_{0,2}\right\}\\
\nonumber & + & a^2b\left\{-T_{1,0}HT_{1,1}+T_{1,1}(HT_{1,0}-\Lambda \phi)-T_{2,0}HT_{0,1}+T_{0,1}HT_{2,0}+2S_{0,2}\right\}\\
\nonumber & + & ab^3\left\{-T_{1,0}HT_{0,3}+T_{0,3}(HT_{1,0}-\Lambda \phi)-2(1+Z)T_{1,1}\pa_yS_{0,2}\right\}\\
\nonumber & + & a^3b\left\{-T_{2,0}HT_{1,1}+T_{1,1}HT_{2,0}\right\}\\
\nonumber & + & b^5\left\{-T_{0,2}HT_{0,3}+T_{0,3}HT_{0,2}-2(1+Z)T_{0,3}\pa_yS_{0,2}\right\}\\
& + & a^2b^3\left\{-T_{2,0}HT_{0,3}+T_{0,3}HT_{2,0}\right\}
\eea

 {\bf Step 1} Construction of $T_{1,0},T_{0,1}$.\\
 
We start by computing the Green's functions of the Hamiltonian $H$. We have $H(\Lambda \phi)=0$ and $$H(\Gamma)=0 \ \ \mbox{for} \ \ y>0$$ with: 
 $$ \Gamma(y)=\Lambda \phi\int_1^y\frac{dx}{x(\Lambda\phi(x))^2}=\frac{y}{2(1+y^2)}\int_1^y\frac{1+2x^2+x^4}{x^3}dx
 $$
 which yields
 \be
 \label{Gamma}
\Gamma(y)=\left\{\begin{array}{ll} O(\frac1y) \ \ \mbox{as} \ \ y\to 0,\\ \frac{y}{4}+O\left(\frac{\log y}{y}\right)\ \  \mbox{as} \ \ y\to +\infty.\end{array}\right . 
\ee
The functions $\Lambda\phi$ and $\Gamma$ allow us to find a regular solution $f$ of an inhomogeneous equation
$$
H f =g 
$$
in the form 
$$
f(y)=\Lambda\phi(y)\int_0^y g(x) \Gamma(x) x dx  -\Gamma(y) \int_0^y g(x) \Lambda\phi(x) x dx 
$$
modulo a multiple of the (regular) element of the kernel of $H$ -- the function $\Lambda\phi$.\\
We let  $T_{1,0}=T_{0,1}=T_1(y)$ be the solution to 
\be
\label{deftone}
HT_{1}=\Lambda \phi
\ee given by:
 \bea
\label{T1zero}
T_{1}(y) &=&\Lambda\phi(y)\int_0^y\Gamma(x)\Lambda\phi(x)xdx-\Gamma(y)\int_0^y(\Lambda\phi)^2xdx\\
& = &-\frac{(1-y^4)\log(1+y^2) + 2y^4 - y^2 - 4y^2 \int_0^y \frac {\log(1+s^2)}{s} ds}{2y(1+y^2)}.
\eea
 We easily see that for $y\to +\infty$
\be
\label{behaviortun}
T_1(y) = -y \log y+ y + O\left(\frac{(\log y)^2}{y}\right), \ \ \Lambda T_1(y) = -y \log y + O\left(\frac{(\log y)^2}{y}\right),
\ee
and at the origin 
\be
\label{annulationorigin}
T_1(y)=\frac12y^3+O(y^5).
\ee
Rescaling the equation for $T_1$ we also observe that
\be
\label{estthtildeun}
H(\Lambda T_1)=2\Lambda \phi+\Lambda^2\phi-\frac{\Lambda V}{y^2}T_1.
\ee 

{\bf Step 2} Construction of the radiation $\Sigma_b$.\\

We define  $\Sigma_b$ to be the solution of the equation 
\be
\label{defsigmab}
H\Sigma_b=c_b\chi_{\frac{B_0}{4}}\Lambda \phi+d_bH\left[(1-\chi_{3B_0})\Lambda\phi)\right],
\ee
with  
\be
\label{defcbdb}
 c_b=\frac{4}{\int \chi_{\frac{B_0}{4}}(\Lambda \phi(x))^2xdx}, \ \ d_b=c_b\int_0^{B_0}\chi_{\frac{B_0}{4}}\Lambda \phi(x)\G(x)xdx.
\ee
given by 
\bee
\Sigma_b(y) & = & -\G(y)\int_0^yc_b\chi_{\frac{B_0}{4}}(\Lambda \phi(x))^2xdx+\Lambda \phi(y)\int_{0}^yc_{b}\chi_{\frac{B_0}{4}}\Lambda \phi(x)\G(x)xdx\\
& - & d_b(1-\chi_{3B_0})\Lambda\phi(y).
\eee
Observe that
\be
\label{estimationimportante}
\Sigma_b=c_bT_1\ \ \mbox{for} \ \ y\leq \frac{B_0}{8},
\ee
\be
\label{eshiiloin}
\Sigma_b=-4\Gamma(y)=-y+O\left(\frac{\log y}{y}\right)\ \ \mbox{for}\ \ y\geq 6B_0,
\ee
and from \fref{defsigmab} and the definition of $B_0=\frac{1}{\sqrt{b}}$:
\be
\label{asymptoticcb}
c_b=\frac{2}{|\log b|}\left(1+O(\frac{1}{|\log b|})\right), \ \ d_b=O\left(\frac{1}{b|\log b|}\right).
\ee
We now estimate $\Sigma_b$ using \fref{Gamma}: for $6B_0\leq y\leq 2B_1$, 
\bea
\label{aymptotiesigmabun}
 \Sigma_b(y)  =- y+O\left(\frac{\log y}{y}\right),
\eea
and for $y\leq 6B_0$:
\bea
\label{aymptotiesigmabdeux}
\nonumber \Sigma_b(y) & = &-c_b\left(\frac{y}{4}+O\left(\frac{\log y}{y}\right)\right)\left[\int_0^y \chi_{\frac{B_0}{4}}(\Lambda \phi(x))^2xdx\right]+c_b\Lambda\phi(y)\int_1^{y}O(x)dx\\
\nonumber & + & O\left(\frac{1}{by|\log b|}{\bf 1}_{B_0\leq y\leq 6B_0}\right)\\
&= &  -y\frac{\int_0^y\chi_{\frac{B_0}{4}}(\Lambda\phi(x))^2xdx}{\int \chi_{\frac{B_0}{4}}(\Lambda\phi(x))^2xdx}+O\left(\frac{1+y}{|\log b|}\right).
\eea
Far out $\Sigma_b$ is the leading order radiation term. It is large and dominant near $B_0$, but small, with a $\frac{1}{|\log b|}$ gain, on compact sets $y\leq C$.\\

 {\bf Step 3} Construction of $T_{0,2}$.\\

Define
\be
\label{defsigmazerotwo}
\Sigma_{0,2}=2(1+Z)\pa_yS_{0,2}-\Lambda T_{0,1}+T_{0,1}+\Sigma_b,
\ee
then from \fref{T1zero}, \fref{aymptotiesigmabun}, \fref{aymptotiesigmabdeux}: for $ y\geq 6B_0$,
\be
\label{boundsigmatwoone}
 \Sigma_{0,2}(y)  =  O\left(\frac{(\log y)^2}{y}\right),
\ee
and for $1\leq y\leq 6B_0$:
\bea
\label{boundsigmatwotwo}
\nonumber \Sigma_{0,2}(y) & = & - y\left(\frac{\int_0^y\chi_{\frac{B_0}{4}}(\Lambda\phi(x))^2xdx}{\int \chi_{\frac{B_0}{4}}(\Lambda\phi(x))^2xdx}-1\right)+O\left(\frac{1+y}{|\log b|}\right)+O\left(\frac{(\log y)^2}{y}\right)\\
& = & O\left(\frac{1+y}{|\log b|}(1+|\log (y\sqrt{b})|)\right).
\eea
Note that the improved behavior of $\Sigma_{0,2}(y)$ for large values of $y$ in 
\fref{boundsigmatwoone}, \fref{boundsigmatwotwo} relies on the cancellation between the terms $\Sigma_b$, $\Lambda T_{0,1}$
and $T_{0,1}$ and thus depends on the introduction
of the radiation $\Sigma_b$ and the presence of the term $b^2T_1$ on the RHS of \fref{computationpsibetau},  which appeared 
there as a consequence of the approximate modulation dynamics equation $b_s=-b^2$ in \fref{fulleqaito}. 
Higher order derivatives are estimated similarily. We now let $T_{0,2}$ be the solution to 
\be
\label{eqtzerotwo}
HT_{0,2}=\Sigma_{0,2}
\ee given by 
\be
\label{formulatzerotwo}
 T_{0,2}(y)=\Lambda \phi(y)\int_{0}^y\Sigma_{0,2}\G(x)xdx-\G(y)\int_0^y\Sigma_{0,2}(x)\Lambda \phi(x)xdx.
\ee
We derive from \fref{boundsigmatwoone}, \fref{boundsigmatwotwo} the bound: 
\be
\label{decayttwo}
\forall y\leq 2B_1, \ \ |T_{0,2}(y)|\lesssim \frac{(1+y)}{b|\log b|}.
\ee
Note that we also have the crude bound:
\be
\label{otherobundttwo}
\forall y\leq 2B_1, \ \ |T_{0,2}(y)|\lesssim (1+y^3)
\ee
and the high order vanishing near the origin:
$$|T_{0,2}(y)|\lesssim y^5 \ \ \mbox{for} \ \ y\leq 1.$$
Again, rescaling the equation for $T_{0,2}$, observe that
 \be
 \label{equationttwotilde}
 H(\Lambda T_{0,2})=2\Sigma_{0,2}+\Lambda \Sigma_{0,2}-\frac{\Lambda V}{y^2}T_{0,2}.
 \ee
 
 {\bf Step 4} Construction of $T_{1,1}$.\\
 
 Define
 \be
 \label{defsigmaoneone}
 \Sigma_{1,1}=\frac{2Z(1+Z)}{r}T_1^2+2(1+Z)T_{0,1}\pa_yT_{0,1}+\Lambda T_{0,1}+ZT_{0,1}-\Sigma_b
 \ee
 then from \fref{T1zero}, \fref{aymptotiesigmabun}, \fref{aymptotiesigmabdeux}: for $6B_0\leq y\leq 2B_1$,
\be
\label{boundtoneone}
\nonumber \Sigma_{1,1}(y)  =  O\left(\frac{(\log y)^2}{y}\right)
\ee
and for $1\leq y\leq 6B_0$:
\bea
\label{boundtoneonetwo}
\nonumber \Sigma_{1,1}(y) & = & y\left(\frac{\int_0^y\chi_{\frac{B_0}{4}-1}(\Lambda \phi(x))^2xdx}{\int \chi_{\frac{B_0}{4}}(\Lambda \phi(x))^2xdx}\right)+O\left(\frac{1+y}{|\log b|}\right)+O\left(\frac{(\log y)^2}{y}\right)\\
& \lesssim &  \frac{1+y}{|\log b|}(1+|\log (y\sqrt{b})|).
\eea
We now let $T_{1,1}$ be the solution to 
\be
\label{eqtonone}
HT_{1,1}=\Sigma_{1,1}
\ee 
given by 
$$
 T_{1,1}(y)=\Lambda \phi(y)\int_{0}^y\Sigma_{1,1}\G(x)xdx-\G(y)\int_0^y\Sigma_{1,1}(x)\Lambda \phi(x)xdx.
$$
We derive from \fref{boundtoneone}, \fref{boundtoneonetwo} the bound: $\forall y\leq 2B_1$,
\be
\label{decayttwotoneone}
\ \ |T_{1,1}(y)|\lesssim  \frac{1+y}{b|\log b|},
\ee 
the crude bound
\be
\label{lossytoneone}
|T_{1,1}(y)|\lesssim (1+y^3),
\ee
and the high order vanishing near the origin $$|T_{1,1}(y)|\lesssim y^5\ \ \mbox{for} \ \ y\leq 1.$$
Observe, by construction and rescaling, that
 \be
 \label{equationtuntilde}
 HT_{1,1}=\Sigma_{1,1}, \ \ H(\Lambda T_{1,1})=2\Sigma_{1,1}+\Lambda \Sigma_{1,1}-\frac{\Lambda V}{y^2}T_{1,1}.
 \ee
 
 {\bf step 5} Construction of $T_{2,0}$.\\
 
 Define 
 $$\Sigma_{2,0}=\frac{2Z(1+Z)}{y}T_{1,0}^2+2(1+Z)T_{1,0}\pa_yT_{1,0}+(1+Z)T_{1,0}=O\left(\frac{|\log y|^2}{1+y}\right)$$ and $T_{2,0}$ be the solution to $$HT_{2,0}=\Sigma_{2,0}$$ given by 
 $$ T_{2,0}(y)=\Lambda \phi(y)\int_{0}^y\Sigma_{2,0}\G(x)xdx-\G(y)\int_0^y\Sigma_{2,0}(x)\Lambda \phi(x)xdx,$$ then $$|T_{2,0}(y)|\lesssim y^5\ \ \mbox{for}\ \ y\leq 1$$ and $$|T_{2,0}(y)|\lesssim (1+y)|\log y|^3.$$

 {\bf Step 6} Construction of $T_{0,3}$.\\
 
 Define
 \bea
 \label{defsigmazerothree}
 \Sigma_{0,3} &=&  2\frac{Z(1+Z)}{y}T_{0,2}T_{0,1}-T_{0,1}\Delta S_{0,2}+2(1+Z)T_{0,1}\pa_yT_{0,2}\\
\nonumber  & + &\Lambda T_{0,2}-2T_{0,2}.
 \eea
We have from \fref{decayttwo}:
  \be
  \label{estsigmazerothree}
  \forall y\leq 2B_1, \ \ |\Sigma_{0,3}(y)|\lesssim \frac{1+y}{b|\log b|},
  \ee
   the crude bound:
  \be
  \label{otehrboudttheree}
   \forall y\leq 2B_1, \ \ |\Sigma_{0,3}(y)|\lesssim (1+y^3),
   \ee
   and the vanishing at the origin: 
   $$|\Sigma_{0,3}(y)|\lesssim y^5\ \ \mbox{for} \ \ y\leq 1.$$
 We then let $T_{0,3}$ be the solution to 
 \be
 \label{eqtzerothree}
 HT_{0,3}=\Sigma_{0,3}
 \ee given by 
 $$
 T_{0,3}(y)=\Lambda \phi(y)\int_{0}^y\Sigma_{0,3}\G(x)xdx-\G(y)\int_0^y\Sigma_{0,3}(x)\Lambda \phi(x)xdx
$$
 for which we obtain from \fref{estsigmazerothree}:
 \be
 \label{esttzerothree}
\forall y\leq 2B_1, \ \  |T_{0,3}(y)|\lesssim \frac{1+y^3}{b|\log b|},
 \ee
 with the crude bound from \fref{otehrboudttheree}:
 \be
 \label{toherbvoiueo}
 |T_{0,3}(y)|\lesssim (1+y^5).
 \ee
Observe that for $y\le 2B_1$
$$
b |T_{0,2}|+b |T_{1,1}|+b|T_{2,0}|+b|S_{0,2}|+b^2 |T_{0,3}|\lesssim \frac {y|\log y|^2}{\log b}
$$
 
 {\bf step 7} Estimate on $\bPsi_0$.\\
 
 We compute from \fref{computationpsiber}, \fref{deftone}, \fref{eqtzerotwo}, \fref{eqtonone}, \fref{eqtzerothree}:
  \bea
 \label{nokeoueofueoue}
 \Psi_0^{(1)}&= & ab\left\{-\Sigma_b\right\}\\
\nonumber & + & ab^3\big\{T_{1,1}\Delta S_{0,2}-2(1+Z)T_{1,1}\pa_yT_{0,2}-ZT_{0,3}-2(1+Z)T_{0,3}\pa_yT_{1,0}+S_{0,2}\Sigma_{1,1}\\
\nonumber & - & \frac{2Z(1+Z)}{y}(T_{1,0}T_{0,3}+T_{0,2}T_{1,1})\big\}\\
\nonumber & + & a^2b\left\{-\frac{2Z(1+Z)}{y}T_{1,0}T_{1,1}-2(1+Z)T_{1,1}\pa_yT_{1,0}-ZT_{1,1}-\Lambda T_{0,2}\right.\\
\nonumber & - & \left .2\frac{Z(1+Z)}{y}T_{2,0}T_{0,1}-2(1+Z)T_{0,1}\pa_yT_{2,0}\right\}\\
\nonumber & + & a^3b\left\{-\frac{2Z(1+Z)}{y}T_{2,0}T_{1,1}-2(1+Z)T_{1,1}\pa_yT_{2,0}\right\}\\
\nonumber & + & b^5\left\{-\frac{2Z(1+Z)}{y}T_{0,2}T_{0,3}+T_{0,3}(\Delta S_{0,2}-2(1+Z)\pa_yT_{0,2})+S_{0,2}\Sigma_{0,3}\right\}\\
\nonumber & + & a^2b^3\left\{-\frac{2Z(1+Z)}{y}T_{2,0}T_{3,0}-2(1+Z)T_{0,3}\pa_yT_{2,0}\right\}
\eea
\bea
\label{nekheoeouefo}
& &  \Psi_0^{(2)}  =    b^2\left\{-\Sigma_b\right\}\\
\nonumber & + &b^4\big\{-T_{0,2}\Delta S_{0,2}+\frac{2Z(1+Z)}{y}T_{0,2}^2+2(1+Z)T_{0,2}\pa_yT_{0,2}\\
\nonumber & - & S_{0,2}\Sigma_{0,2}+2(1+Z)S_{0,2}\pa_yS_{0,2}-\Lambda T_{0,3}\big\}\\
\nonumber & + & ab^2\left\{\frac{4Z(1+Z)}{y}T_1T_{0,2}-T_{1,0}\Delta S_{0,2}+2(1+Z)(T_{1,0}\pa_yT_{0,2}+T_{0,2}\pa_yT_{1,0})-\Lambda T_{1,1}+ZT_{0,2}\right\}\\
 \nonumber & + & a^3\left\{\frac{4Z(1+Z)}{y}T_{2,0}T_{1,0}+2(1+Z)(\pa_yT_{2,0}T_{1,0}+T_{2,0}\pa_yT_{1,0})+ZT_{2,0}\right\}\\
 \nonumber & + & a^2b^2\left\{-S_{0,2}HT_{2,0}+\frac{4Z(1+Z)}{y}T_{0,2}T_{2,0}-T_{2,0}\Delta S_{0,2}+2(1+Z)(\pa_yT_{0,2}T_{2,0}+T_{0,2}\pa_yT_{2,0})\right\}\\
 \nonumber & + & a^4\left\{\frac{2Z(1+Z)}{y}T_{2,0}^2+2(1+Z)T_{2,0}\pa_y T_{2,0}\right\},
 \eea
 \bea
 \label{computationpsibeqbisbis}
&& \Psi_0^{(3)}= b^3\left\{T_1\Sigma_b\right\}\\
\nonumber & + & a^2b\left\{-T_{1,0}\Sigma_{1,1}-T_{2,0}\Lambda \phi+T_{0,1}\Sigma_{2,0}\right\}\\
\nonumber & + & ab^3\left\{-T_{1,0}\Sigma_{0,3}-2(1+Z)T_{1,1}\pa_yS_{0,2}\right\}\\
\nonumber & + & a^3b\left\{-T_{2,0}\Sigma_{1,1}+T_{1,1}\Sigma_{2,0}\right\}\\
\nonumber & + & b^5\left\{-T_{0,2}\Sigma_{0,3}+T_{0,3}\Sigma_{0,2}-2(1+Z)T_{0,3}\pa_yS_{0,2}\right\}\\
\nonumber & + & a^2b^3\left\{-T_{2,0}\Sigma_{0,3}+T_{0,3}\Sigma_{2,0}\right\}
\eea
We estimate the error pointwise using  the assumption $$|a|\leq \frac{b}{|\log b|}$$ and the bounds on $T_{i,j}$. We first observe the high order cancellation near the origin:
$$|\Psi^{(1)}_0+C_1aby^3|+|\Psi_0^{(2)}+\frac{b^2}{|\log b|}y^3|+\frac{|\Psi_0^{(3)}|}{b}\lesssim \frac{b^2}{|\log b|}y^5\ \ \mbox{for} \ \ y\leq 1$$ for some universal constants $C_1,C_2$.
We next estimate for $y\leq 2B_1$:
\be
 \label{defpsivnoej}
  \Psi_0^{(1)}=  ab\left\{-\Sigma_{b}\right\}+\frac{b^2}{|\log b|^2}O\left(y{\bf 1}_{y\leq B_0}+(by^2)y{\bf 1}_{B_0\leq y\leq 2B_1}\right),
\ee
 \be
 \label{cenofheofeo}
 \Psi_0^{(2)}=b^2\left\{-\Sigma_{b}\right\}+\frac{b^2}{|\log b|}O\left(y{\bf 1}_{y\leq B_0}+(by^2)y{\bf 1}_{B_0\leq y\leq 2B_1}\right),
 \ee
 \be
 \label{estimatepsiothree}
 \Psi_0^{(3)}=\frac{b^2}{|\log b|}O\left[(by^2)|\log y|^2{\bf 1}_{y\leq 2B_1}\right].
 \ee 
This yields the bound:
$$\int_{y\leq 2B_1}  \frac{|\Psi_0^{(1)}|^2}{y^6}+\int_{y\leq 2B_1}  \frac{|\Psi_0^{(2)}|^2}{y^6}+\frac{1}{b^2}\int_{y\leq 2B_1} \frac{|\Psi_0^{(3)}|^2}{y^8}\lesssim \frac{b^4}{|\log b|^2}.$$ Estimates for higher order derivatives and localized bounds are obtained similarily and \fref{estderivatoveweight}, \eqref{estderivloc}, \eqref{estderivloc3} follow.\\

 {\bf step 8} Estimate on $H\bPsi_0$.\\

We now aim at deriving improved estimates for $\Psi_0^{(i)}$, $i=1,2$. The construction of $T_{i,j}$ 
results in the following cancellations for $y\leq 2B_1$:
\bea
\label{htwocomputone}
\nonumber & & H\left(\Psi_0^{(1)}\right)  =  -ab\left\{c_b\chi_{\frac{B_0}{4}}\Lambda \phi-d_bH\left[(1-\chi_{3B_0})\Lambda\phi)\right]\right\}+ab^3\left\{\Sigma_{0,3}\right\}+a^2b\left\{\Sigma_{1,1}\right\}\\
\nonumber & + & ab^3H\big\{T_{1,1}\Delta S_{0,2}-2(1+Z)T_{1,1}\pa_yT_{0,2}-(1+Z)T_{0,3}-2(1+Z)T_{0,3}\pa_yT_{1,0}+S_{0,2}\Sigma_{1,1}\\
 & - & \frac{2Z(1+Z)}{y}(T_{1,0}T_{0,3}+T_{0,2}T_{1,1})\big\}\\
\nonumber & + & a^2bH\left\{-\frac{2Z(1+Z)}{y}T_{1,0}T_{1,1}-2(1+Z)T_{1,0}\pa_yT_{1,0}-(1+Z)T_{1,1}-\Lambda T_{0,2}\right.\\
\nonumber & - & \left .2\frac{Z(1+Z)}{y}T_{2,0}T_{0,1}-2(1+Z)T_{0,1}\pa_yT_{2,0}\right\}\\
\nonumber & + & a^3bH\left\{-\frac{2Z(1+Z)}{y}T_{2,0}T_{1,1}-2(1+Z)T_{1,1}\pa_yT_{2,0}\right\}\\
\nonumber & + & b^5H\left\{-\frac{2Z(1+Z)}{y}T_{0,2}T_{0,3}+T_{0,3}(\Delta S_{0,2}-2(1+Z)\pa_yT_{0,2})+S_{0,2}\Sigma_{0,3}\right\}\\ 
\nonumber & + & a^2b^3H\left\{-\frac{2Z(1+Z)}{y}T_{2,0}T_{3,0}-2(1+Z)T_{0,3}\pa_yT_{2,0}\right\}\\
\nonumber & = & -ab\left\{c_b\chi_{\frac{B_0}{4}}\Lambda \phi-d_bH\left[(1-\chi_{3B_0})\Lambda\phi)\right]\right\}\\
\nonumber & + & a^2bO\left[(1+y){\bf 1}_{y\leq 6B_0}+\frac{|\log y|^3}{1+y}\right]+ab^3O\left(\frac{1+y}{b|\log b|}+(1+y)|\log y|^2\right)\\
\nonumber & + & a^3bO\left(\frac{|\log y|^3}{y}\right)++b^5O\left[(1+y^3)|\log y|^2\right]+a^2b^3O\left((1+y)|\log y|^3\right),
\eea
 \bea
 \label{htwocomputtwo}
 & & H\left( \Psi_0^{(2)}\right)   =    b^2\left\{-c_b\chi_{\frac{B_0}{4}}\Lambda \phi+d_bH\left[(1-\chi_{3B_0})\Lambda\phi)\right]\right\}\\
\nonumber& - & b^4H(\Lambda T_{0,3})-ab^2\left\{H\Lambda T_{1,1}-\Sigma_{0,2}\right\}\\
\nonumber & + &b^4H\left\{-T_{0,2}\Delta S_{0,2}+2(1+Z)T_{0,2}\pa_yT_{0,2}+\frac{2Z(1+Z)}{y}T_{0,2}^2-S_{0,2}\Sigma_{0,2}+2(1+Z)S_{0,2}\pa_yS_{0,2}\right\}\\
 \nonumber& + & ab^2H\left\{\frac{4Z(1+Z)}{y}T_1T_{0,2}-T_{1,0}\Delta S_{0,2}+2(1+Z)(T_{1,0}\pa_yT_{0,2}+T_{0,2}\pa_yT_{1,0})+(1+Z)T_{0,2}\right\}\\
  \nonumber & + & a^3H\left\{\frac{4Z(1+Z)}{y}T_{2,0}T_{1,0}+2(1+Z)(\pa_yT_{2,0}T_{1,0}+T_{2,0}\pa_yT_{1,0})+ZT_{2,0}\right\}\\
 \nonumber & + & a^2b^2H\left\{-S_{0,2}\Sigma_{2,0}+\frac{4Z(1+Z)}{y}T_{0,2}T_{2,0}-T_{2,0}\Delta S_{0,2}+2(1+Z)(\pa_yT_{0,2}T_{2,0}+T_{0,2}\pa_yT_{2,0})\right\}\\
 \nonumber & + & a^4H\left\{\frac{2Z(1+Z)}{y}T_{2,0}^2+2(1+Z)T_{2,0}\pa_y T_{2,0}\right\},\\
 \nonumber &= & b^2\left\{-c_b\chi_{\frac{B_0}{4}}\Lambda \phi+d_bH\left[(1-\chi_{3B_0})\Lambda\phi)\right]\right\}+a^3O\left(\frac{|\log y|^3}{1+y}\right)\\
 \nonumber& - & b^4\left\{H(\Lambda T_{0,3})+O\left[(1+y)|\log y|^2\right]\right\}+ ab^2O\left[(1+y){\bf 1}_{y\leq 6B_0}+\frac{|\log y|^2}{1+y}\right]\\
  \nonumber & + & a^3O\left(\frac{|\log y|^2}{1+y}\right)+a^2b^2O\left(\frac{|\log y|^3}{y}\right)+a^4O\left(\frac{|\log y|^4}{1+y^3}\right).
\eea
 In particular, using  \fref{esttzerothree}, we obtain the bounds:
 $$\int_{y\leq 2B_1}|H\Psi_0^{(1)}|^2+|H\Psi_0^{(2)}|^2\lesssim  \frac{b^4}{|\log b|^2}\left[|\log b|+(bB_1^2)^2\right]\lesssim b^4|\log b|^2,
 $$
  $$\int_{y\leq 2B_1}\frac{|H\Psi_0^{(1)}|^2+|H\Psi_0^{(2)}|^2}{y^2}\lesssim  \frac{b^4}{|\log b|^2},$$
 and \fref{htwolossy}, \fref{estderivatoveweighithH} are proved.\\
 We now turn to the flux computation \fref{fluxcomputationone}, \fref{fluwnumbertwo}. From \fref{htwocomputtwo}, \fref{estunphim}:
 \bee
 \frac{(H(\Psi_0^{(2)}),\Phi_M)}{(\Lambda\phi, \Phi_M)} &=& \frac{1}{(\Lambda\phi, \Phi_M)}\left[\left(-b^2c_b\chi_{\frac{B_0}{4}}\Lambda \phi,\Phi_M\right)+O\left(C(M){b^3}\right)\right]\\
 & = & -c_bb^2+O\left(C(M){b^2}\right)=-\frac{2b^2}{|\log b|}+O\left(\frac{b^2}{|\log b|^2}\right),
\eee
and similarily from \fref{htwocomputone}:
 \bee
 \frac{(H(\Psi_0^{(1)}),\Phi_M)}{(\Lambda\phi, \Phi_M)} &=& \frac{1}{(\Lambda\phi, \Phi_M)}\left[\left(-abc_b\chi_{\frac{B_0}{4}}\Lambda \phi,\Phi_M\right)+O\left(C(M){b^3}\right)\right]\\
 & = & -\frac{2ab}{|\log b|}+O\left(\frac{b^2}{|\log b|^2}\right).
\eee

  {\bf step 9} $H^4$ estimates.\\
 
 We estimate from \fref{htwocomputone}, \fref{htwocomputtwo}: $$\int_{y\leq 2B_1}|H^2\Psi_0^{(1)}|^2+\int_{y\leq 2B_1}|H^2\Psi_0^{(2)}|^2\lesssim \frac{b^6}{|\log b|^2}+b^8\int_{y\leq B_1}|H^2(\Lambda T_{0,3})|^2.$$
 The estimate \fref{estimatecrucial} will follow from the bound
 \be
 \label{neioeohoe}
 b^8 \int_{y\leq B_1}|H^2(\Lambda T_{0,3})|^2\lesssim \frac{b^6}{|\log b|^2}
 \ee
 {\it Proof of \fref{neioeohoe}}: We  compute using \fref{toherbvoiueo}:
 $$
 H(\Lambda T_{0,3}) =  2\Sigma_{0,3}+\Lambda \Sigma_{0,3}-\frac{\Lambda V}{y^2}T_{0,3},
 $$
 \bee
 H^2(\Lambda T_{0,3})=H\left(2\Sigma_{0,3}+\Lambda \Sigma_{0,3}\right)+O\left(\frac{1}{1+y}\right).
 \eee
Furthermore,  from \fref{defsigmazerothree}:
 \bee
 H(\Sigma_{0,3})& = &  -H\left(T_{0,1}\Delta S_{0,2}-2(1+Z)T_{0,1}\pa_yT_{0,2}-\Lambda T_{0,2}+2T_{0,2}\right)\\
 & = & H(\Lambda T_{0,2}-2T_{0,2})+O\left(\frac{(\log y)^3}{y}\right)\\
 & = & \Lambda \Sigma_{0,2}+O\left(\frac{(\log y)^5}{y}\right)
 \eee
 and thus by rescaling:
 $$H(\Lambda\Sigma_{0,3})=2\Lambda \Sigma_{0,2}+\Lambda^2\Sigma_{0,2}+O\left(\frac{(\log y)^5}{y}\right).$$ Using the bounds \fref{boundsigmatwoone}, \fref{boundsigmatwotwo} we conclude that
 \bee
& & b^8 \int_{y\leq B_1}|H^2(\Lambda T_{0,3})|^2\\
& \lesssim & \frac{b^8}{|\log b|^2}\left[\int_{y\leq 6B_0}(1+y^2)(1+|\log(y\sqrt{b})|)^2+\int_{6B_0\leq y\leq 2B_1} \frac{(\log y)^4}{y^2}\right]+b^8\int_{y\leq 2B_1}\frac{(\log y)^{10}}{1+y^2}\\
& \lesssim & \frac{b^6}{|\log b|^2}
\eee
as desired for \fref{neioeohoe}.\\
This concludes the proof of Proposition \ref{propprofile}.
 
 
 \subsection{Localization of the profile}
 
 
 In this section, we modify the approximate profile $\bw_0$ constructed above 
 to obtain a slowly modulated blow up profile {\it localized} in the zone $y\leq 2B_1$.
 
 \begin{proposition}[Localized profile]
 \label{propositionlocalization}
 Let a $\mathcal C^1$ map $s\to (b(s),a(s))\in \Bbb R^*_+\times \Bbb R$ defined on $[0,s_0]$ with a priori bounds: $\forall s\in [0,s_0]$, 
 \be\label{aprioirerror}
 |a|\leq \frac{b}{|\log b|}, \ \ 0<b<b^*, \ \ |a_s|\leq \frac{b^2}{|\log b|}, \ \ |b_s|\leq 10 b^2.
 \ee
 We define the localized profile $$\tbw_0=\left|\begin{array}{lll}\talpha_0\\\tbeta_0\\\tgamma_0\end{array}\right.=\chi_{B_1}\bw_0$$ defined from the cut-off profiles:
 $$\tilde{T}_{i,j}=\chi_{B_1}T_{i,j}, \ \ \tilde{T}_1=\chi_{B_1}T_1,\ \ \tilde{S}_{0,2}=\chi_{B_1}S_{0,2}.
 $$ 
 and the modulation vector:
 \bea
 \label{defmodulation}
 \mbox{Mod}(t):= \mbox{Mod}(s,y) & = & -a_s\left|\begin{array}{lll}\tt_1\\0\\0\end{array}\right .-(b_s+b^2)\left|\begin{array}{lll}2b\tt_{0,2}\\\tt_1\\2b\tilde{S}_{0,2}\end{array}\right .\\
 \nonumber & + & \left(\lsl+b\right)\left|\begin{array}{lll}\Lambda\phi+\Lambda\alphat_0+\gammat_0\Lambda\phi\\ \Lambda\tbeta_0\\\Lambda \tgamma_0-\talpha_0\Lambda\phi\end{array}\right .-(\Theta_s+a)\left|\begin{array}{lll}-\tbeta_0 Z\\\Lambda\phi+\talpha_0Z+\tgamma_0\Lambda\phi\\-\tbeta_0\Lambda\phi\end{array}\right.
 \eea
  The profile $\tbw_0$ satisfies the equation:
  \be
 \label{equationprofile}
-\pa_s\tbw_0+\lsl\Lambda\tbw_0-\Theta_sZR\tbw-(e_z+\tbw_0) \wedge\left[\Bbb H\tbw_0 +\Lambda \phi\left|\begin{array}{lll}\Theta_s\\ \lsl\\0\end{array}\right.\right]=\mbox{Mod}(t)+\btPsi_0
\ee
with the bounds:
\be
\label{infinityb}
|\tPsi_0|^2\lesssim \frac {b^2}{|\log b|^2},
\ee
\be
\label{estderivatoveweightbis}
\int \frac{|\pa_y^{i}\tPsi_0^{(1)}|^2}{y^{6-2i}}+\int  \frac{|\pa_y^{i}\tPsi_0^{(2)}|^2}{y^{6-2i}}\lesssim \frac{b^4}{|\log b|^2},\ \ 0\leq i\leq 3,
\ee
\be
\label{estderivatoveweightthreebis}
\int \frac{|\pa_y^{i}\tPsi_0^{(3)}|^2}{y^{8-2i}}\lesssim \frac{b^6}{|\log b|^2},\ \ 0\leq i\leq 4,
\ee
\be
\label{htwolossybis}
\int|H\tPsi_0^{(1)}|^2+|H\tPsi_0^{(2)}|^2\lesssim b^4 |\log b|^2,
\ee
\be
\label{estderivatoveweighithHbis}
\int \frac{|\pa_y^{i}H\tPsi_0^{(1)}|^2}{y^{2-2i}}+ 
\frac{|\pa_y^{i}H\tPsi_0^{(2)}|^2}{y^{2-2i}}\lesssim \frac{b^4}{|\log b|^2},\ \ 0\leq i\leq 1,
\ee
\be
\label{interp}
\int|A H\tPsi^{(i)}_{0}|^2\lesssim b^5, \ \ i=1,2,
\ee
 \be
\label{estimatecruciabisl}
\int|H^2\tPsi^{(i)}_{0}|^2\lesssim \frac{b^6}{|\log b|^2}, \ \ i=1,2,
\ee
\be
\label{fluxcomputationonebis}
\frac{(H\tPsi_0^{(2)},\Phi_M)}{(\Lambda\phi, \Phi_M)}=-\frac{2b^2}{|\log b|}+O\left(\frac{b^2}{|\log b|^2}\right),
\ee
\be
\label{fluwnumbertwobis}
\frac{(H\tPsi_0^{(1)},\Phi_M)}{(\Lambda\phi, \Phi_M)}=-\frac{2ab}{|\log b|}+O\left(\frac{b^2}{|\log b|^2}\right).
\ee
\end{proposition}

{\bf Proof of Proposition \ref{propositionlocalization}}\\

{\bf step 1} Equation for $\tbw_0$.\\

We first compute:
\bee
& & \lsl\Lambda\tbw_0-\Theta_sZR\tbw_0-(e_z+\tbw_0) \wedge\left[\Lambda \phi\left|\begin{array}{lll}\Theta_s\\ \lsl\\0\end{array}\right.\right]=  -b\tbw_0+aZR\tbw_0-(e_z+\tbw_0) \wedge\left[\Lambda \phi\left|\begin{array}{lll}-a\\ -b\\0\end{array}\right.\right]\\
& + & \left(\lsl+b\right)\left|\begin{array}{lll}\Lambda\phi+\Lambda\alphat_0+\gammat_0\Lambda\phi\\ \Lambda\tbeta_0\\\Lambda \tgamma_0-\talpha_0\Lambda\phi\end{array}\right .-(\Theta_s+a)\left|\begin{array}{lll}-Z\tbeta_0 \\\Lambda\phi+Z\talpha_0+\tgamma_0\Lambda\phi\\-\tbeta_0\Lambda\phi\end{array}\right.,
\eee
and
\bee
-\pa_s\tbw_0= & - & a_s\left|\begin{array}{lll}\tt_1+2a\tt_{2,0}\\b\tt_{1,1}\\ 0\end{array}\right .\\
& - & b_s\left|\begin{array}{lll}2b\tt_{0,2}+b^2\frac{\pa\tt_{0,2}}{\partial b}+a^2\frac{\pa \tt_{2,0}}{\pa b}\\ \tt_1+b\frac{\pa\tt_1}{\partial b}+3b^2\tt_{0,3}+b^3\frac{\pa\tt_{0,3}}{\partial b}+a\tt_{1,1}+ab\frac{\pa\tt_{1,1}}{\partial b}\\2b\tilde{S}_{0,2}+b^2\frac{\pa\tilde{S}_{0,2}}{\partial b}\end{array}\right ..
\eee
Combining this with  \fref{defmodulation} we have the explicit formula \fref{equationprofile} for $\tPsi_0$:
$$
\btPsi_0 = \bE_1+\bE_2
$$ with 
\be
\label{deferroreone}
\bE_1=- a_s\left|\begin{array}{lll}2a\tt_{2,0}\\b\tt_{1,1}\\ 0\end{array}\right .-  b_s\left|\begin{array}{lll}b^2\frac{\pa\tt_{0,2}}{\partial b}+a^2\frac{\pa \tt_{2,0}}{\pa b}\\ b\frac{\pa\tt_1}{\partial b}+3b^2\tt_{0,3}+b^3\frac{\pa\tt_{0,3}}{\partial b}+a\tt_{1,1}+ab\frac{\pa\tt_{1,1}}{\partial b}\\b^2\frac{\pa\tilde{S}_{0,2}}{\partial b}\end{array}\right .,
\ee
\be
\label{eferroroetwo}
\bE_2=-b \Lambda\tbw_0+a Z R \tbw_0-  (e_z+\tbw_0)\wedge\left[\H\tbw_0-\Lambda \phi\left|\begin{array}{lll}a\\ b\\0\end{array}\right.\right]+\left|\begin{array}{lll}
2b(b^2+a^2)\tt_{0,2}\\ (b^2+a^2)\tt_1+a(b^2+a^2) \tt_{1,1}\\2b(b^2+a^2)\tilde{S}_{0,2}\end{array}\right ..
\ee

{\bf step 2} Modulation error term $\bE_1$.\\

From \fref{defsigmab}, 
  $$\frac{\partial c_b}{\partial b}=O\left(\frac{1}{b|\log b|^2}\right), \ \ \frac{\partial \log c_b}{\partial b}=O\left(\frac{1}{b|\log b|}\right),$$
  and thus
 \bee
\frac{\partial \Sigma_{b}}{\partial b}(y) & = &  \frac{\partial c_b}{\partial b}T_1{\bf 1}_{y\leq \frac{B_0}{8}}+O\left(\frac{1}{b^2y|\log b|}{\bf 1}_{\frac{B_0}{8}\leq y\leq 6B_0}\right),
\eee 
which yields the bound: 
 \be
 \label{cneoeoyfeo}
  \frac{\partial \Sigma_{b}}{\partial b}(y)=O\left(\frac{y}{b|\log b|}{\bf 1}_{y\leq 6B_0}\right).
 \ee  
  We now estimate using the explicit formula for $T_{i,j}$: 
  $$ \frac{\partial \tt_{0,2}}{\partial b}=O\left(\frac{1+y^3}{b|\log b|}{\bf 1}_{y\leq B_0}\right)+O\left(\frac{1+y}{b^2|\log b|}{\bf 1}_{B_0\leq y\leq 2B_1}\right).
  $$
  $$\frac{\partial \tt_{1,1}}{\partial b}=O\left(\frac{1+y^3}{b|\log b|}{\bf 1}_{y\leq B_0}\right)+O\left(\frac{1+y}{b^2|\log b|}{\bf 1}_{B_0\leq y\leq B_1}\right), 
  $$
 $$
  \frac{\partial \tt_{0,3}}{\partial b}=O\left(\frac{1+y^5}{b|\log b|}{\bf 1}_{y\leq B_0}\right)+O\left(\frac{1+y^3}{b^2|\log b|}{\bf 1}_{B_0\leq y\leq B_1}\right),
$$  
$$\frac{\partial \tt_{2,0}}{\partial b}=O\left(\frac{(1+y)|\log y|^3}{b}{\bf 1}_{B_0\leq y\leq 2B_0}\right),$$ $$\frac{\partial\tilde{S}_{0,2}}{\partial b}=O\left(\frac{y^2|\log y|^2}{b}{\bf 1}_{B_1\leq y\leq 2B_1}\right).$$
We therefore obtain
  $$
|\bE_1^{(1)}|+|\bE_1^{(2)}|  \lesssim  b^2\left[O\left(\frac{1+y}{|\log b|}{\bf 1}_{y\leq 2B_1}\right)+O\left(\frac{by^3}{|\log b|}{\bf 1}_{B_0\leq y\leq 2B_1}\right)\right],
  $$
 $$
  |\bE_1^{(3)}|\lesssim O\left(b^3y^2|\log y|^2{\bf 1}_{B_1\leq y\leq 2B_1}\right)
 $$
  which  gives the bounds
  $$\int \frac{|\pa_y^{i}\bE_1^{(1)}|^2}{y^{6-2i}}+\int  \frac{|\pa_y^{i}\bE_i^{(2)}|^2}{y^{6-2i}}\lesssim \frac{b^4}{|\log b|^2},\ \ 0\leq i\leq 3,$$
    $$\int |H\bE_1^{(1)}|^2+\int|H\bE_1^{(2)}|^2\lesssim \frac {b^4}{|\log b|^2} (bB_1^2)^2\le b^4|\log b|^2,$$
$$\int \frac{|\pa_y^{i}\bE_1^{(3)}|^2}{y^{8-2i}}\lesssim \frac{b^6}{|\log b|^2},\ \ 0\leq i\leq 4.$$
  These estimates show that $\bE_1$ is negligible for $y\leq M$ in the flux computation \fref{fluxcomputationonebis}, \fref{fluwnumbertwobis}. \\
  We now seek additional cancellations for $H^2\bE_1$. 
  We have:
  $$H^2\left(\frac{\partial \tt_{0,2}}{\partial b}\right)=\chi_{B_1}H^2\left(\frac{\partial T_{0,2}}{\partial b}\right)+O\left(\frac{1+y}{y^4b^2|\log b|}{\bf 1}_{B_1\leq y\leq 2B_1}\right).$$ By construction, 
  \bee
  H^2(\frac{\partial T_{0,2}}{\partial b}) & = & H(\frac{\partial \Sigma_{0,2}}{\pa b})=H(\frac{\partial \Sigma_b}{\pa b})\\
  & = & O\left(\frac{1}{b(1+y)|\log b|^2}{\bf 1}_{y\leq 6B_0}\right)+O\left(\frac{1}{b(1+y)|\log b|}{\bf 1}_{B_0\leq y\leq 6B_0}\right)
  \eee
  and thus
  $$H^2\left(\frac{\partial \tt_{0,2}}{\partial b}\right)=O\left(\frac{{\bf 1}_{y\leq 6B_0}}{b(1+y)|\log b|^2}\right)+O\left(\frac{{\bf 1}_{B_0\leq y\leq 6B_0}}{b(1+y)|\log b|}\right)+O\left(\frac{1+y}{y^4b^2|\log b|}{\bf 1}_{B_1\leq y\leq 2B_1}\right).$$
  Similarily, $$H^2(\frac{\partial T_{1,1}}{\partial b})=O\left(\frac{{\bf 1}_{y\leq 6B_0}}{b(1+y)|\log b|^2}\right)+O\left(\frac{{\bf 1}_{B_0\leq y\leq 6B_0}}{b(1+y)|\log b|}\right)+O\left(\frac{1+y}{y^4b^2|\log b|}{\bf 1}_{B_1\leq y\leq 2B_1}\right).$$
  Finally, 
  $$ H^2(\frac{\partial \tt_{0,3}}{\partial b})=\chi_{B_1}H^2\left(\frac{\partial T_{0,3}}{\partial b}\right)+O\left(\frac{1+y^3}{y^4b^2|\log b|}{\bf 1}_{B_1\leq y\leq 2B_1}\right)$$ and 
  $$H^2(\frac{\partial T_{0,3}}{\partial b})=H(\frac{\partial \Sigma_{0,3}}{\partial b})=O\left(\frac{y}{b|\log b|}{\bf 1}_{y\leq B_0}\right)+O\left(\frac{1}{b^2y|\log b|}{\bf 1}_{B_0\leq y\leq 6B_0}\right).$$ 
  We then obtain
  $$H^2(\frac{\partial \tt_{0,3}}{\partial b})=O\left(\frac{y}{b|\log b|}{\bf 1}_{y\leq B_0}\right)+O\left(\frac{1}{b^2y|\log b|}{\bf 1}_{B_0\leq y\leq 6B_0}\right)+O\left(\frac{1+y^3}{y^4b^2|\log b|}{\bf 1}_{B_1\leq y\leq 2B_1}\right).$$
Using the bound:
$$|H^2(\tt_{2,0})|+|H^2(\tt_{1,1})|\lesssim \frac{1}{(1+y)|\log b|}{\bf 1}_{y\leq 6B_0}+O\left(\frac{|\log y|^2}{1+y^3} {\bf 1}_{B_1\leq y\leq 2B_1}\right)$$ 
we derive the pointwise bound:
\bee
&&|H^2\bE_1^{(1)}|+ |H^2\bE_1^{(2)}|\lesssim \frac{b^3}{|\log b|}O\left(\frac{|\log y|^2}{1+y^3}\right)\\
& & + O\left(\frac{b^3{\bf 1}_{y\leq 6B_0}}{(1+y)|\log b|^2}\right)+O\left(\frac{b^3{\bf 1}_{B_0\leq y\leq 6B_0}}{(1+y)|\log b|}\right)+O\left(\frac{b^3}{y|\log b|}{\bf 1}_{B_1\leq y\leq 2B_1}\right)
\eee
and therefore the desired improvement:
$$\int |H^2\bE_1^{(1)}|^2+\int |H^2\bE_1^{(2)}|^2\lesssim \frac{b^6}{|\log b|^2}.$$

{\bf step 3} Localization error $\bE_2$.\\

Let $\bE_2$ be given by \fref{eferroroetwo}. In view of \fref{defpisovectorial}, $E_2$ contains all the error terms induced by the localization. With the exception of $(1-\chi_{B_1}\Lambda\phi)$ all the localization error terms are supported in the region 
$[B_1,2B_1]$.
Using the formulas \fref{computationpsiber}, \fref{computationpsibetau}, the  bounds on $T_{i,j}$ and 
 \fref{defpsivnoej}, \fref{cenofheofeo}, \fref{estimatepsiothree} we obtain  
 \bea
 \label{cnoheoheo}
 \nonumber && \bE_2^{(1)} =  \chi_{B_1}\Psi_0^{(1)}-b\,(1-\chi_{B_1})\Lambda \phi+b\,O\left(\frac{|\log b|}{B_1}{\bf 1}_{B_1\leq y\leq 2B_1}\right)+b^3\,O\left(\frac{B_1}{b}{\bf 1}_{B_1\leq y\leq 2B_1}\right)\\
 \nonumber & + & ab\,O\left(B_1|\log b|{\bf 1}_{B_1\leq y\leq 2B_1}\right)+  \frac{b^2}{|\log b|}O\left[(by^2)y{\bf 1}_{B_1\leq y\leq 2B_1}\right]\\
 & = & \chi_{B_1}\Psi_0^{(1)}+O\left(b^2B_1{\bf 1}_{B_1\leq y\leq 2B_1}\right)+O\left(\frac{b^3y^3}{|\log b|}{\bf 1}_{B_1\leq y\leq 2B_1}\right)+O\left(\frac{b}{y}{\bf 1}_{y\geq B_1}\right),
 \eea
 \bea
 \label{vndovnoroiig}
\nonumber &&  \bE_2^{(2)}=\chi_{B_1}\Psi_{0}^{(2)}+a(1-\chi_{B_1})\Lambda \phi+ aO\left(\frac{|\log b|}{B_1}{\bf 1}_{B_1\leq y\leq 2B_1}\right)+b^2O\left(B_1|\log b|{\bf 1}_{B_1\leq y\leq 2B_1}\right)\\
\nonumber & + & a^2\left(\frac{|\log b|^3}{B_1}{\bf 1}_{B_1\leq y\leq 2B_1}\right)+O\left(\frac{b^3y^3}{|\log b|}{\bf 1}_{B_1\leq y\leq 2B_1}\right)\\
   & = & \chi_{B_1}\Psi_{0}^{(2)}+O\left(b^2|\log b|B_1{\bf 1}_{B_1\leq y\leq 2B_1}\right)+O\left(\frac{b^3y^3}{|\log b|}{\bf 1}_{B_1\leq y\leq 2B_1}\right)+O\left(\frac{b}{y}{\bf 1}_{y\geq B_1}\right),
 \eea
 \bea
 \label{vjdbohoehoeeo}
\nonumber  \bE_2^{(3)} & = & \chi_{B_1}\Psi_0^{(3)}+ab\,O\left(\frac{y|\log y|}{1+y}{\bf 1}_{B_1\leq y\leq 2B_1}\right)+b^3O\left(\frac{y}{b|\log b|(1+y)}{\bf 1}_{B_1\leq y\leq 2B_1}\right)\\
 & + & \frac{b^2}{|\log b|}O\left((by^2)^3|\log y|^2{\bf 1}_{B_0\leq y\leq 2B_1}\right).
  \eea
Higher order derivatives are estimated similarily. We then easily check from there that $\bE_2$ that does not perturb the estimates \fref{estderivatoveweightbis}--\fref{fluwnumbertwobis}.\\
 Finally, the estimate \eqref{interp} 
follows by interpolating between \eqref{htwolossybis} and \eqref{estimatecruciabisl}:
  $$
  \int |AH\tPsi_0^{(i)}|^2=   \int A^*AH\tPsi_0^{(i)}\, H\tPsi_{0}^{(i)} \le   \left (\int |H^2\tPsi_0^{(i)}|^2\right)^{\frac 12}  
  \left ( \int |H\tPsi_0^{(i)}|^2\right)^{\frac 12}.
  $$  
This concludes the proof of Proposition \ref{propositionlocalization}.
 

\section{The bootstrap regime and Lyapunov control }

 
 We derive in this section the dynamical tools at the heart of the description of the blow up solutions for initial data close to $Q_{b_0,a_0}$. We in particular set up the bootstrap argument, compute the modulation equations and derive the mized Energy/Morawetz functional which is our main tool to measure dispersion.



\subsection{Bootstrap setup and orthogonality conditions}


Given a sufficiently large constant $M>0$, define the function:
\be
\label{defphim}
\Phi_M=\chi_M\Lambda\phi-c_MH(\chi_M\Lambda\phi)
\ee
with $$c_M=\frac{(\chi_M\Lambda\phi,T_1)}{(H(\chi_M\Lambda\phi),T_1)}=c_{\chi}\frac{M^2}{4}(1+o_{M\to +\infty}(1)).$$ 
The choice of the constant $c_M$ ensures that 
\be
\label{controlltwo}
\int |\Phi_M|^2\lesssim |\log M|, \ \ (\Phi_M,T_1)=0
\ee
and that the scalar products \be
\label{estunphim}
(\Lambda\phi, \Phi_M)=(HT_1,\Phi_M)=(\chi_M\Lambda \phi,\Lambda\phi)=4\log M(1+o_{M\to +\infty}(1))
\ee
 are non degenerate. We recall the localized approximate profile
$$\tbw_0=\left|\begin{array}{lll}\talpha_0\\\tbeta_0\\\tgamma_0\end{array}\right.,$$ 
 dependent on the time-dependent parameters $a(t), b(t)$, given by Proposition \ref{propositionlocalization}
and consider the full modulated profile
 $$e^{\Theta R}(Q+\tbw_0)_{\lambda},$$ 
determined by the four parameters $(\lambda,\Theta,a,b)$. 
By the standard modulation theory, following from the application of the implicit function theorem,
 the invertibility of the Jacobian matrix: 
 $$\left|\begin{array}{llll} (-\Lambda \phi,\Phi_M)& 0 & (H\Lambda \phi,\Phi_M)& 0\\ 0& (\Lambda \phi,\Phi_M) & 0 & (H\Lambda\phi,\Phi_M)\\  (T_1,\Phi_M) &0& (HT_1,\Phi_M)&0\\  0 & (T_1,\Phi_M)& 0&(HT_1,\Phi_M)\end{array}\right|=(\Lambda\phi,\Phi_M)^4\neq 0$$ 
at $(\lambda,\Theta,a,b)=(1,0,0,0)$ ensures that
any smooth map close enough to $Q$ in the $\dot{H}^1$ topology admits a unique decomposition
 \be
 \label{decompuun}
 u=\mathcal S(Q+v)(t,x)
 \ee with $v$ expressed in the Frenet basis associated to $\S Q$ in the form 
 \be
 \label{decompudeux}
v=\bw+\tbw_0, \ \ \bw=\left|\begin{array}{lll}\alpha\\\beta\\\gamma\end{array}\right., \ \ \tbw_0=\left|\begin{array}{lll}\talpha_0\\\tbeta_0\\\tgamma_0\end{array}\right.,
\ee
and where he coordinate functions $(\alpha,\beta)$ are chosen to satisfy the orthogonality conditions: 
 \be
\label{choiceortho}
(\alpha,\Phi_M)=(\beta,\Phi_M)=0, \ \ (H\alpha,\Phi_M)=(H\beta,\Phi_M)=0.
\ee
\begin{remark}
The non-degeneracy of the above Jacobian also implies that any map 
$$
u=\S(Q+\tbw_0+\bw)
$$
with $\S$ and $\tbw_0$ defined by the parameters $(\lambda,\Theta,a,b)$ close to the point
$(1,0,0,0)$ and $\bw$ such that $\bw$ is small in the ${\dot{H}^1}$  topology and 
$\|\bw\|_{\dot{H}^1(y\le 2M)}\le C(M) b^{10}$
admits a new decomposition 
$$
u=\S^{(1)}(Q+\tbw^{(1)}_0+\bw^{(1)})
$$
defined by the parameters $(\lambda^{(1)},\Theta^{(1)},a^{(1)},b^{(1)})$ with $\bw^{(1)}$ satisfying the 
orthogonality conditions \eqref{choiceortho}. The map   $(\lambda,\Theta,a,b)\to (\lambda^{(1)},\Theta^{(1)},a^{(1)},b^{(1)})$
is a diffeomorphism of the form 
$$
(\lambda^{(1)},\Theta^{(1)},a^{(1)},b^{(1)})=(\lambda,\Theta,a,b)+\Psi(\lambda,\Theta,a,b)
$$
with the property that $\|\Psi(\lambda,\Theta,a,b)\|_{C^1}\le C(M) b^{10}$. In particular, if we allow 
the parameter $a$ to vary in the range $|a|\le \frac{b}{2|\log b|}$ then the range of values of the 
adjusted parameter $a^{(1)}$ will contain at least the interval $[-\frac{b}{4|\log b|},\frac{b}{4|\log b|}]$.
\label{rem:a}
\end{remark}
Let 
$$
\bw^\perp=\left|\begin{array}{lll}\alpha\\\beta\\0\end{array}\right.
$$
We now introduce the Sobolev norms adapted to the linear operator $\H$:
\be
\label{defefour}
\mathcal E_2=\int |\A^*\A\bw^\perp|^2,\ \ \mathcal E_4=\int|(\A^*\A)^2 \bw^\perp|^2.
\ee
Here, 
$$
\A^*\A\bw^\perp=\left|\begin{array}{lll} H\alpha\\ H\beta\\0\end{array}\right.,
$$
see \eqref{eq:defA}, \eqref{eq:defA*}.\\

We make the following assumptions on the initial data:
\begin{itemize}
\item Control of $a(0), b(0)$: 
\be
\label{inital control}
0<b(0)<b^*(M)\ll 1, \ \ |a(0)|\leq \frac {b(0)}{4|\log b(0)|}.
\ee 
Moreover, we may up to a fixed rescaling on data assume: $$\lambda(0)=1.$$

\item Energy control of $\bw(0)$:
\be
\label{initialboundddirichlet}
\int|\nabla \bw(0)|^2+\int\left|\frac{\bw(0)}{y}\right|^2\leq \delta(b^*)\ll1,
\ee
\be
\label{controlinit} 
|\mathcal E_2(0)|+|\mathcal E_4(0)|\leq b(0)^{10},
\ee
where here and in the sequel, $\delta(b^*)$ denotes some generic constant with 
\be
\label{smamlelcno}
\delta(b^*)\to 0\ \ \mbox{as}\ \ b^*\to 0.
\ee
\end{itemize}

The propagation of regularity by the equivariant Schr\"odinger map flow ensures that these bounds are propagated on some small time interval $[0,t_1)$. Given a large enough universal constant $K>0$, independent of $M$, we assume the following bounds: $\forall t\in[0,t_1]$,
\begin{itemize}
\item Pointwise control of $a,b$: 
\be
\label{inital}
0<b(t)\leq Kb^*(M), \ \ |a(t)|\leq \frac{b(t)}{|\log b(t)|};
\ee 
\item Energy control of $\bw(t)$:
\be
\label{initialboundddirichletbootinit}
\int|\nabla \bw(t)|^2+\int\left|\frac{\bw(t)}{y}\right|^2\leq K\delta(b^*),
\ee
\be
\label{htwobootinit}
|\mathcal E_2(t)|\leq K b^2(t)|\log b(t)|^7,
\ee
\be
\label{controlinitbootinit} 
|\mathcal E_4(t)|\leq K\frac{b^4(t)}{|\log b(t)|^2}.
\ee
\end{itemize}

At the heart of our analysis is the statement that these bounds are propagated all the way to blow up time, or equivalently:

\begin{proposition}[Trapped regime]
\label{propboot}
Assume that $K$ has been chosen large enough, independent of $M$. Then we can find $$a(0)=a(b(0), w(0))\in [-\frac{b(0)}{4|\log b(0)|},\frac{a(0)}{4|\log b(0)|}$$ such that the corresponding solution to \fref{nlsmap} satisfies: $\forall t\in [0,t_1)$, 
\begin{itemize}
\item Pointwise control of $a$ by $b$: 
\be
\label{initalcontrolboot}
0<b(t)\leq \frac{K}{2}b^*(M),\ \ |a(t)|\leq \frac{b(t)}{2|\log b(t)|};
\ee 
\item Energy control of $\bw(t)$:
\be
\label{initialboundddirichletboot}
\mathcal E_1(t)=\int|\nabla \bw(t)|^2+\int\left|\frac{\bw(t)}{y}\right|^2\leq \frac{K}{2}\delta(b^*),
\ee
\be
\label{htwoboot}
|\mathcal E_2(t)|\lesssim \frac{K}{2} b^2(t)|\log b(t)|^7,
\ee
 and for some universal constant $\eta>0$, independent of $M$, 
\be
\label{controlinitboot} 
|\mathcal E_4(t)|\leq (1-\eta)K\frac{b^4(t)}{|\log b(t)|^2}.
\ee
\end{itemize}
\end{proposition}

\begin{remark}
\label{remarkparmaeters} All along the proof of Proposition \ref{propboot}, we implicitely assume that $M$ is some large enough parameter, and then this fixes the smallness of initial data through the constant $b^*=b^*(M)$ in \fref{inital control}, and the smallness of $\delta(b^*)$ constants \fref{smamlelcno} which will appear all along the file. In particular, for a given generic constant $C(M)$, we can always assume $$C(M)\delta(b^*)\lesssim \delta(b^*)\to 0\ \ \mbox{as}\ \ b^*=b^*(M)\to 0.
$$
\end{remark}

The rest of this section is devoted to the derivation of the Lyapunov type functional for $\bw$, which  in section \ref{sectionboot}
will be shown to lead to  the proof of Proposition \ref{propboot}.


\subsection{Vectorial equation in the Frenet basis}


We now consider the decomposition \fref{decompuun}, \fref{decompudeux} and recall from \fref{eqcoridntesglobale} the notation $$\hbw=\tbw_0+\bw$$ and the normalization
\be
\label{normalization}
|e_z+\widehat{\bw}|^2=1.
\ee
We introduce the two matrices: \be
\label{defrz}
R_z=e_z\wedge, \ \  \hJ=(e_z+\hbw)\wedge.
\ee
We write the equation for $\bw$ in the Frenet basis using \fref{fulleqaito}, \fref{equationprofile}:
\be
\label{equationdefvectorial}
 \pa_s \bw-\lsl\Lambda \bw+\Theta_sZR\bw=-\widehat{J}\Bbb H\bw-\bw\wedge\left[\H\tbw_0+\Lambda\phi\left|\begin{array}{lll}\Theta_s\\\lsl\\0\end{array}\right. \right]+\mbox{Mod}(t)+\btPsi_0.
  \ee
 Let 
 \be
 \label{defF}
 {\bf  f}=-\Theta_sZR\bw-\bw\wedge\left[\H\tbw_0+\Lambda\phi\left|\begin{array}{lll}\Theta_s\\\lsl\\0\end{array}\right. \right]+\mbox{Mod}(t)+\btPsi_0
 \ee
 which we rewrite in a more explicit form:
 \be
 \label{fcoordinates}
 {\bf  f}=-\Theta_sZR\bw+\mbox{Mod}(t)+\btPsi_0-\bw\wedge\left|\begin{array}{lll} (\Theta_s \Lambda \phi+a H\tt_1)+b^2 H\tt_{0,2}+a^2H\tt_{2,0}-2(1+Z)\pa_y\tgamma_0\\ \left(\lsl\Lambda\phi+b H\tt_1\right)+abH\tt_{1,1}+b^3H\tt_{0,3}\\-\Delta \tgamma_0-2(1+Z)(\pa_y+\frac{Z}{y})\talpha_0\end{array}\right ..
 \ee 
 For the renormalized vector $$ \bW(t,r)=\bw(s,y),$$ we equivalently obtain the system in the original variables $(t,r)$:
$$\pa_t \bW=-\hJ_{\mu}\Bbb H_\mu \bW+{\bf F},\ \ {\bf F}=\frac{1}{\mu^2}{\bf f}_\mu$$ with $$\hJ_{\mu}=(e_z+\hbW)\wedge, \ \ \Bbb H_\mu{\bf W}=\left|\begin{array}{lll}H_\l\alpha_\mu\\H_\l\beta_\mu\\ -\Delta\gamma_\mu\end{array}\right. +\frac{2(1+Z_\mu)}{\mu}\left|\begin{array}{lll}-\pa_r\gamma_\mu\\0\\\pa_r\alpha_\mu+\frac{Z_\mu}{r}\alpha_\mu\end{array}\right ..$$
The rescaled operators are given by:
\be
\label{amuone}
A_{\mu}=-\pa_r+\frac{Z_{\mu}}{r},\ \ A_\mu^*=\pa_r+\frac{1+Z_\mu}{r}, \quad Z_\lambda(r)=Z\left(\frac r\lambda\right),
\ee
\be
\label{amutwo}
H_{\l}=A_{\mu}^*A_{\mu}=-\Delta+\frac{V_{\mu}}{r^2}, \quad V_\mu(r)=V\left(\frac r\lambda\right).
\ee


\subsection{Computation of the modulation equations}


In this section we derive the modulation equations for $(\lambda,\Theta,a,b)$ with the error terms
controlled in terms of the appropriate powers of the parameter $b$ and the energy $\mathcal E_4$.

\begin{lemma}[Modulation equations]
\label{roughbound}
There holds the following bounds on the modulation parameters:
\be
\label{roughboundparameters}
|\Theta_s+a|+|\lsl+b|\lesssim O\left(C(M)b^3\right),
\ee
\be
\label{loideb}
\left|b_s+b^2\left(1+\frac{2}{|\log b|}\right)\right|+|a_s+\frac{2ab}{|\log b|}|\lesssim \frac{1}{\sqrt{\log M}}\left(\sqrt{\mathcal E_4}+\frac{b^2}{|\log b|}\right).
\ee
\end{lemma}

\begin{remark} The $\frac{1}{\sqrt{\log M}}$ smallness gain in \fref{loideb} will be crucial for the rest of the analysis.
\end{remark}
 
{\bf Proof of Lemma \ref{roughbound}}\\

{\bf step 1} Equations for $a,b$.\\

Let \be
\label{vecotrmluadiati}
U(t)=|b_s+b^2|+|\lsl+b|+|\Theta_s+a|+|a_s|.
\ee 
We project \fref{equationdefvectorial} onto the first two coordinates, commute the resulting system with $H$ and then take the $L^2$ scalar product of each equation with $\Phi_M$. The linear terms $\pa_s(H\alpha,\Phi_M)$ and 
$\pa_s (H\beta,\Phi_M)$ vanish thanks to the choice of orthogonality conditions \fref{choiceortho}. Next, we observe from \fref{defmodulation} that on the support of $\Phi_M$, i.e., for $y\leq 2M$:
\bee
 \mbox{Mod}(s,y) & = &   -a_s\left|\begin{array}{lll}T_1\\0\\0\end{array}\right .-(b_s+b^2)\left|\begin{array}{lll}0\\T_1\\0\end{array}\right .+  \left(\lsl+b\right)\left|\begin{array}{lll}\Lambda\phi\\0\\0\end{array}\right .-(\Theta_s+a)\left|\begin{array}{lll}0\\\Lambda\phi\\0\end{array}\right.\\
 & + & O\left(C(M)b|U(t)|\right).
 \eee 
 All nonlinear terms in \fref{equationdefvectorial}  are estimated by brute force in the zone $y\leq 2M$ using the 
 bounds of Appendix B, and we obtain the system:
\bea
\label{coehoheh}
&&(H^2\beta,\Phi_M)-a_s(HT_1,\Phi_M)+(H\tPsi_0^{(1)},\Phi_M)\\
\nonumber &= & C(M)b\,O\left(\sqrt{\mathcal E_4}+\frac{b^2}{|\log b|}+|U(t)|\right),
\eea
$$-(H^2\alpha,\Phi_M)-(b_s+b^2)(HT_1,\Phi_M)+(H\tPsi_0^{(2)},\Phi_M)=C(M)b\,O\left(\sqrt{\mathcal E_4}+\frac{b^2}{|\log b|}+|U(t)|\right).$$
We now observe from \fref{controlltwo}:
 $$ |(H^2\alpha,\Phi_M)|+|H^2\beta,\Phi_M)|\lesssim \left[\|H^2\alpha\|_{L^2}+\|H^2\beta\|_{L^2}\right]\sqrt{\log M}\leq\sqrt{\mathcal E_4}\sqrt{8\log M}$$ from the definition of $\mathcal E_4$. The flux computation \fref{fluxcomputationonebis}, \fref{fluwnumbertwobis} and the growth \fref{estunphim} now imply
\be
\label{estakpaja}
\left|a_s+\frac{2ab}{|\log b|}\right|+\left|b_s+b^2+\frac{2b^2}{|\log b|}\right|\lesssim  O\left(\frac{b^2}{|\log b||\sqrt{\log M}}+\sqrt{b}|U(t)|\right)
\ee
for $|b|<b^*(M)$ small enough.\\

{\bf step 2} Control of the parameters $\lambda,\Theta$.\\

We project \fref{equationdefvectorial} onto the first two coordinates and take the $L^2$ scalar product of each equation with 
$\Phi_M$. The linear terms $\pa_s(\alpha,\Phi_M)$, $\pa_s(\beta,\Phi_M)$, $(H\alpha,\Phi_M)$ and $(H\beta,\Phi_M)$
vanish
thanks to the choice of the orthogonality conditions \fref{choiceortho}. The cancellation $(\Phi_M,T_1)=0$ of  \fref{controlltwo}
also eliminates the contribution of the leading order terms involving $a_s,b_s+b^2$. 
Treating the nonlinear terms, crudely, by the estimates of Appendix B, we obtain
\bea
\label{vnoeeiogere}
\nonumber (\Lambda\phi,\Phi_M)\left(|\lsl +b|+|\Theta_s+a|\right) & \lesssim & \left|(\tPsi_0^{(1)},\Phi_M)\right|+ \left|(\tPsi_0^{(2)},\Phi_M)\right|\\
& + & C(M)b\, O\left(\sqrt{\mathcal E_4}+|U(t)|\right).
\eea
We now compute from \fref{nokeoueofueoue}, \fref{nekheoeouefo} and \fref{estimationimportante} for $i=1,2$:
\bee
\left|\tPsi_0^{(i)},\Phi_M)\right| & \lesssim & b^2\left|(\Sigma_b,\Phi_M)\right|+O(C(M)b^3)=c_bb^2\left|(T_1,\Phi_M)\right|+O(C(M)b^3)\\
& = & O(C(M)b^3).
\eee
Injecting this into \fref{vnoeeiogere} yields together with \fref{estakpaja} the estimates \fref{roughboundparameters}, \fref{loideb}. This concludes the proof of Lemma \ref{roughbound}.


\subsection{Mixed Energy/Morawetz Lyapounov functional}

 
 We now turn to the heart of the proof which is the bootstrap control of the $\mathcal E_4$ norm  $$\mathcal E_4=\int|H^2 \bw^\perp|^2.$$ 
 This will be done through the derivation of a suitable mixed energy/Morawetz identity. 
 
 \begin{proposition}[Mixed energy/Morawetz estimate]
 \label{propmonononicity}
We have the following differential inequality
 \bea
 \label{crucialboot}
&& \frac{d}{dt}\left\{\frac{1}{\mu^6}\left[\mathcal E_4(1-\delta(b^*))+O\left(\frac{b^4}{\log b^2}\right)\right]\right\}\\
\nonumber  & \leq & \frac{b}{\mu^8}\left[2\left(1-d_0+\frac{C}{\sqrt{\log M}}\right)\mathcal E_4+O\left(\frac{b^4}{|\log b|^2}\right)\right]
 \eea
 with constants independent of $M$ and  
 \be
 \label{important}
 0<d_0<1.
 \ee
 \end{proposition}
 
  {\bf Proof of Proposition \ref{propmonononicity}}\\
  
  {\bf step 1} Vectorial formulation and adapted derivatives.\\
  
  Recall the notations \fref{defrz}. We begin by introducing  the suitable second derivative of $\bW$:  
  $$\bW_2:=\hJ_\mu \Bbb H_\mu \bW.$$ The $\bW$ equation:
  $$\pa_t \bW=-\bW_2+F$$ yields the equation for $\bW_2$: 
  $$  \pa_t\bW_2=-\hJ_\mu\Bbb H_\mu\bW_2 +[\pa_t,\hJ_\mu\Bbb H_\mu]\bW+\hJ_\mu\Bbb H_\mu {\bf F}.
  $$
 We then decompose
  $$\bW=\bW^\perp+W^ze_z, \ \ \bW^\perp=-R_z^2\bW, \ \ W^z=\bW\cdot e_z.$$ 
  We note the following formulas used extensively below:
  $$
\A R_z=R_z\A,\qquad \A^* R_z=R_z\A^*,\qquad   R_z\H R_z = R_z\A^* \A R_z.
  $$

  We now rewrite 
  \be
  \label{expansioncomut}
  [\pa_t,\hJ_\mu\H_\mu]\bW=\pa_t\hbW\wedge  \H_\mu \bW+\hbW\wedge[\pa_t,\H_\mu]\bW+R_z[\pa_t,\H_\mu]\bW
  \ee
 so that
   \be
  \label{equationwtwo}
  \pa_t \bW_2=-\hJ_\mu\Bbb H_\mu\bW_2 +R_z[\pa_t,\H_\mu]\bW+\mathcal Q_1+\hJ_\mu\Bbb H_\mu {\bf F}
  \ee
  with $$\mathcal Q_1=\pa_t\hbW\wedge  \H_\mu \bW+\hbW\wedge[\pa_t,\H_\mu]\bW.$$
  
  {\bf step 2} Energy identity for $\bW_2$.\\
  
 We now proceed to the derivation of a (second derivative)  energy identity on the $\bW_2$ equation. As was mentioned 
 earlier, the quasilinear nature of the system of equations for $\bW$ can lead to a potential loss of derivatives in energy
 estimates. To avoid this loss we have already introduced a nonlinear quantity $\bW_2$ and will now control the time
 derivative of a nonlinear second derivative of $\bW_2$. We compute
 \bee
&& \frac12\frac{d}{dt}\left\{\int|\hJ_\mu\H_\mu\bW_2|^2\right\}=   - \int \hJ_\mu\H_\mu\bW_2\cdot\left[\hJ_\mu\H_\mu \hJ_\mu\H_\mu\bW_2\right]+\int \hJ_\mu\H_\mu\bW_2\cdot(\hJ_\mu\H_\mu)^2{\bf F}\\
& + &   \int  \hJ_\mu\H_\mu\bW_2\cdot\left[[\pa_t,\hJ_\mu\H_\mu]\bW_2+\hJ_\mu\H_\mu\left\{R_z[\pa_t,\H_\mu]\bW+\mathcal Q_1\right\}\right]\\
& = &  - \int \hJ_\mu\H_\mu\bW_2\cdot\left[\hJ_\mu\H_\mu \hJ_\mu\H_\mu\bW_2\right]+\int \hJ_\mu\H_\mu\bW_2\cdot(\hJ_\mu\H_\mu)^2{\bf F}\\
& + &  \int R_z  \H_\mu\bW_2\cdot\left[[\pa_t,\hJ_\mu\H_\mu]\bW_2+\hJ_\mu\H_\mu\left\{R_z[\pa_t,\H_\mu]\bW\right\}\right]\\
& + &  \int  \hbW\wedge \H_\mu\bW_2\cdot\left[[\pa_t,\hJ_\mu\H_\mu]\bW_2+\hJ_\mu\H_\mu\left(R_z[\pa_t,\H_\mu]\bW\right)\right]+ \int \hJ_\mu\H_\mu\bW_2\cdot  \hJ_\mu\H_\mu\mathcal Q_1
 \eee
Let us introduce the decomposition of $\bW_2$: 
\be
\label{defwtwozero}
\bW_2=\bW_2^0+\bW_2^1, \ \ \bW_2^0=R_z\H_\mu \bW^\perp.
\ee
 We have the following identities:
 \begin{align*}
&R_z[\pa_t,\H_\mu]\bW=\frac{\pa_tV_\mu}{r^2}R_z\bW^\perp-2\frac{\mu^2\pa_tZ_{\mu}-\mu\mu_t(1+Z_\mu)}{\mu^3}\pa_r W^ze_y,\\
&R_z[\pa_t,\H_\mu]W^z e_z=-2\frac{\mu^2\pa_tZ_{\mu}-\mu\mu_t(1+Z_\mu)}{\mu^3}\pa_r W^ze_y.
\end{align*}
We rewrite
$$\pa_tV_\l=-\frac{1}{\l^2}\left(\lsl+b\right)(\Lambda V)_\l+\frac{b}{\l^2}(\Lambda V)_\l$$
and use the expansion \fref{expansioncomut} to obtain the energy identity:
\bea
\label{firstforuenehre}
 \nonumber &&\frac12\frac{d}{dt}\left\{\int|\hJ_\mu\H_\mu\bW_2|^2\right\}=   - \int \hJ_\mu\H_\mu\bW_2\cdot\left[\hJ_\mu\H_\mu \hJ_\mu\H_\mu\bW_2\right]+\int \hJ_\mu\H_\mu\bW_2\cdot(\hJ_\mu\H_\mu)^2{\bf F}\\
 \nonumber & + & \int  R_z\H_\mu\bW^0_2\cdot\left[\frac{b(\Lambda V)_\mu}{\l^2r^2}R_z\bW^0_2+R_z\H_\mu\left(\frac{b(\Lambda V)_\mu}{\l^2r^2}R_z\bW^\perp\right)\right]\\
 & + & \mathcal Q_2+ \int \hJ_\mu\H_\mu\bW_2\cdot  \hJ_\mu\H_\mu\mathcal Q_1
 \eea
 with
 \bea
 \label{defqun}
\nonumber  \mathcal Q_2 & = &  \int  R_z\H_\mu\bW^0_2\cdot\left[-\left(\lsl+b\right)\frac{(\Lambda V)_\l}{\l^2r^2}R_z\bW^0_2+R_z\H_\mu\left(-\left(\lsl+b\right)\frac{(\Lambda V)_\l}{\l^2r^2}R_z\bW^\perp\right)\right]\\
 & + &   \int  R_z\H_\mu(\bW^1_2)\cdot\left[R_z[\pa_t,\H_\mu]\bW^0_2+R_z\H_\mu\left(R_z[\pa_t,\H_\mu]\bW^\perp\right)\right]\\
\nonumber  & + &  \int  R_z\H_\mu\bW_2\cdot\left[R_z[\pa_t,\H_\mu](\bW^1_2)+R_z\H_\mu\left(R_z[\pa_t,\H_\mu](\bW^ze_z)\right)\right]\\ 
\nonumber & + &  \int R_z \H_\mu\bW_2\cdot\left[\pa_t\hbW\wedge  \H_\mu \bW_2+\hbW\wedge[\pa_\tau,\H_\mu]\bW_2+\hbW\wedge\H_\mu( R_z[\pa_t,\H_\mu]\bW)\right]\\
\nonumber  & + &\int  \hbW\wedge \H_\mu\bW_2\cdot\left[[\pa_t,\hJ_\mu\H_\mu]\bW_2+\hJ_\mu\H_\mu\left\{ R_z[\pa_t,\H_\mu]\bW\right\}\right].
 \eea
 
 {\bf step 3} Morawetz correction for the remaining quadratic interactions.\\

The second line of the energy identity \eqref{firstforuenehre}  still contains unsigned quadratic terms. To remove them 
we add another identity reminiscent, in spirit, to the Morawetz identity for the Schr\"odinger equation.\\
Let 
 \be\label{defwthree}
 \bW_3=R_z\Bbb A_{\mu}\bW^0_2, \ \ \Bbb A_{\mu}\bW=\left|\begin{array}{lll} A_{\mu}(\bW\cdot e_x)\\ A_{\mu}(\bW\cdot e_y)\\ 0\end{array}\right.
 \ee
 Note that $$R_z\H_\mu \bW^0_2=\A^*_\mu\bW_3.$$  
 We compute from $\bW^\perp=-R_z^2\bW$:
 \be
 \label{eqwperp}
 \pa_t\bW^\perp=-R_z^2(-\bW_2+{\bf F})=-\bW_2^0+R_z^2\bW_2^1+{\bf F^\perp}
 \ee
 \be
 \label{cheooewoiew}
 \pa_t\bW_2^0=-R_z\H_\mu\bW_2^0+\mathcal Q_3+R_z\H_\mu {\bf F^\perp}
  \ee
 with 
 \be
 \label{defqthree}
 \mathcal Q_3=[\pa_t,R_z\H_\mu]\bW^\perp+R_z\H_\mu(R_z^2\bW_2^1).
 \ee
 We then proceed with the following Morawetz type computation:
\bee
&&\frac{d}{dt}\left\{\int\left[2\frac{b(\Lambda Z)_{\mu}}{\l^2r}\bW^0_2+\A_\mu\left(\frac{b(\Lambda V_\mu)}{\l^2r^2}R_z\bW^\perp\right)\right]\cdot \bW_3\right\}\\
& = &  \int\left[2\frac{b(\Lambda Z)_{\mu}}{\l^2r}\pa_t \bW^0_2+\A_{\mu}\left(\frac{b(\Lambda V_\mu)}{\l^2r^2}R_z\pa_t \bW^\perp\right)\right]\cdot \bW_3\\
& + & \int\left[2\frac{b(\Lambda Z)_{\mu}}{\l^2r}\bW^0_2+\A_\mu\left(\frac{b(\Lambda V_\mu)}{\l^2r^2}R_z\bW^\perp\right)\right]\cdot R_z\A_{\mu}\pa_t \bW^0_2\\
 & + & \mathcal Q_4
 \eee
 with
 \bea
 \label{Qfour}
\nonumber \mathcal Q_4 & = &   \int \left[2\frac{d}{dt}\left(\frac{b(\Lambda Z)_{\mu}}{\l^2r}\right)\bW^0_2+\frac{\pa_{t}Z_\mu b(\Lambda V)_\mu}{\l^2r^3}R_z\bW^\perp+\A_\mu\left(\frac{d}{dt}\left(\frac{b(\Lambda V)_{\mu}}{\l^2r^2}\right)\bW^\perp\right)\right]\cdot \bW_3\\
& + & \int\left[2\frac{\pa_t Z_{\mu}}{r}\bW^0_2+\A_\mu\left(\frac{\pa_t V_\mu}{r^2}R_z\bW^\perp\right)\right]\cdot R_z\frac{\pa_t Z_\mu}{r}\bW^0_2.
\eea
 We now inject \fref{cheooewoiew} and compute the leading order quadratic terms:
\bee
&&\int\left[2\frac{b(\Lambda Z)_{\mu}}{\l^2r}(-R_z\H_{\mu}\bW^0_2)+\A_{\mu}\left(\frac{b(\Lambda V_\mu)}{\l^2r^2}(-R_zW^0_2)\right)\right]\cdot \bW_3\\
& = &  -2\int \frac{b(\Lambda Z)_{\mu}}{\l^2r}A^*_{\mu}\bW_3\cdot \bW_3-  \int\frac{b(\Lambda V_\mu)}{\l^2r^2}R_z\bW^\perp_2\cdot R_z\H_\mu \bW^0_2,
\eee 
and
 \bee
  &&\int\left[2\frac{b(\Lambda Z)_{\mu}}{\l^2r}\bW^0_2+\A_\mu\left(\frac{b(\Lambda V_\mu)}{\l^2r^2}R_z\bW^\perp\right)\right]\cdot R_z\A_{\mu}(-R_z \H_\mu\bW^0_2)\\
  & = & \int\left[-2 \A^*_{\mu}\left(\frac{b(\Lambda Z)_{\mu}}{\l^2r}\bW^0_2\right)-\H_\mu\left(\frac{b(\Lambda V_\mu)}{\l^2r^2}R_z \bW^\perp\right)\right] \cdot R_z^2\H_\mu \bW^0_2\\
  & = & \int \left[2\frac{b(\Lambda Z)_{\mu}}{\l^2r}\A_\mu \bW_2^0-2\frac{b(\Lambda V_\mu)}{\l^2r^2}\bW^0_2-\H_\mu\left(\frac{b(\Lambda V_\mu)}{\l^2r^2}R_z \bW^\perp\right)\right]\cdot R_z^2\H_\mu \bW^0_2\\
  & = &  -2\int \frac{b(\Lambda Z)_{\mu}}{\l^2r}A^*_{\mu}\bW_3\cdot \bW_3-2\int \frac{b(\Lambda V_\mu)}{\l^2r^2}\bW^0_2\cdot R_z^2\H_\mu \bW^0_2\\
  & - & \int \H_\mu\left(\frac{b(\Lambda V_\mu)}{\l^2r^2}R_z \bW^\perp\right)\cdot R_z^2\H_\mu \bW_2^0.
   \eee
   Here we used the algebra associated with the function $Z$:
 \bee
\pa_y(\pa_yZ)+\frac{1}{y^2}\Lambda Z(1+2Z) & = & \frac{1}{y^2}\left(\Lambda^2Z-\Lambda Z+\Lambda Z(1+2Z)\right)\\
& = & \frac{1}{y^2}\left[\Lambda^2Z+2Z\Lambda Z\right]=\frac{\Lambda V}{y^2},
\eee
which implies that for any function $f$:
 \bee
 \nonumber A^*_{\mu}\left(\frac{\Lambda Z}{y}f\right)& = & -A_{\mu}\left(\frac{\Lambda Z}{y}f\right)+\frac{1+2Z}{y}\frac{\Lambda Z}{y}f\\
\nonumber& = & -\frac{\Lambda Z}{y}Af+\frac{f}{y^2}\left[\pa_y\left(\frac{\Lambda Z}{y}\right)+\frac{1+2Z}{y}\frac{\Lambda Z}{y}\right]\\
& = & -\frac{\Lambda Z}{y}A f+\frac{\Lambda V}{y^2}f.
\eee
Similarily, the algebra $$\Lambda Z=Z^2-1\leq 0, \ \ -\frac12\Lambda^2Z+Z\Lambda Z=0,$$ 
implies, after an integration by parts:
\bea
\label{nskljouwoeu}
\nonumber\int \frac{b(\Lambda Z)_{\mu}}{\l^2r}\A^*_\mu \bW_3\cdot \bW_3 & = &\frac{b}{\l^8} \int \frac{|\bw_3|^2}{y^2}\left[-\frac{y}{2}\pa_y(\Lambda Z)+(1+Z)\Lambda Z\right]\\
&  =&  \int  \frac{b(\Lambda Z)_{\mu}}{\l^2r^2}|\bW_3|^2<0.
\eea 
We therefore obtain the Morawetz type formula:
\bee
&&\frac{d}{dt}\left\{\int\left[2\frac{b(\Lambda Z)_{\mu}}{\l^2r}\bW^0_2+\A_\mu\left(\frac{b(\Lambda V_\mu)}{\l^2r^2}R_z\bW^\perp\right)\right]\cdot \bW_3\right\}\\
& = & -4\int \frac{b(\Lambda Z)_{\mu}}{\l^2r}|\bW_3|^2+\int \left[\frac{b(\Lambda V_\mu)}{\l^2r^2}R_z\bW^0_2+R_z\H_\mu\left(\frac{b(\Lambda V_\mu)}{\l^2r^2}R_z\bW^\perp\right)\right]\cdot R_z\H_\mu\bW^0_2\\
& + & \mathcal Q_4\\
& + &  \int\left[2\frac{b(\Lambda Z)_{\mu}}{\l^2r}(\mathcal Q_3+R_z\H_\mu{\bf F^\perp})+\A_{\mu}\left(\frac{b(\Lambda V_\mu)}{\l^2r^2}R_z[R_z^2W_2^1+{\bf F^\perp}]\right)\right]\cdot \bW_3\\
& + & \int\left[2\frac{b(\Lambda Z)_{\mu}}{\l^2r}\bW^0_2+\A_\mu\left(\frac{b(\Lambda V_\mu)}{\l^2r^2}R_z\bW^\perp\right)\right]\cdot R_z\A_{\mu}\left(\mathcal Q_3+R_z\H_\mu{\bf F^\perp}\right)\\
& = & -4 \int \frac{|\bW_3|^2}{r^2}\frac{b(\Lambda Z)_{\mu}}{\l^2r}+\int \left[\frac{b(\Lambda V_\mu)}{\l^2r^2}R_z\bW^0_2+R_z\H_\mu\left(\frac{b(\Lambda V_\mu)}{\l^2r^2}R_z\bW^\perp\right)\right]\cdot R_z\H_\mu\bW^0_2\\
& + & \mathcal Q_4+\mathcal Q_5\\
& + & \int R_z\H_\mu{\bf F^\perp}\cdot\left[2\frac{b(\Lambda Z)_{\mu}}{\l^2r}\bW_3-2R_z\A_\mu^*\left(\frac{b(\Lambda Z)_{\mu}}{\l^2r}\bW_2^0\right)-\H_\mu\left(\frac{b(\Lambda V_\mu)}{\l^2r^2}\bW^\perp\right)\right]\\
& +& \int \frac{b(\Lambda V_\mu)}{\l^2r^2}R_z{\bf F^\perp} \cdot \A^*_\mu W_3
\eee
 with
 \bea
 \label{defqfive}
 \mathcal Q_5& = &  \int\left[2\frac{b(\Lambda Z)_{\mu}}{\l^2r}\mathcal Q_3-\A_{\mu}\left(\frac{b(\Lambda V_\mu)}{\l^2r^2}R_z\bW^1_2\right)\right]\cdot \bW_3\\
\nonumber  & + &  \int\left[2\frac{b(\Lambda Z)_{\mu}}{\l^2r}\bW^0_2+\A_\mu\left(\frac{b(\Lambda V_\mu)}{\l^2r^2}R_z\bW^\perp\right)\right]\cdot \A_{\mu}R_z\mathcal Q_3\\
\nonumber & = & \int \bW_3\cdot\left[\left(2\frac{b(\Lambda Z)_{\mu}}{\l^2r}\right)\left([\pa_t,R_z\H_\mu]\bW^\perp+R_z\H_\mu(R_z^2\bW_2^1)\right)-\A_{\mu}\left(\frac{b(\Lambda V_\mu)}{\l^2r^2}R_z\bW^1_2\right)\right]\\
\nonumber & + & \int R_z\left[\pa_t,R_z\H_\mu]\bW^\perp+R_z\H_\mu(R_z^2\bW_2^1)\right]\cdot \A^*_\l \left[2\frac{b(\Lambda Z)_{\mu}}{\l^2r}\bW^0_2+\A_\mu\left(\frac{b(\Lambda V_\mu)}{\l^2r^2}R_z\bW^\perp\right)\right]
 \eea
 {\bf step 4} Mixed energy/Morawetz Lyapunov functional.\\
 
 We combine the energy identity and the Morawetz identity and obtain
 \bea
 \label{miwedfucntional}
&&\nonumber \frac12\frac{d}{dt}\left\{\int|\hJ_\mu\H_\mu\bW_2|^2-2\int\left[\frac{b(\Lambda Z)_{\mu}}{\l^2r}\bW^0_2+\A_\mu\left(\frac{b(\Lambda V_\mu)}{\l^2r^2}R_z\bW^\perp\right)\right]\cdot \bW_3\right\}\\
 & = &    - \int \hJ_\mu\H_\mu\bW_2\cdot\left[\hJ_\mu\H_\mu \hJ_\mu\H_\mu\bW_2\right]+4\int \frac{b(\Lambda Z)_{\mu}}{\l^2r}|\bW_3|^2\\
\nonumber & + &  \mathcal Q_2+ \int \hJ_\mu\H_\mu\bW_2\cdot  \hJ_\mu\H_\mu\mathcal Q_1+\mathcal Q_4+\mathcal Q_5\\
\nonumber & + & \int \hJ_\mu\H_\mu\bW_2\cdot(\hJ_\mu\H_\mu)^2{\bf F}+ \int \frac{b(\Lambda V_\mu)}{\l^2r^2}R_z{\bf F^\perp} \cdot \A^*_\mu \bW_3\\
\nonumber & + & \int R_z\H_\mu{\bf F^\perp}\cdot\left[2\frac{b(\Lambda Z)_{\mu}}{\l^2r}\bW_3-2R_z\A_\mu^*\left(\frac{b(\Lambda Z)_{\mu}}{\l^2r}\bW_2^0\right)-\H_\mu\left(\frac{b(\Lambda V_\mu)}{\l^2r^2}R_z\bW^\perp\right)\right].
 \eea 
By  \fref{nskljouwoeu}, the leading order quadratic term has the right sign from \fref{nskljouwoeu}. 
After dealing with the quasilinear term 
$$
 \int \hJ_\mu\H_\mu\bW_2\cdot\left[\hJ_\mu\H_\mu \hJ_\mu\H_\mu\bW_2\right],
$$
we will aim at estimating all the remaining terms in the RHS of \fref{miwedfucntional}. In the process we will 
make an intensive implicit use of the interpolation estimates of Appendix B.\\

 {\bf step 5} Quasilinear term.\\
  
 The mixed energy/Morawetz identity \fref{miwedfucntional} is compatible with the quasilinear structure of the problem thanks to the following:
 
  \begin{lemma}[Gain of two derivatives]
  \label{gainderivatives}
  There exists a universal constant 
  \be
  \label{crucialdzero}
  0<d_1<1
  \ee
  \def\bG{{\bf \Gamma}}
  such that for any vector $\bG$ with radial coordinates in the Frenet basis, there holds:
  \be
  \label{defquadra}
  -\int \hJ\H \bG\cdot \hJ\H(\hJ\H \bG)\leq b\left[(1-d_1)\|\hJ\H\bG\|_{L^2}^2+\delta(b^*)\|\H\bG\|_{L^2}^2\right]
  \ee
  \end{lemma}
  
 Lemma \ref{gainderivatives} is crucial.  Its proof is mostly algebraic and is detailed in Appendix C.\\
 
 We now apply Lemma \ref{gainderivatives} with $\bG=\bw_2$. Observe from \fref{decompmpe}, \fref{betterhtwoglobal} that 
 \be
 \label{esthtwoefour}
  \|\hJ\H\bw_2\|_{L^2}^2=\mathcal E_4+\delta(b^*)\left(\mathcal E_4+\frac{b^4}{|\log b|^2}\right)
  \ee 
  and we thus, after rescaling, get the bound:
 \be
 \label{boundquailinearterm}
 - \int \hJ_\mu\H_\mu\bW_2\cdot\left[\hJ_\mu\H_\mu \hJ_\mu\H_\mu\bW_2\right]\leq \frac{b}{\l^8}\left[(1-d_0)\mathcal E_4+\delta(b^*)\frac{b^4}{|\log b|^2}\right]
 \ee
 for some universal constant $0<d_0<1$.\\
 
 {\bf step 6} Control of the Lyapunov functional.\\
 
 We estimate the Lyapunov functional 
 $$
\left\{\int|\hJ_\mu\H_\mu\bW_2|^2-2\int\left[\frac{b(\Lambda Z)_{\mu}}{\l^2r}\bW^0_2+\A_\mu\left(\frac{b(\Lambda V_\mu)}{\l^2r^2}R_z\bW^\perp\right)\right]\cdot \bW_3\right\}
 $$
appearing on the left hand side  in \fref{miwedfucntional}. From \fref{esthtwoefour}:
 $$\int|\hJ_\mu\H_\mu\bW_2|^2=\frac{1}{\l^6}\left[\mathcal E_4+ \delta(b^*)\left(\mathcal E_4+\frac{b^4}{|\log b|^2}\right)\right].$$ Next,
\bee
&& \left|\int\left[\frac{b(\Lambda Z)_{\mu}}{\l^2r}\bW^0_2+\A_\mu\left(\frac{b(\Lambda V_\mu)}{\l^2r^2}R_z\bW^\perp\right)\right]\cdot \bW_3\right|\\
& \lesssim & \frac{b}{\l^6}\left(\int |\A^*\bw_3|^2+\frac{|\bw_3|^2}{y^2(1+|\log y|^2)}\right)^{\frac12}\left(\int\frac{1+|\log y|^2}{1+y^4}|\bw_2^0|^2+\int \frac{|\bw^\perp|^2}{1+y^8}\right)^{\frac 12}\\
& \lesssim & \frac{b}{\l^6}|\log b|^C\left(\mathcal E_4+\frac{b^4}{|\log b|^2}\right)\lesssim \frac{\delta(b^*)}{\l^6}\left(\mathcal E_4+\frac{b^4}{|\log b|^2}\right),
\eee
where we used the logarithmic lossy bounds \fref{lossyboundwperp}.\\

{\bf step 7} Treatment of lower order $\mathcal Q$ terms.\\

To treat these terms, we will systematically use the bound: $$|\Lambda Z|+|\Lambda V|\lesssim \frac{y^2}{1+y^4}.$$ We first estimate from the $\bW$ equation:
$$\|\pa_t \bW\|^2_{L^{\infty}}\lesssim \|\bW_2\|^2_{L^{\infty}}+\|\bF\|^2_{L^{\infty}}\lesssim \|\bW_2\|^2_{L^{\infty}}+\frac{1}{\l^4}\|{\bf f}\|^2_{L^{\infty}}.
$$
We estimate from the estimates of Appendix B:
$$\|\bW_2\|^2_{L^{\infty}}\lesssim \frac{1}{\l^4}b^3|\log b|^C,$$ and from \fref{fcoordinates}: 
$$ 
 \|{\bf f}\|_{L^{\infty}}^2 \lesssim   \|\bw\|_{L^{\infty}}^2\left(|\Theta_s|^2+\frac{b^2}{|\log b|^2}\right)+\frac{b^2}{|\log b|^2}\lesssim  \frac{b^2}{|\log b|^2}.
$$
Therefore,
\be
\label{estdtwhat}
\|\pa_t\bW\|^2_{L^{\infty}}\lesssim \frac{b^2}{\l^4|\log b|^2}.
\ee
We now estimate one by one all terms appearing in \fref{defqun} with the help of the bounds in  Appendix B and \fref{estdtwhat}:
$$
\l^8|\mathcal Q_2|\lesssim \left|\lsl+b\right|b^4|\log b|^C+b\delta(b^*)\left(\mathcal E_4+\frac{b^4}{|\log b|^2}\right)\lesssim b\delta(b^*)\left(\mathcal E_4+\frac{b^4}{|\log b|^2}\right).$$
Similarily, from \fref{Qfour} and $|b_s|\lesssim b^2$:
$$\l^8|\mathcal Q_4|\lesssim (|b_s|+b^2)b^4|\log b|^C+b\delta(b^*)\left(\mathcal E_4+\frac{b^4}{|\log b|^2}\right)\lesssim b\delta(b^*)\left(\mathcal E_4+\frac{b^4}{|\log b|^2}\right),$$
and from \fref{defqfive}:
$$
\l^8|\mathcal Q_5| \lesssim b\delta(b^*)\left(\mathcal E_4+\frac{b^4}{|\log b|^2}\right).
$$
{\bf step 8} $\bF$ terms.\\

We rewrite the last line of \fref{miwedfucntional} as follows:
\bee
&&\int R_z\H_\mu{\bf F^\perp}\cdot\left[2\frac{b(\Lambda Z)_{\mu}}{\l^2r}\bW_3-2R_z\A_\mu^*\left(\frac{b(\Lambda Z)_{\mu}}{\l^2r}\bW_2^0\right)-\H_\mu\left(\frac{b(\Lambda V_\mu)}{\l^2r^2}R_z\bW^\perp\right)\right]\\
& = &  \int R_z\H_\mu{\bf F^\perp}\cdot 2\frac{b(\Lambda Z)_{\mu}}{\l^2r}\bW_3-2\int \A_{\mu}R_z\H_\mu{\bf F^\perp}\cdot R_z\frac{b(\Lambda Z)_{\mu}}{\l^2r}\bW_2^0\\
& - & \int \A_{\mu}R_z\H_\mu{\bf F^\perp}\cdot\A_\mu\left(\frac{b(\Lambda V_\mu)}{\l^2r^2}R_z\bW^\perp\right).
\eee
The $\bF$ contribution to \fref{miwedfucntional} is  estimated by Cauchy-Schwarz :
\bee
&&\left|\int \hJ_\mu\H_\mu\bW_2\cdot(\hJ_\mu\H_\mu)^2{\bf F}+ \int \frac{b(\Lambda V_\mu)}{\l^2r^2}R_z{\bf F^\perp} \cdot \A^*_\mu \bW_3\right.\\
\nonumber & + &\left. \int R_z\H_\mu{\bf F^\perp}\cdot\left[2\frac{b(\Lambda Z)_{\mu}}{\l^2r}\bW_3-2R_z\A_\mu^*\left(\frac{b(\Lambda Z)_{\mu}}{\l^2r}\bW_2^0\right)-\H_\mu\left(\frac{b(\Lambda V_\mu)}{\l^2r^2}R_z\bW^\perp\right)\right]\right|\\
& \lesssim & \frac{b}{\l^8}\sqrt{\mathcal E_4}\left[\frac{1}{b^2}\int |(\hJ\H)^2{\bf f}|^2+\int \frac{|{\bf f}^{\perp}|^2}{1+y^8}+\int\frac{1+|\log y|^2}{1+y^4}|R_z\H {\bf f}^{\perp}|^2\right]^{\frac 12}\\
& + &  \frac{b}{\l^8}\sqrt{\mathcal E_4}C(M)\left[\int \frac{1+|\log y|^2}{1+y^2}|\A R_z\H{\bf f}^\perp|^2\right]^{\frac12},
\eee
where the terms involving $\bW_3$ do not require a $C(M)$ constant thanks to the coercivity of the operator $\A^*$, combined with
 the two dimensional Hardy inequality: 
 $$\mathcal E_4=\|\A^*\bw_3\|_{L^2}^2\gtrsim \int |\pa_y\bw_3|^2+\int \frac{|\bw_3|^2}{1+y^4}\gtrsim \int |\pa_y\bw_3|^2+\int \frac{|\bw_3|^2}{y^2(1+|\log y|^2)}.$$ The same applies to the term containing $\hJ_\mu\H_\mu\bW_2$, thanks to the estimate 
\eqref{betterhtwoglobalbis} of Appendix B.\\
We now claim:
\be
\label{estunf}
\int \frac{|{\bf f}^{\perp}|^2}{1+y^8}+\int\frac{1+|\log y|^2}{1+y^4}|R_z\H {\bf f}^{\perp}|^2\lesssim\frac{b^4}{|\log b|^2}+\frac{ \mathcal E_4}{\sqrt{|\log M|}},
\ee
and the improved bound:
\be
\label{estdeuxf}
\frac{1}{b^2}\int |(\hJ\H)^2{\bf f}|^2+\frac 1{\delta(b^*)}\int \frac{1+|\log y|^2}{1+y^2}|\A R_z\H{\bf f}^\perp|^2\lesssim \frac{b^4}{|\log b|^2}+\frac{\mathcal E_4}{\log M},
\ee
which together with the above chain of estimates concludes the proof of \fref{crucialboot}. It thus remains to prove \fref{estunf}, \fref{estdeuxf}.\\
We recall the expression \eqref{fcoordinates} for ${\bf f}$:
 $$
 {\bf  f}=-\Theta_sZR\bw+\mbox{Mod}(t)+\btPsi_0-\bw\wedge\left|\begin{array}{lll} (\Theta_s \Lambda \phi+a H\tt_1)+b^2 H\tt_{0,2}+a^2H\tt_{2,0}-2(1+Z)\pa_y\tgamma_0\\ \left(\lsl\Lambda\phi+b H\tt_1\right)+abH\tt_{1,1}+b^3H\tt_{0,3}\\-\Delta \tgamma_0-2(1+Z)(\pa_y+\frac{Z}{y})\talpha_0\end{array}\right ..
 $$

{\bf step 9} $\btPsi_0$ contribution.\\

We claim the fundamental estimate:
\be
\label{estpsibfunade}
\int |(\hJ\H)^2\btPsi_0|^2\lesssim \frac{b^6}{|\log b|^2}.
\ee
The remaining estimates involving $\btPsi_0$ in \fref{estunf}, \fref{estdeuxf} follow directly from the estimates of Proposition \ref{propositionlocalization} and are left to the reader. In particular, we note that the second part of \eqref{estdeuxf} follows immediately 
from \eqref{interp}.\\
{\it Proof of \fref{estpsibfunade}}: We decompose $\hJ\H\btPsi_0=\mathcal C_1+\mathcal C_2$ with $$\mathcal C_1=e_z\wedge \H\btPsi_0=\left|\begin{array}{lll}-H\tPsi_0^{(2)}\\H\tPsi_0^{(1)}-2(1+Z)\pa_y\tPsi_0^{(3)}\\0\end{array}\right.,$$ $$\mathcal C_2=\hbw\wedge\H\btPsi_0=\left|\begin{array}{lll}\betah(-\Delta \tPsi_0^{(3)}+2(1+Z)(\pa_y+\frac{Z}{y})\tPsi_0^{(1)})-\gammah H\tPsi_0^{(2)}\\\gammah(H\tPsi_0^{(1)}-2(1+Z)\pa_y\tPsi_0^{(3)})-\alphah(-\Delta \tPsi_0^{(3)}+2(1+Z)(\pa_y+\frac{Z}{y})\tPsi_0^{(1)})\\\alphah H\tPsi_0^{(2)}-\betah(H\tPsi_0^{(1)}-2(1+Z)\pa_y\tPsi_0^{(3)})\end{array}\right.$$
Therefore,
\bee
\int|R_z\H\mathcal C_1|^2 & \lesssim& \int |H^2\tPsi_0^{(1)}|^2+ |H^2\tPsi_0^{(1)}|^2+\int\left|H\left((1+Z)\pa_y\tPsi_0^{(3)}\right)\right|^2\\
& \lesssim & \frac{b^6}{|\log b|^2},
\eee
From Proposition \ref{propositionlocalization}, \fref{estlinftysecond} and \eqref{estlinftysecondgamma}:
\bee
\int|\hbw\wedge\H\mathcal C_1|^2&\lesssim& \|\hbw\|_{L^{\infty}}^2\frac{b^6}{|\log b|^2}+\int|\hbw|^2\left|(1+Z)\left(\pa_y+\frac{Z}{y}\right)H\tPsi_0^{(2)}\right|^2\\
& \lesssim &  \frac{b^6}{|\log b|^2}+\left\|\frac{\hbw}{1+y^2}\right\|^2_{L^{\infty}}\int \left|\left(\pa_y+\frac{Z}{y}\right)H\tPsi_0^{(2)}\right|^2\\
& \lesssim & \frac{b^6}{|\log b|^2}.
\eee
The $\mathcal C_2$ bound
$$\int |(e_z+\hbw)\wedge \H\mathcal C_2|^2\lesssim \frac{b^6}{|\log b|^2}
$$
can be obtained in a similar fashion. Indeed, we estimate 
$$
\int |(e_z+\hbw)\wedge \H\mathcal C_2|^2\lesssim (1+ \|\hbw\|_{L^\infty})\left (\int_{y\le 1} \H\mathcal C_2|^2+ 
\int_{y\ge 1} \H\mathcal C_2|^2\right) \lesssim \int_{y\le 1} \H\mathcal C_2|^2+ 
\int_{y\ge 1} \H\mathcal C_2|^2.
$$
We concentrate on the region $y\ge 1$ as estimates for $y\le 1$ are easier. 
According to Proposition \ref{propositionlocalization}, the terms $\tPsi_0^{(3)}$ and 
$H^2\tPsi_0^{(i)}$ with $i=1,2$ exhibit better behavior with respect to $b$. Thus we consider a generic contribution 
to $\H\mathcal C_2$:
\begin{align*}
\Delta\left (\hat\alpha H\tPsi_0^{(2)}\right)&=\Delta\hat\alpha H\tPsi_0^{(2)}+\hat\alpha \Delta H\tPsi_0^{(2)}+
2 {\pa_y\hat\alpha} {\pa_y H\tPsi_0^{(2)}}\\ &= \Delta\hat\alpha H\tPsi_0^{(2)}+\hat\alpha  H^2\tPsi_0^{(2)}-
2 {\pa_y\hat\alpha} {AH\tPsi_0^{(2)}}+ H\tPsi_0^{(2)} O(\frac {| \hat\alpha|}{y^2}+ \frac {|\pa_y\hat\alpha|}{y})
\end{align*}
Using \eqref{estimatecruciabisl}, \eqref{htwolossybis}, \eqref{estperplinfty}, \eqref{estlinftysecond} we obtain
\begin{align*}
\int_{y\ge 1} |\Delta\hat\alpha H\tPsi_0^{(2)}|^2 &+ \int_{y\ge 1}|\hat\alpha  H^2\tPsi_0^{(2)}|^2+\int_{y\ge 1}
\left (\frac {|\hat\alpha|^2}{y^4}+ \frac {|\pa_y\hat\alpha|^2}{y^2}\right) |H\tPsi_0^{(2)}|^2 \\ &\lesssim \|\Delta\hat\alpha\|_{L^\infty(y\ge 1)}^2 
\int_{y\ge 1}  |H\tPsi_0^{(2)}|^2 + \|\hat\alpha\|_{L^\infty(y\ge 1)}^2 \int_{y\ge 1} |H^2\tPsi_0^{(2)}|^2\\ &+
\left(\|\frac {\hat\alpha}{y^2}\|_{L^\infty(y\ge 1)}+\|\frac {\pa_y\hat\alpha}{y}\|_{L^\infty(y\ge 1)}  \right) \int_{y\ge 1}
|H\tPsi_0^{(2)}|^2\\ &\lesssim b^3 |\log b|^5 b^4 + \delta(b^*) \frac {b^6}{|\log b|^2} + b^3 |\log b|^5 b^4 \lesssim
\delta(b^*)\frac {b^6}{|\log b|^2}.
\end{align*}
It remains to estimate 
\begin{align*}
\int |{\pa_y\hat\alpha} \,{AH\tPsi_0^{(2)}}|^2\lesssim 
 \|\pa_y \hat\alpha\|_{L^\infty}^2 \int |A H\tPsi_0^{(2)}|^2\lesssim b^2|\log b|^8 {b^5} \le \delta(b^*)\frac {b^6}{|\log b|^2},
\end{align*}
where we used \eqref{interp} and \eqref{estlinftydeuxfa}.\\

{\bf step 10} $\Mod(t)$ contribution.\\
 
 We recall the definitions of $U(t)$ and $\mbox{Mod}(t)$ given in \fref{vecotrmluadiati}  be given by \fref{defmodulation} respectively:
  \bee
\notag
 \mbox{Mod}(t)=& = & -a_s\left|\begin{array}{lll}\tt_1\\0\\0\end{array}\right .-(b_s+b^2)\left|\begin{array}{lll}2b\tt_{0,2}\\\tt_1\\2b\tilde{S}_{0,2}\end{array}\right .\\
 \nonumber & + & \left(\lsl+b\right)\left|\begin{array}{lll}\Lambda\phi+\Lambda\alphat_0+\gammat_0\Lambda\phi\\ \Lambda\tbeta_0\\\Lambda \tgamma_0-\talpha_0\Lambda\phi\end{array}\right .-(\Theta_s+a)\left|\begin{array}{lll}-\tbeta_0 Z\\\Lambda\phi+\talpha_0Z+\tgamma_0\Lambda\phi\\-\tbeta_0\Lambda\phi\end{array}\right.
 \eee
 
 We first estimate from \fref {roughboundparameters} and \eqref{loideb}:
 $$
 \int \frac{|\Mod(t)^{\perp}|^2}{1+y^8}+\int\frac{1+|\log y|^2}{1+y^4}|R_z\H \Mod(t)^{\perp}|^2 \lesssim U^2(t) \lesssim  \frac{b^4}{|\log b|^2}+\frac{\mathcal E_4}{|\log M|}.
$$
 Using the cancellation $$AHT_1=A\Lambda\phi=0$$ we also have  the improved bound:
 $$
 \int \frac{1+|\log y|^2}{1+y^2}|\A R_z\H\Mod^\perp|^2  \lesssim  bU^2(t)\lesssim \delta(b^*)\left(\frac{b^4}{|\log b|^2}+\delta(b^*) \mathcal E_4\right).
$$
The estimate for 
$$
\frac 1{b^2} \int |(\hJ\H)^2\Mod|^2
$$
is more involved. We first observe:
$$\Mod=-\Mod_1+\Mod_2, \ \ \Mod_2=O\left[\left(\frac{b^2}{|\log b|}+\sqrt{\mathcal E_4}\right)b(1+y)|\log y|{\bf 1}_{y\leq 2B_1}\right]
$$ 
with 
$$\Mod_1=a_s\left|\begin{array}{lll}\tt_1\\0\\0\end{array}\right .+(b_s+b^2)\left|\begin{array}{lll}2b\tt_{0,2}\\\tt_1\\0\end{array}\right .-\left(\lsl+b\right)\left|\begin{array}{lll}\Lambda\phi\\0\\0\end{array}\right .+(\Theta_s+a)\left|\begin{array}{lll} 0\\\Lambda \phi\\0\end{array}\right ..$$ The bound for $\Mod_2$ follows from \eqref{defmodulation} and easily leads to the desired estimate. We 
focus on $\Mod_1$. 
We compute:
\be
\label{expresionmodone}
\hJ\H\Mod_1=\left|\begin{array}{lll}-(b_s+b^2)H\tt_1\\-a_sH\tt_1-2b(b_s+b^2)H\tt_{0,2}\\ 0\end{array}\right . + \hbw\wedge O\left[U(t)\frac{|\log y|}{1+y}{\bf 1}_{y\leq 2B_1}\right].
\ee 
We now reapply the operator $\hJ \H$. The second term in estimated by brute force using the estimates of Appendix B. 
For instance, using \eqref{inteproloatedbound},
\begin{align*}
|U(t)|^2 \int |\bw^\perp|^2 |H \frac{\log y}{1+y}|^2\lesssim \frac {b^4}{|\log b|^2}  \int \frac {|\bw^\perp|^2|\log y|^2}{1+y^6} \le
 \frac {b^4}{|\log b|^2} b^3 |\log b|^C \le  \frac {b^6}{|\log b|^2}
\end{align*}
For the first term:
\bee
&&\hJ\H\left|\begin{array}{lll}-(b_s+b^2)H\tt_1\\-a_sH\tt_1-2b(b_s+b^2)H\tt_{0,2}\\ 0\end{array}\right .=(e_z+\hbw)\wedge\left|\begin{array}{lll}-(b_s+b^2)H^2\tt_1\\-a_sH^2\tt_1-2b(b_s+b^2)H^2\tt_{0,2}\\ -2(b_s+b^2)(1+Z)(\pa_y+\frac{Z}{y})H\tt_1\end{array}\right .\\
& = & U(t)O\left[\frac{|\log y|}{y^3}{\bf 1}_{B_1\leq y\leq 2B_1}+b\left(\frac{{\bf 1}_{y\leq 6B_0}}{(1+y)|\log b|}+\frac{1}{1+y^2}\right)+\frac{|\hbw|}{1+y^4}\right].
\eee

As a result we obtain the bound
$$\int |(\hJ\H)^2\Mod|^2\lesssim b^2|U(t)|^2\lesssim b^2\left(\frac{b^4}{|\log b|^2}+\frac{\mathcal E_4}{\log M}\right).$$

 {\bf step 11} Contribution of the phase term.\\
 
 Let $${\bf f}_1=-\Theta_s ZR_z\bw,$$ then using the bootstrap assumptions on $a$ and the modulation bounds of
 Lemma \ref{roughbound},
 \bee
\int \frac{|{\bf f}_1^{\perp}|^2}{1+y^8}+\int\frac{1+|\log y|^2}{1+y^4}|R_z\H {\bf f}_1^{\perp}|^2&+&\int \frac{1+|\log y|^2}{1+y^2}|AR_z\H{\bf f}_1^\perp|^2+\int|(\hJ\H)^2{\bf f}_1|^2\\
& \lesssim & |\Theta_s|^2C(M)\mathcal E_4 \lesssim  C(M)\frac{b^2}{|\log b|^2}\mathcal E_4.
\eee
Here, near the origin we used the decomposition $(\hJ\H)^2(ZR_z\bw)=(\hJ\H)^2(R_z\bw)+(\hJ\H)^2((Z-1)R_z\bw)$. 
  The bound $|Z-1|\lesssim y^2$ eliminates  a possible singularity at the origin, which may arise for instance from the term 
  \begin{align*}
  H^2\left ((Z-1) \alpha\right)&= H\left ((Z-1) H\alpha-\pa_y^2 (Z-1) \alpha - \frac {\pa_y (Z-1)}y \alpha -2 \pa_y (Z-1) \pa_y\alpha\right)
  \\ &=
  (Z-1) H^2\alpha-\pa_y^2 (Z-1) H\alpha - \frac {\pa_y (Z-1)}y H\alpha\\& - 2 \pa_y (Z-1) \pa_y H\alpha
  -\pa_y^2 (Z-1) H\alpha+\pa_y^4 (Z-1) H\alpha+2\pa_y^3 (Z-1) \pa_y\alpha \\&+\frac {\pa_y^3 (Z-1)}y \alpha-
   \frac {\pa_y (Z-1)}y H\alpha+  \pa_y^2\frac {\pa_y (Z-1)}y \alpha+  2\pa_y \frac {\pa_y (Z-1)}y \pa_y \alpha \\&+
   \frac 1y\pa_y \frac {\pa_y (Z-1)}y \alpha-2 \pa_y (Z-1) H\pa_y\alpha+2 \pa^3_y (Z-1) \pa_y\alpha\\&+
   2 \frac{\pa^2_y (Z-1)}y \pa_y\alpha+4 \pa^2_y (Z-1) \pa^2_y\alpha.
  \end{align*}
The only terms singular at the origin appears in the commutator 
\bee
&&-2\pa_y (Z-1) H\pa_y\alpha\\
& = & -2\pa_y (Z-1) \pa_y H\alpha +2\frac{\pa_y (Z-1)}{y^2} \pa_y \alpha-4 \frac{\pa_y (Z-1)}{y^3} V\alpha+2
\frac{\pa_y (Z-1)}{y^2} \pa_y V\alpha\\ 
&= & 2\pa_y (Z-1) \pa_y H\alpha +2
\frac{\pa_y (Z-1)}{y^2} \pa_y V\alpha-4 \frac{\pa_y (Z-1)}{y^3} (V-1)\alpha\\
& + & 2\frac{\pa_y (Z-1)}{y^2} \pa_y \alpha-4 \frac{\pa_y (Z-1)}{y^3} \alpha,
\eee
 where due to the vanishing $|V-1|\lesssim y^2$ and $|\pa_y V|\lesssim y$ we only need to take into account the last two terms,
 and the term
 $$ 
 2 \frac {\pa^2_y (Z-1)}{y}\pa_y \alpha.
 $$ 
 When combined, they generate 
 $$
 2\frac{\pa_y (Z-1)}{y^2} \pa_y \alpha-4 \frac{\pa_y (Z-1)}{y^3} \alpha+2 \frac {\pa^2_y (Z-1)}{y}\pa_y \alpha=
 4 \frac{\pa_y (Z-1)}{y^2} {A\alpha}+ O(|\pa_y\alpha|)
 $$
 The estimate 
 $$
 \int_{y\le 1} \frac {|A\alpha|^2}{y^2} \lesssim C(M)\mathcal E_4
 $$
 follows from \eqref{eq:anno}. The estimates for $y\ge 1$ can be obtained in a similar fashion to the terms already treated above. We omit the details.\\
   
  {\bf step 12} Contribution of the remaining term.\\
  
  In view of \fref{fcoordinates}, it remains to estimate the contribution of the term:
 \be
 \label{estftwo}
 {\bf  f}_2=\bw\wedge\left|\begin{array}{lll} (\Theta_s \Lambda \phi+a H\tt_1)+b^2 H\tt_{0,2}+a^2H\tt_{2,0}-2(1+Z)\pa_y\tgamma_0\\ \left(\lsl\Lambda\phi+b H\tt_1\right)+abH\tt_{1,1}+b^3H\tt_{0,3}\\-\Delta \tgamma_0-2(1+Z)(\pa_y+\frac{Z}{y})\talpha_0\end{array}\right ..
\ee

The singularity at the origin is treated using the vanishing of $\bw$ at the origin  and the estimates of Appendix B. We focus on 
the estimates for $y\geq 1$ for which we rewrite using the equation $HT_1=\Lambda \phi$:
\bea
\label{cnofheoheohftwo}
\nonumber {\bf f}_2=\bw\wedge & O& \left(|U(t)|\frac{1}{y}{\bf 1}_{y\leq 2B_1}+\frac{b}{y}{\bf 1}_{y\geq B_1}+b\frac{|\log y|{\bf 1}_{B_1\le y\leq 2B_1}}{y}+b^2y{\bf 1}_{y\leq 2B_1}\right.\\
&& \left . +\frac{b}{|\log b|}\frac{|\log y|}{1+y^2}{\bf 1}_{y\leq 2B_1}\right)
\eea
 from which 
  \bee
&&\int \frac{|{\bf f}_2^{\perp}|^2}{1+y^8}+\int\frac{1+|\log y|^2}{1+y^4}|R_z\H {\bf f}_2^{\perp}|^2+\int \frac{1+|\log y|^2}{1+y^2}|\A R_z\H{\bf f}_2^\perp|^2+\int|(\hJ\H)^2{\bf f}_2|^2\\ 
& \lesssim & \frac{C(M)b^2}{|\log b|^2}\left(\frac{b^4}{|\log b|^2}+\mathcal E_4\right)\lesssim  b^2\left(\frac{b^4}{|\log b|^2}+\delta(b^*)\mathcal E_4\right).
\eee
 This concludes the proof of \fref{estunf}, \fref{estdeuxf}.\\
  This concludes the proof of Proposition \ref{propmonononicity}.


\section{Closing the Bootstrap}

We are now in position to close the bootstrap and complete the proof of Proposition \ref{propboot}.

\subsection{Energy estimates}
\label{sectionboot}

In this section, we close the three Sobolev bounds of Proposition \ref{propboot}.\\


{\bf step 1} Energy bound.\\

The Dirichlet energy $\int_{{\Bbb R}^2}|\nabla v|^2$ 
for a map $$v=\alphah e_r+\betah e_{\tau}+(1+\gammah)Q$$ is given by the expression 
\bea
\label{cnoeiieouo}
\nonumber E(v) &=& (H\alphah,\alphah)+(H\betah,\betah)+(-\Delta\gammah,\gammah)+2\int(1+Z)\left(\gammah\pa_r\alphah-\alphah\pa_r\gammah+\frac{Z}{r}\alphah\gammah\right)\\
& + & \int(1+Z)^2,
\eea
where 
$$
E(Q)=\int(1+Z)^2.
$$
We now linearize this expression by letting $$\hat{\alpha}=\alphat_0+\alpha, \ \ \betah=\betat_0+\beta, \ \ \gammah=\gammat_0+\gamma.$$ We have by construction of the profile:
\be
\label{boudprofile}
 \|\pa_y {\tbw}_0\|_{L^2}^2+\|\frac{{\tbw}_0}{y}\|_{L^2}^2+|E( \alphat_0 e_r+\betat_0 e_{\tau}+(1+\gammat_0)Q)-E(Q)|\lesssim \sqrt{b}
 \ee and thus from the initial smallness assumption \fref{initialboundddirichlet}: $$|E(\hbw(0))-E(Q)|\lesssim \delta(b^*).$$ Moreover, the choice of orthogonality conditions \fref{choiceortho} ensures the coercivity estimate:
 \be
 \label{cojoeofje}
 (H\alpha,\alpha)+(H\beta,\beta)\geq C(M)\left[\|\pa_y\bw^\perp\|_{L^2}^2+\|\frac{\bw^\perp}{y}\|_{L^2}^2\right]
 \ee
see Appendix A. Returning to \eqref{cnoeiieouo}, the relation
$$
2\gammah=-\gammah^2-\alphah^2-\betah^2,
$$
implies that $\gammah$ is a quadratic quantity which, together with the bootstrap bounds, allows us to estimate
the cross terms.  Injecting \fref{boudprofile}, \fref{cojoeofje} into \fref{cnoeiieouo} now yields \fref{initialboundddirichletboot}.\\

{\bf step 2} Closing the $\mathcal E_4$ bound.\\

We integrate the
differential inequality  \fref{crucialboot} and use the bootstrap bound \eqref{controlinitbootinit}. This yields: $\forall t\in [0,t_1),$
\be
\label{monotnyintegree}
\mathcal E_4(t)\leq \left(\frac{\mu(t)}{\mu(0)}\right)^6\mathcal E_4(0)+(2(1-d_1)K+C)\mu^6(t)\int_0^t\frac{b}{\mu^8}\frac{b^4}{|\log b|^2}d\sigma
\ee
for some universal constants $0<d_1<1$, $C>0$ independent of $M$. Let $C_1,C_2$ be two large enough universal constants  
and define 
\be
\label{conno}
\alpha_1=2-\frac{C_1}{\sqrt{\log M}}, \ \ \alpha_2=2+\frac{C_2}{\sqrt{\log M}}
\ee  
Observe that $$b(t)>0$$ from 
\eqref{inital} and \eqref{htwobootinit}
since  $b(t_0)=0$ would imply that $u(t_0)$ is a rescaling of $Q$ which, by uniqueness, contradicts the initial data assumption. We observe from the modulation equations \eqref{roughboundparameters}, \eqref{loideb} and the bootstrap assumption  
\eqref{controlinitbootinit} that:
\bee
\frac{d}{ds}\left\{\frac{|\log b|^{\alpha_i}b}{\mu}\right\}& = & \frac{|\log b|^{\alpha_i}}{\mu}\left[\left(1-\frac{\alpha_i}{|\log b|}\right)b_s+b^2-b\left(\lsl+b\right)\right]\\
& = & \left(1-\frac{\alpha_i}{|\log b|}\right)\frac{|\log b|^{\alpha_i}}{\mu}\left[b_s+b^2\left(1+\frac{\alpha_i}{|\log b|}+O\left(\frac{1}{|\log b|^2}\right)\right)\right]\\
& &\left\{ \begin{array}{ll}\leq 0 \ \ \mbox{for} \ \ i=1\\\geq 0 \ \ \mbox{for}  \  \ i=2.
\end{array}\right .
\eee
Note that the last inequality requires the choice of $C_1,C_2\gg \sqrt K$.
Integrating this from $0$ to $t$ yields: 
\be
\label{lawintegrationone}
\frac{b(0)}{\mu(0)}\left(\frac{|\log b(0||}{|\log b(t)|}\right)^{\alpha_2}
\leq \frac{b(t)}{\mu(t)}\leq \frac{b(0)}{\mu(0)}\left(\frac{|\log b(0||}{|\log b(t)|}\right)^{\alpha_1}.
\ee
This yields in particular using the initial bound \fref{controlinit}: 
\be
\label{estationtermezero}
\left(\frac{\mu(t)}{\mu(0)}\right)^6\mathcal E_4(0)\leq (b(t)|\log b(t)|^{\alpha_2})^6\frac{\mathcal E_0}{(b(0)|\log b(0)|^{\alpha_2})^6}\leq \frac{b^4(t)}{|\log b(t)|^2}.
\ee We now compute explicitely using $b=-\mu\mu_{t}+O\left(\frac{b^2}{|\log b|}\right)$:
\bee
\int_0^t\frac{b}{\mu^8}\frac{b^4}{|\log b|^2}d\tau& = & \frac16\left[\frac{b^4}{\mu^6|\log b|^2}\right]_0^t-\frac 16\int_0^{t}\frac{b_tb^3}{\mu^6|\log b|^2}\left(4+\frac{2}{|\log b|}\right)d\tau\\
& + & O\left(\int_0^t\frac{b}{\mu^8}\frac{b^6}{|\log b|^2}d\tau\right).
\eee
Using the monotonicity bound $\lambda^2 b_t=b_s\le -b^2$ 
from \fref{loideb}  we obtain:
$$\mu^6(t)\int_0^t\frac{b}{\mu^8}\frac{b^4}{|\log b|^2}d\sigma \leq \frac12\left[1+O\left(\frac{1}{|\log b_0|}\right)\right]\frac{b^4(t)}{|\log b(t)|^2}.$$ Combiningthis together with \fref{estationtermezero} and inserting into 
\fref{monotnyintegree} yields $$\mathcal E_4(t)\leq \left((1-d_2)K+C\right)\frac{b^4(t)}{|\log b(t)|^2}$$ for some universal constants $0<d_2<1$ and $C>0$ independent of M. The desired bound \fref{controlinitboot}  follows for $K$ large enough independent of $M$.\\

{\bf step 3} Closing the $\mathcal E_2$ bound.\\

Observe that interpolating between the $\mathcal E_1$ and the $\mathcal E_4$ bounds does not give enough decay in $b$ and we need a dynamical argument. Let us come back to the original map $u$ and compute from \fref{nlsmap}, \fref{defv}:
\bea
\label{ddderiveeseconde}
\frac12\frac{d}{dt}\left\{\int|u\wedge\Delta u|^2\right\} & = & \int u\wedge\Delta u\cdot\left[(u\wedge\Delta u)\wedge \Delta u+u\wedge\Delta(u\wedge\Delta u)\right]\\
\nonumber & = & \int u\wedge\Delta u\cdot\left[u\wedge\Delta(u\wedge\Delta u)\right]\\
\nonumber & = & \frac{1}{\l^4}\int v\wedge\Delta v\cdot\left[v\wedge\Delta(v\wedge\Delta v)\right]
\eea
We now recall that for any vector $a$ with radial coordinates in the Frenet basis $$a=\alpha e_r+\beta e_\tau+\gamma Q, \ \ \Gamma=\left|\begin{array}{lll}\alpha\\\beta\\\gamma\end{array}\right ., $$ there holds
\be
\label{ceoihfeohe}
a\cdot [u\wedge \Delta a]=a\cdot\left[u\wedge(\Delta a+|\nabla Q|^2a)\right]
=-\Gamma\cdot \left[(e_z+\hbw)\wedge\H \Gamma\right]=-\Gamma\cdot \hJ\H\Gamma.
\ee We use the decomposition $$v=e_z+\hbw$$ in terms of coordinates in the Frenet basis, and apply \fref{ceoihfeohe} with $$a=v\wedge \Delta v=v\wedge(\Delta v+|\nabla Q|^2 v)\ \ \mbox{so that}\ \ \Gamma=-\hJ\H \hbw$$ to conclude from \fref{ddderiveeseconde}:
$$
\frac12\frac{d}{dt}\left\{\frac{1}{\l^2}\int|\hJ\H\hbw|^2\right\}  =  -\frac{1}{\l^4}\int \hJ\H \hbw\cdot \hJ\H (\hJ\H\hbw).
$$
We now split $\hbw=\bw+\tbw_0$ and obtain equivalently:
\bee
\frac12\frac{d}{dt}\left\{\frac{1}{\l^2}\int|\hJ\H\hbw|^2\right\}&=&-\frac{1}{\l^4}\left[\int \bw_2\cdot \hJ\H \bw_2+\int \hJ\H \tbw_0\cdot \hJ\H (\hJ\H\tbw_0)\right.\\
& + & \left.\int\bw_2\cdot \hJ\H (\hJ\H\tbw_0)+\int \hJ\H\tbw_0\cdot \hJ\H\bw_2\right].
\eee
We then estimate from Cauchy Schwarz, Lemma \ref{gainderivatives}, \fref{estimateinterm}, \fref{betterhtwoglobal}, \fref{estprofile}:
\bee
\nonumber &&\frac12\frac{d}{dt}\left\{\frac{1}{\l^2}\int|\hJ\H\hbw|^2\right\}\\
\nonumber & \lesssim & \frac{1}{\l^4}\left[\|\bw_2\|_{L^2}\|\hJ\H\bw_2\|_{L^2}+b\|\H\tbw_0\|_{L^2}^2+\|\H\hJ\bw_2\|_{L^2}\|\H\tbw_0\|_{L^2}+\|\hJ\H\bw_2\|_{L^2}\|\H\tbw_0\|_{L^2}\right]\\
\nonumber & \lesssim & \frac{1}{\l^4}\left[\sqrt{\mathcal E_2}\frac{b^2}{|\log b|^2}+b^3|\log b|^2+\left(\frac{b^4}{|\log b|^2}b^2|\log b|^2\right)^{\frac12}\right]\\&\lesssim& \frac{1}{\l^4}\left[\sqrt{\mathcal E_2}\frac{b^2}{|\log b|^2}+b^3|\log b|^2\right]\lesssim \frac{b^3|\log b|^2}{\l^4}
\eee
where we used the boostrap bound \fref{htwobootinit} in the last step. We integrate this in time using \fref{htwobound} and the initial bound \fref{htwobootinit} to derive:
$$\matchal E_2(t)\lesssim b^2(t)|\log b(t)|^2+\lambda^2(t)b_0^{10}+\lambda^2(t)\int_0^t\frac{b^3}{\lambda^4}d\tau.$$
 We then estimate using \fref{lawintegrationone} and $b^2\lesssim -b_s$:
  \bee
  \int_0^t \frac{b^3}{\lambda^4}\lesssim \frac{b^2(0)|\log\, b(0)|^{2\alpha_1}}{\lambda^2(0)} \int_0^t-\frac{b_t}{b|\log\, b|^{2\alpha_1}}\lesssim \frac{b^2(0)|\log\, b(0)|}{\lambda^2(0)}
  \eee
  where we used $\alpha_1-1>0$ for $M$ large enough from \fref{conno}. Hence from \fref{lawintegrationone}:
\bea
\label{interjpjpejep}
\mathcal E_2(t) & \lesssim & b^2(t)|\log b(t)|^2+\lambda^2(t)b_0^{10}+\lambda^2(t) \frac{b^2(0)|\log\, b(0)|}{\lambda^2(0)}\\
\nonumber & \lesssim & b^2(t)|\log b(t)|^2+b^2(0)|\log b(0)|\frac{b^2(t)}{b^2(0)|\log b(0)|^{2\alpha_2}}|\log b(t)|^{2\alpha_2}\\
\nonumber & \lesssim & b^2|\log b(t)|^5
\eea
and \fref{htwoboot} is proved.


\subsection{Modulation parameters $a,b$} \label{sec:ab}


In this section we prove \eqref{initalcontrolboot}. The bound $b(t)\leq \frac{K}{2}b^*(M)$ follows immediately from
the monotonicity $b_s\le 0$. The estimate
$$
|a(t)|\leq \frac{b(t)}{2|\log b(t)|}
$$
will be shown to hold for a special choice of the data $a(0)$, which depends on the data 
$\bw(0)$ and $b(0)$. This restriction is consistent with the statement of our main result.\\
Indeed, let us show that for a given parameter $b(0)$ and the data $\bw(0)$ satisfying the assumptions \eqref{inital control},
\eqref{initialboundddirichlet}, \eqref{controlinit} we can find a value of the parameter $a(0)$ with the property that 
$|a(0)|\le \frac{b(0)}{4|\log b(0)|}$ and such that the bound \eqref{initalcontrolboot}
$$
\forall t\in [0,t_1), \ \ |a(t)|\leq \frac{b(t)}{2|\log b(t)|}
$$
holds along the dynamics generated by the data $(a(0),b(0),w(0))$. Define
$$
\kappa(s)=a(s) \frac {|\log b(s)|}{2b(s)}.
$$
Considering the data $a(0)$ in the interval ${\mathcal I}=[-\frac {b(0)}{4|\log b(0)|},\frac {b(0)}{4|\log b(0)|})]$ ensures that 
$\kappa(0)\in [-\frac 12,\frac 12]$ and provides
the existence of $s_*=s_*(a(0))\in (0,+\infty]$ such that $|\kappa(s)|<1$ for all $0\le s< s_*$ and $|\kappa(s_*)=1$, if 
$s_*<+\infty$. Note that if $s_*(a_0)=+\infty$ then the corresponding value of $a(0)$ produces the desired conclusion.

Given the bound $|\kappa(s)|<1$ all the conclusions of Theorem \ref{propboot} as well as Lemma \ref{roughbound}
hold on the interval $[0,t_*(s_*))$. We then compute with the help of the modulation equations of Lemma \ref{roughbound}:
\bea
\label{nondgern}
\nonumber \frac d{ds} \kappa(s)&= & a_s \frac {|\log b|}{2b} - a \frac {b_s}{2b^2} - a \frac {|\log b|}{2b^2} b_s=
-a+\frac a2(1+|\log b|) +a + O\left(\frac {|a|+|b|}{\sqrt {\log M}}\right)\\ &= &b\kappa (1+o(1)) +  O\left(\frac {b}{\sqrt {\log M}}\right).
\eea
This equation and positivity of $b(s)$ implies that $\kappa(s)$ is monotonically decreasing if $\kappa(0)\le -\frac {C}{\sqrt {\log M}}$
and monotonically increasing if $\kappa(0)\ge \frac {C}{\sqrt {\log M}}$ for some universal constant $C$. 
We now define two subintervals ${\mathcal I}_+$ and ${\mathcal I}_-$ of ${\mathcal I}=[-\frac 12,\frac 12]$ with the 
property that for $\kappa(0)\in {\mathcal I}_\pm$ there exists a finite value $s_*$ such that $|\kappa(s)|<1$ for all
$0\le s<s_*$ and $\kappa(s_*)=\pm 1$.\\
If ${\mathcal I}_+$ is empty\footnote{Note that we can easily show in fact that $\pm\frac12\in \mathcal I_\pm$} then any $\kappa(0)\ge \frac {C}{\sqrt {\log M}}$ yields the claim, similarily for ${\mathcal I}_-$. We therefore assume that both ${\mathcal I}_\pm$ are non empty. Then the nondegeneracy \fref{nondgern} ensures that $\pm\frac{d}{ds}\kappa(s^*)>0$ for in $\mathcal I_\pm$, and thus $\mathcal I_\pm$ are open non empty disjoint subsets
$[-\frac 12,\frac 12]$. As a consequence, there exists at least 
one value of $\kappa(0)\in [-\frac 12,\frac 12]$ such that $\kappa(0)\not\in ({\mathcal I}_+\cup {\mathcal I}_-)$.
For this value of $\kappa(0)$ we have $|\kappa(s)|<1$ for all $0\le s<\infty$, as desired.


\section{Sharp description of the singularity formation}


We are now in position to conclude the proof of Theorem \ref{thmmain}.

\subsection{Finite time blow up}
Let $T\leq +\infty$ be the life span  of $u$, then the estimates of Proposition \ref{propboot} hold on $[0,T)$. From \fref{lawintegrationone}, $$-\frac{d}{dt}\sqrt\lambda=-\frac{1}{2\l\sqrt\l}\lsl\gtrsim \frac{b}{\l \sqrt{\lambda}}\gtrsim C(u_0)>0$$ and thus $\lambda(t)$ becomes equal to zero in finite time which implies $$T<+\infty.$$ Observe then from \fref{lawintegrationone} that this implies 
\be
\label{boudnaryb}
\lambda(T)=b(T)=0.
\ee
Integrating the modulation equation \eqref{loideb} from $s=\infty$ we obtain 
\be\label{eq:firstb}
b(s)=\frac 1s + O(\frac {1}{s\log s}).
\ee

\subsection{Asymptotics} 


We now prove convergence of the phase $\Theta(t)$ as $t\to T$, give the 
sharp description of the blow up speed and prove convergence of the energy excess.\\

\noindent
{\bf step 1} Refined bound for $a$.\\

The bound \fref{initalcontrolboot} is not sufficient to prove the convergence of the phase $\Theta$ since $$\int_0^{s}\frac{b}{|\log b|}\sim |\log b(s)|\to +\infty \ \ \mbox{as} \ \ s\to +\infty.$$ We claim however that the global $\mathcal E_4$ bound now allows for an 
additional logarithmic gain:
\be
\label{improvdedo}
|a(t)|\lesssim C(\delta_0) \frac{b(t)}{|\log b(t)|^{\frac 32}}.
\ee
for some small enough universal constant $\delta_0>0$ and a large constant $C(\delta_0)$. This, together with 
\eqref{roughboundparameters}, implies tzhat 
$$
|\Theta_s|\lesssim C(\delta_0) \frac{b(s)}{|\log b(s)|^{\frac 32}}
$$
and, after application of the bound \eqref{eq:firstb}, leads to the convergence of the phase $\Theta$:
\be
\label{nbhohoe}
\Theta(t)\to \Theta(T)\in \Bbb R\ \ \mbox{as} \ \ t\to T.
\ee
To show \eqref{improvdedo}, let $$B_{\delta}=\frac{1}{b^{\delta}}$$ for some sufficiently small $\delta>0$. We project \fref{equationdefvectorial} onto the first component $\alpha$, commute the equation with $H$ and take the inner product with $\chi_{B_\delta}\Lambda \phi$. As in the proof of \fref{coehoheh} we obtain:
 $$(H\tt_1,\chi_{B_{\delta}}\Lambda \phi)=4\log \Bd\left(1+O\left(\frac{1}{|\log b|}\right)\right).$$ 
 As in the proof of \fref{fluwnumbertwo} we estimate from \fref{htwocomputone}:
\be
\label{fluxoneff}
 -(H\tilde{\Psi}_0^{(1)},\chi_{B_{\delta}}\Lambda \phi) =O\left(\frac{b^2}{|\log b|^2}{\log \Bd}\right).
\ee
Moreover, $$|(H^2\alpha,\chi_{B_{\delta}}\Lambda\phi)|\lesssim \sqrt{\mathcal E_4}\left(\int_{y\leq B_{\delta}}|\Lambda \phi|^2\right)^{\frac12}\lesssim \frac{b^2}{\sqrt{|\log b|}}.$$ All other nonlinear terms are easily estimated using the $\mathcal E_4$ bound and the smallness of $\delta>0$, leading to the bound:
\bee
a_s & = & -\frac{1}{(H \tt_1,\chi_{B_{\delta}}\Lambda\phi)}(\pa_s H\alpha,\chi_{B_\delta}\Lambda \phi)+O\left(C(\delta) \frac{b^2}{|\log b|^{\frac 32}}\right)+O\left(b b^{-C\delta}\frac{b^2}{|\log b|}\right)\\
& = & -\frac{d}{ds}\left\{\frac{1}{(H \tt_1,\chi_{B_{\delta}}\Lambda\phi)}(H\alpha,\chi_{B_\delta}\Lambda \phi)\right\}+O\left(C(\delta) \frac{b^2}{|\log b|^{\frac 32}}\right)+O\left(b^{-C\delta} \frac{b^3}{|\log b|}\right).
\eee
Let 
$$\tilde{a}=a+\frac{1}{(H \tt_1,\chi_{B_{\delta}}\Lambda\phi)}(H\alpha,\chi_{B_\delta}\Lambda \phi)=
a+O\left(b^{-C\delta} \frac{b^2}{|\log b|^{2}}\right),$$ and thus: $$\tilde{a}_s=O\left(\frac{b^2}{|\log b|^{\frac 32}}\right).$$ 
We now integrate this identity in time using \fref{eq:firstb} to and the convergence $\tilde a\to 0$ for $s\to\infty$ from \fref{eq:firstb}, \fref{improvdedo} to conclude:
$$|\tilde{a}|\lesssim \frac{b}{|\log b|^{\frac 32}},$$ and \fref{improvdedo} is proved.\\

\noindent
{\bf step 2} Derivation of the blow up speed.\\

Arguing as for $a$ we now slightly refine our control for $b$. We project \fref{equationdefvectorial} onto the second component $\beta$, commute with $H$ and take the inner product with $\chi_{B_\delta}\Lambda \phi$. We use
\bea
\label{fluxone}
\nonumber -(H\tilde{\Psi}_0^{(2)},\chi_{B_\delta}\Lambda \phi) =4c_bb^2\log B_{\delta}\left(1+O\left(\frac{1}{\log B_\delta}\right)\right)
\eea
to get:
\bee
&&b_s+b^2\left(1+\frac{2}{|\log b|}\right)\\
& = &   -\frac{1}{(H\tt_1,\chi_{B_{\delta}}\Lambda\phi)}(\pa_sH\beta,\chi_{B_\delta}\Lambda \phi)+O\left(C(\delta) \frac{b^2}{|\log b|^{\frac 32}}\right)+O\left(b^{-C\delta} \frac{b^3}{|\log b|}\right)\\
& = & -\frac{d}{ds}\left\{\frac{1}{(H\tt_1,\chi_{B_{\delta}}\Lambda\phi)}(H\beta,\chi_{B_\delta}\Lambda \phi)\right\}
+O\left(C(\delta) \frac{b^2}{|\log b|^{\frac 32}}\right)+O\left(b^{-C\delta} \frac{b^3}{|\log b|}\right).
\eee
We now let 
\be
\label{lienbbtilde}
\tilde{b}=b+\frac{1}{(H\tt_1,\chi_{B_{\delta}}\Lambda\phi)}(H\beta,\chi_{B_\delta}\Lambda \phi)=b+O\left(\frac{b^2}{|\log b|^{2}}b^{-C\delta}\right)
\ee
and obtain the pointwise refined control:$$\left|\tilde{b}_s+\bt^2\left(1+\frac{2}{|\log \tilde{b}|}\right)\right|\lesssim \frac{\tilde{b}^2}{|\log \tilde{b}|^{\frac 32}}.$$ Equivalently, $$\frac{\bt_s}{\bt^2\left(1+\frac{2}{|\log \bt|}\right)}+1=O\left(\frac{1}{|\log \bt|^{\frac 32}}\right).$$ 
We now integrate this in time using $\lim_{s\to +\infty}\bt(s)=0$ from \fref{boudnaryb} and get:
$$\bt(s)=\frac1s-\frac{2}{s\log s}+O\left(\frac{1}{s|\log s|^{\frac 32}}\right)$$ and thus from \fref{lienbbtilde}: 
\be
\label{esitmateforbbis}  
b(s)=\frac1s-\frac{2}{s\log s}+O\left(\frac{1}{s|\log s|^{\frac 32}}\right).
\ee We now use the modulation equation \fref{roughboundparameters} to conclude: 
\be
\label{esitmateforb}
-\lsl=\frac1s-\frac{2}{s\log s}+O\left(\frac{1}{s|\log s|^{\frac 32}}\right).
\ee We rewrite this as $$\left|\frac{d}{ds}\log\left(\frac{s\lambda(s)}{(\log s)^2}\right)\right|\lesssim \frac{1}{s|\log s|^{\frac 32}}$$ and thus integrating in time yields the existence of a constant $\kappa(u_0)>0$, dependent on the data $u_0$, such that: 
$$\log \frac{s\lambda(s)}{(\log s)^2}=\frac{1}{\kappa(u_0)}\left[1+O\left(\frac{1}{|\log s|^{\frac 12}}\right)\right].
$$
Therefore, 
$$-\log \lambda=\log s\left[1-2\frac{\log \log s}{\log s}+\frac 1{\kappa(u_0)} O\left(\frac 1{|\log s|^{\frac 32}}\right)\right]$$ 
and hence $$\frac{1}{s}=\kappa'(u_0)\frac{\lambda}{|\log \lambda|^2}\left(1+o(1)\right)$$ 
for some constant $\kappa'(u_0)$. The identity \eqref{esitmateforbbis}  now implies
$$
b=\kappa'(u_0) \frac{\lambda}{|\log \lambda|^2}\left(1+o(1)\right),
$$
which in turn means that 
$$-\lambda\lambda_t=\kappa'(u_0)\frac{\lambda}{|\log \lambda|^2}\left(1+o(1)\right)$$ and thus $$-|\log \lambda|^2\lambda_t=\kappa'(u_0)(1+o(1)).$$ Integrating from $t$ to $T$ with $\lambda(T)=0$ yields 
$$\lambda(t)=\kappa'(u_0)\frac{T-t}{|\log (T-t)|^2}\left[1+o(1)\right].$$ 
Observe in particular the control 
\be
\label{rationbl}
\frac{b}{\lambda}=\frac{k'(u_0)}{|\log (T-t)|^2}(1+o(1)).
\ee
{\bf step 3} Strong convergence of the excess of energy.\\

We now turn to the proof of \fref{strongvonegrebe}. We recall the decomposition:
$$u(t,x)=e^{\Theta(t)R}(Q+\hbw)_\l=e^{\Theta(t)R}(Q)_\l+\tilde{u}.$$ 
The $H^1$ bound $$\|\nabla \tilde{u}\|_{L^2}\lesssim 1$$ is a simple consequence of the orbital stability bound and the energy critical scaling invariance. We now claim the $\dot{H}^2$ bound: 
\be
\label{htwoboundbis}
\|\Delta \tilde{u}\|_{L^2}\lesssim 1.
\ee
Assume \fref{htwobound}, then a simple localization argument using \fref{nlsmap} and the bound \fref{htwoboundbis} yields the strong convergence outside the blow up point: 
$$\forall R>0, \ \  \nabla u\to \nabla u^*\ \ \mbox{as} \ \ t\to T\ \ \mbox{in} \ \ L^2(|x|\geq R)$$
or equivalently:
\be
\label{neoveoeho}
\forall R>0, \ \  \nabla\tilde{u}\to \nabla u^*\ \ \mbox{as} \ \ t\to T\ \ \mbox{in} \ \ L^2(|x|\geq R)
\ee for some equivariant map $\nabla u^*\in L^2$.
The convergence \fref{strongvonegrebe} now follows from $$E(u_0)=E(Q)+E(u^*)$$ which is the consequence of the conservation of energy, the outer convergence  \fref{neoveoeho} and the local compactness of Sobolev embedding through the $\dot{H}^2$ bound \fref{htwoboundbis}. The convergence \fref{strongvonegrebe} and the bound \fref{htwoboundbis} now yield the regularity \fref{htworegulairity}.\\
{\it Proof of \ref{htwoboundbis}}: Let us recall by definition that $$\tilde{u}=\mathcal S\hbw$$ and thus from \fref{ecomptuaiowquat} and the decay $|\nabla Q|\lesssim 1+Z$:
\bee
\lambda^2\int|\Delta \tilde{u}|^2& \lesssim & \int\frac{|\hbw|^2}{1+y^8}+\int|\hJ \H\hbw|^2\\
& \lesssim & \matchal E_2+b^2|\log b|^2
\eee
where we used \fref{cneohoheo}, \fref{estun}, \fref{htwobound} and \fref{estimateinterm}. We now recall \fref{interjpjpejep} which together with \fref{rationbl} yields: 
$$\matchal E_2+b^2|\log b|^2\lesssim \l^2$$ and \fref{htwoboundbis} follows.\\
This concludes the proof of Theorem \ref{thmmain}.


\section*{Appendix A: $L^2$ coercivity estimates}


This Appendix is devoted to the derivation of $L^2$ weighted coercivity estimates for the operator $H$ and its iterate $H^2$, which generalize related results in \cite{RodSter}, \cite{RaphRod}.

\subsection{Hardy inequalities}

  \begin{lemma}[Logarithmic Hardy inequalities]
\label{lemmaloghrdy}
$\forall R>2$, $\forall v\in\dot{H}^1_{rad}(\RR^2)$ and $\gamma> 0$, there holds the following controls:
\be
\label{harfylog}
 \int_{y\le R} \frac{|v|^2}{y^2(1+|\log y|)^2}\lesssim  \int_{1\leq y\leq 2} |v|^2+
 \int_{y\le R} |\nabla v|^2,
\ee
\bea
\label{harfylog'}
&&\frac{\gamma^2}{4} \int_{1\le y\le R} \frac{|v|^2}{y^{2+\gamma}(1+|\log y|)^2}\\
\nonumber & \leq &  C_\gamma\int_{1\leq y\leq 2} |v|^2+
 \int_{1\le y\le R} \frac{|\nabla v|^2}{y^\gamma(1+|\log y|)^2},
\eea
\be
\label{estmedium}
|v|^2_{L^{\infty}(1\leq y\leq R)}\lesssim \int_{1\leq y\leq 2}|v|^2+R^2\int_{1\le y\le R} \frac{|\nabla v|^2}{y^2},
\ee
\be
\label{hardyone}
\int_{y\leq R}  |v|^2\lesssim R^2\left(\int_{y\leq 2}|v|^2+\log R\int_{y\leq R}|\nabla v|^2\right),
\ee
\be
\label{harybis}
\int_{R\leq y\leq2 R}  \frac{|v|^2}{y^2}\lesssim \int_{y\leq 2}|v|^2+\log R\int_{y\leq 2R}|\nabla v|^2.
\ee
If $\int_{y\leq 1}\frac{|v|^2}{y^2}<+\infty,$
 then:
\be
\label{harybis-non}
\int_{y\leq2 R}  \frac{|v|^2}{y^2}\lesssim \log R \int_{y\leq 2}|v|^2+(\log R)^2\int_{y\leq 2R}|\nabla v|^2.
\ee
\end{lemma}

\begin{proof} Let $v$ smooth and radially symmetric. First recall from the one dimensional Sobolev embedding $H^1(1\leq y\leq 2)\hookrightarrow L^{\infty}(1\leq y\leq 2)$ that $$|v(1)|^2\lesssim \int_{1\leq y\leq 2}(|v|^2+|\pa_yv|^2).$$ Let $f(y)=-\frac{{\bf e}_y}{y(1+|\log y|)}$ so that $$\nabla \cdot f=\left|\begin{array}{ll}\frac{1}{y^2(1+|\log y|)^2} \ \ \mbox{for} \ \ y\geq 1\\  -\frac{1}{y^2(1+|\log y|)^2} \ \ \mbox{for} \ \ y\leq 1\end{array}\right.$$ and integrate by parts to get: 
\bee
\nonumber & & \int_{\e\leq y\le R} \frac{|v|^2}{y^2(1+|\log y|)^2} =-\int_{\e\le y\le 1}|v|^2\nabla \cdot f +\int_{1\le y\le R}|v|^2\nabla \cdot f 
\\
 & \lesssim &- \left[\frac{|v|^2}{1+|\log (y)|}\right]_{1}^R+ \left[\frac{|v|^2}{1+|\log (y)|}\right]_{\e}^1+\int_{y\le R} \left|v\partial_y v \frac{1}{y(1+|\log y|)}\right|\\
& \lesssim & |v(1)|^2+\left(\int_{\e\leq y\le R} \frac{|v|^2}{y^2(1+|\log y|)^2}\right)^{\frac{1}{2}}\left(\int_{\e\leq y\le R} |\nabla v|^2\right)^{\frac{1}{2}},
\eee
and \fref{harfylog} follows. To prove \fref{harfylog'}, let $\gamma>0$ and $$f(y)=-\frac{{\bf e}_y}{y^{\gamma+1}(1+\log y)^2}$$ so that for $y\geq 1$:
$$ \nabla\cdot f= \frac{1}{y^{\gamma+2}(1+\log y)^2}\left[\gamma+\frac{2}{(1+\log y)^3}\right]\geq \frac{\gamma}{y^{\gamma+2}(1+\log y)^2}.$$ We then integrate by parts to get: 
\bee
\nonumber & & \gamma\int_{1\le y\le R} \frac{|v|^2}{y^{\gamma+2}(1+\log y)^2}  \leq \int_{1\le y\le R}|v|^2\nabla \cdot f\\
\nonumber & \le &- \left[\frac{|v|^2}{y^\gamma(1+\log (y))^2}\right]_{1}^R + 2\int_{1\le y\le R} \frac{|v\partial_y v |}{y^{\gamma+1}
(1+\log y)^2}\\
& \leq & C\int_{1\le y\le 2}|v|^2+2\left(\int_{1\le y\le R} \frac{|v|^2}{y^{\gamma+2}(1+\log y)^2}\right)^{\frac{1}{2}}\left(\int_{1\le y\le R} \frac{|\nabla v|^2}{y^{\gamma}(1+\log y)^2}\right)^{\frac{1}{2}}
\eee
and  \fref{harfylog'} follows. To prove \fref{estmedium}, we have: $\forall y\in [1,R]$, $$|v(y)|=\left|v(1)+\int_{1}^yv'(r)dr\right|\lesssim |v(1)|+R\left(\int_{1\leq y\leq R} \frac{|\nabla v|^2}{y^2}\right)^{\frac{1}{2}},$$ and \fref{estmedium} follows. Similarily, $$|v(y)|=\left|v(1)+\int_{1}^yv'(r)dr\right|\lesssim |v(1)|+\left(\int_{y\leq R} |\nabla v|^2\right)^{\frac{1}{2}}\sqrt{\log R},$$ and \fref{hardyone}, \fref{harybis} follow by squaring this estimate and integrating in $R$. Finally,
\eqref{harybis-non} follows from \fref{harybis} by summing over dyadic $R$-intervals.
\end{proof}


\subsection{Sub-coercivity estimates}


In this section we establish weighted sub-coercive estimates for the operators $H$ and $H^2$ which will play a key role 
in the proof of the coercive estimates under additional orthogonality conditions.
\begin{lemma}[Sub-coercivity for $H$]
\label{subcolemmaa}
Let $u$ be a function with the property 
\be
\label{assumption}
\int \frac{|u|^2}{y^4Å(1+|\log y|)^2}+\int |\pa_y (Au)|^2+\int\frac{|Au|^2}{y^2(1+y^2)}<+\infty
\ee
then
\bea
\label{keyestnatappendix}
& &  \int (Hu)^2\approx \int\frac{|Au|^2}{y^2(1+y^2)}+\int|\pa_y (Au)|^2\\
 \nonumber & \gtrsim&  \int\frac{|\pa^2_y u|^2}{(1+|\log y|)^2} +\int\frac{|\pa_y u|^2}{y^2(1+|\log y|)^2}+ \int \frac{|u|^2}{y^4(1+|\log y|)^2}\\
 & - & C\int \frac{|u|^2}{1+y^5}.
\eea
\end{lemma}
 {\it Proof of \fref{keyestnatappendix}}: First observe from \fref{defhtilde} that:
$$\int (Hu)^2=\int(A^*Au)^2=(Au,\tilde{H}(Au))\approx \int|\pa_yAu|^2+\int\frac{|Au|^2}{y^2(1+y^2)}.$$ Let now a smooth cut off function $\chi(y)=1$ for $y\leq 1$, $\chi(y)=0$ for $y\geq 2$, and consider the decomposition: $$u=u_1+u_2=\chi u+(1-\chi) u.$$ Then from \fref{harfylog}:
\be
\label{fihoehyeog}
  \int\frac{|Au|^2}{y^2(1+y^2)}+\int|\pa_y (Au)|^2\gtrsim\left[ \int \frac{|Au|^2}{y^2(1+y^2)}+\frac{|Au|^2}{y^2(1+|\log y|)^2}\right].
  \ee
   For the first term, we rewrite:
\bee
\int \frac{|Au|^2}{y^2(1+y^2)}& \geq& \int \frac{|Au_1|^2}{y^2}+2\int \frac{(Au_1)(Au_2)}{y^2(1+y^2)}\\
& \gtrsim & \left[\int \frac{|y\pa_y\left(\frac{u_1}{y}\right)|^2}{y^2}-\int\frac{|Z-1|^2}{y^2}|u_1|^2-C\int_{1\leq y \leq 2}|u|^2\right]
\eee
where in the last step we integrated by parts the quantity: $$(Au_1)(Au_2)=(\chi Au-\chi'u)((1-\chi)Au+\chi'u)\geq \chi (Au) \chi'u-\chi'u(1-\chi)(Au)-(\chi')^2u^2.$$ We hence conclude from $|Z(y)-1|\lesssim y$ for $y\leq 1$ and the Hardy inequality \fref{harfylog} applied to $\frac{u_1}{ y}\in H^1_{rad}$ fom \fref{assumption} that: 
\be
\label{firstestimate}
\int \frac{|Au|^2}{y^2(1+y^2)}\gtrsim \left[\int\frac{|u_1|^2}{y^4(1+|\log y|)^2}-C\int_{y\leq 2}|u|^2\right].
\ee
Similarily we estimate:
\bea
\label{estprlimemiddle}
&&\nonumber \int \frac{|Au|^2}{y^2(1+|\log y|)^2} \geq \int \frac{|Au_2|^2}{y^2(1+|\log y|)^2}+2\int \frac{(Au_1)(Au_2)}{y^2(1+|\log y|)^2}\\
\nonumber & \gtrsim &\left[\int \frac{1}{y^2(1+|\log y|)^2}|\pa_yu_2+\frac{u_2}{y}|^2-C\int\frac{|1+Z|^2}{y^2(1+|\log y|)^2}|u_2|^2-C\int_{1\leq y \leq 2}|u|^2\right]\\
& \gtrsim & \left[\int \frac{|\pa_yu_2|^2}{y^2(1+|\log y|)^2}+\int\frac{|u_2|^2}{y^4(1+|\log y|)^2}-C\int \frac{|u_2|^2}{y^6(1+|\log y|)^2}
\right]
\eea
where we applied the weighted Hardy \fref{harfylog'} to $yu_2$ with $\gamma=4$
and integrated by parts for the last step using the bound $|1+Z(y)|\lesssim \frac{1}{y^2}$ for $y\geq 1$. \fref{fihoehyeog}, \fref{firstestimate} and \fref{estprlimemiddle} imply:
\be
\label{estprlimeojeoi}
\int\frac{|Au|^2}{y^2(1+y^2)}+\int|\pa_y (Au)|^2\gtrsim\left[\int \frac{|u|^2}{y^4(1+|\log y|)^2}-C\int \frac{|u|^2}{1+y^5}\right].
\ee
This implies using again \fref{harfylog}:
\bea
\label{pouetpoute}
\nonumber \int\frac{|\pa_y u|^2}{y^2(1+|\log y|)^2} & \lesssim & \int\frac{|A u|^2}{y^2(1+|\log y|)^2}+\int\frac{|u|^2}{y^4(1+|\log y|)^2}\\
& \lesssim & \int |\pa_y (Au)|^2+\int\frac{|Au|^2}{y^2(1+y^2)}+\int \frac{|u|^2}{1+y^5}.
\eea
Finally, examining the expression 
$$
\pa_y (A u)=\pa_y (-\pa_y u+\frac {Z}y u)
$$
we also obtain
$$
\int\frac{|\pa^2_y u|^2}{(1+|\log y|)^2} 
 \lesssim  \int |\pa_y (Au)|^2+\int\frac{|Au|^2}{y^2(1+y^2)}+\int \frac{|u|^2}{1+y^5}
 $$
which together with \fref{estprlimeojeoi}, \fref{pouetpoute} concludes the proof of \fref{keyestnatappendix} and of Lemma \ref{subcolemmaa}.

\begin{lemma}[Weighted sub-coercivity for $H$]
\label{lem:h3}
Let $u$ be a function with the property 
\bea
\label{hypeohj}
\nonumber &&\int \frac{|H u|^2}{y^4(1+|\log y|)^2}+ \int \frac{|\pa_yH u|^2}{y^2(1+|\log y|)^2}\\
& + & \int \frac{|u|^2}{y^4(1+y^4)(1+|\log y|)^2}+  \int \frac{(\pa_y u)^2}{y^2(1+y^4)(1+|\log y|)^2}<\infty,
\eea
then
\bea
\label{neoieoeuo}
&&\int \frac{|H u|^2}{y^4(1+|\log y|)^2}+\int \frac{|\pa_y H u|^2}{y^2(1+|\log y|)^2}\\
\nonumber &  \gtrsim & \int \frac{|u|^2}{y^4(1+y^4)(1+|\log y|)^2}+ \int \frac{(\pa_y u)^2}{y^2(1+y^4)(1+|\log y|)^2}\\
\nonumber & + & \int \frac{|\pa^2_y u|^2}{y^4(1+|\log y|)^2} + \int \frac{(\pa_y^3 u)^2}{y^2(1+|\log y|)^2}\\
\nonumber & - & C\left[ \int \frac{(\pa_y u)^2}{1+y^8}+ \int\frac{|u|^2}{1+y^{10}} \right].
\eea
\end{lemma}

{\it Proof of Lemma \ref{lem:h3}}: Let $\chi(y)$ be a smooth cut-off function with support in $y\ge 1$ and equal to 1 for $y\ge 2$. We first consider 
\bee
&&\int \chi \frac{|H u|^2}{y^4(1+|\log y|)^2} = \int \chi \frac{|- \pa_y (y\pa_y u) + \frac Vy u|^2}{y^6(1+|\log y|)^2}\\ 
&= & \int \chi \frac{|\pa_y (y\pa_y u)|^2}{y^6(1+|\log y|)^2} -2  \int \chi \frac{\pa_y (y\pa_y u)\cdot V u}{y^7(1+|\log y|)^2}+ \int \chi \frac{V^2 |u|^2}{y^8(1+|\log y|)^2}\\ 
 &= & \int \chi \frac{|\pa_y (y\pa_y u)|^2}{y^6(1+|\log y|)^2} +
 2  \int \chi \frac{ V (\pa_y u)^2}{y^6(1+|\log y|)^2}+ \int  \chi \frac{V^2 |u|^2}{y^8(1+|\log y|)^2}\\ 
 &-& 
 \int |u|^2 \Delta\left( \frac{\chi V }{y^6(1+|\log y|)^2}\right) 
\eee
We now observe that for $k\ge 0$ and $y\ge 1$
$$
\pa_y^k V(y)= \pa_y^k (1)+ O(y^{-2-k})
$$
We may thus apply twice the Hardy inequality with sharp constant \fref{keyestnatappendix} with $\gamma=6$ and get for a sufficiently large universal constant $R$:
\bee
&&\int \chi \frac{|\pa_y (y\pa_y u)|^2}{y^6(1+|\log y|)^2} +2  \int \chi \frac{ V (\pa_y u)^2}{y^6(1+|\log y|)^2}-\int |u|^2 \Delta\left( \frac{\chi V }{y^6(1+|\log y|)^2}\right)\\
& \geq & (9+2)\int _{y\geq R}\frac{ (\pa_y u)^2}{y^6(1+|\log y|)^2}-31\int_{y\ge R} \frac{ |u|^2}{y^8(1+|\log y|)^2}-C\int_{y\geq 1}\frac{(\pa_y u)^2}{y^8}\\
& \geq & (99-31)\int_{y\geq R} \frac{|u|^2}{y^8(1+|\log y|)^2} -C\left[\int_{y\ge 1} \frac{(\pa_y u)^2}{y^8}+
 \int_{y\ge 1} \frac{|u|^2}{y^{10}}\right]
 \eee
and hence the bound away from the origin:
\begin{align*}
\int \chi \frac{|H u|^2}{y^4(1+|\log y|)^2}& \gtrsim \int_{y\ge 2} \frac{|\pa^2_y u|^2}{y^4(1+|\log y|)^2} + 
\int_{y\ge 2} \frac{(\pa_y u)^2}{y^6(1+|\log y|)^2}
+ \int_{y\ge 2} \frac{|u|^2}{y^8(1+|\log y|)^2}\\& - C\left[\int_{y\ge 1} \frac{(\pa_y u)^2}{y^8}+
 \int_{y\ge 1} \frac{|u|^2}{y^{10}}\right].
\end{align*}
The control of the third derivative away from the origin follows from: 
\bee
&&\int \chi \frac{|\pa_y H u|^2}{y^2(1+|\log y|)^2}=\int \chi\frac{|\pa_y \left(\frac 1y(-\pa_y (y\pa_y u) +\frac V{y^2})u\right)|^2}{y^2(1+|\log y|)^2}\\
 &\gtrsim&  \int \chi\frac{|\pa^3_y u|^2}{y^2(1+|\log y|)^2}-C\left[ \int \chi \frac{|\pa^2_y u|^2}{y^4(1+|\log y|)^2}+
 \int \chi \frac{|\pa_y u|^2}{y^6(1+|\log y|)^2} + \int \chi \frac{|u|^2}{y^8(1+|\log y|)^2}\right].
\eee
Near the origin, we first observe from \fref{hypeohj} that 
\be
\label{bouadnry}
\int\frac{|Au|^2}{y^2(1+|\log y|^2)}<+\infty.
\ee
We now observe from $$\pa_y(\log (\Lambda \phi))=\frac{Z}{y}$$ that $$A^*f=\pa_yf+\frac{1+Z}{y}f=\frac{1}{y\Lambda \phi}\pa_y(y\Lambda \phi f),$$ and thus from \fref{bouadnry}: 
\be
\label{inversona}
Au(y)=\frac{1}{y\Lambda \phi(y)}\int_0^y \tau\Lambda\phi(\tau)Hu(\tau)d\tau.
\ee
 We then estimate from Cauchy-Schwarz and Fubini:
\bee
&&\int_{y\le 1}\frac{|Au|^2}{y^5(1+|\log y|^2)}dy  \lesssim \int_{0\le y\leq 1}\int_{0\le\tau\le y}\frac{y^5}{y^9(1+|\log y|^2)}|Hu(\tau)|^2dyd\tau\\
& \lesssim & \int_{0\le \tau\le 1}|Hu(\tau)|^2\left[\int_{\tau\leq y\leq 1}\frac{dy}{y^4(1+|\log y|^2)} \right]d\tau\lesssim  \int_{\tau\le 1}\frac{|Hu(\tau)|^2}{\tau^3(1+|\log \tau|^2)} d\tau
\eee
and thus: 
\be
\label{cnoceooehoe}
\int_{y\le 1}\frac{|Au|^2}{y^6(1+|\log y|^2)}\lesssim  \int_{y\le 1}\frac{|Hu|^2}{y^4(1+|\log y|^2)}
\ee
which implies from $A^*(Au)=Hu$:
\be
\label{cndkohoheohe}
\int_{y\le 1}\frac{|\pa_yAu|^2}{y^4(1+|\log y|^2)}\lesssim  \int_{y\le 1}\frac{|Hu|^2}{y^4(1+|\log y|^2)}.
\ee
We now rewrite near the origin: $$Hu=-\pa^2_yu+\frac{1}{y}\left(-\pa_y u+\frac{u}{y}\right)+\frac{V-1}{y^2}u=-\pa_y^2u+\frac{Au}{y}+\frac{(V-1)+(1-Z)}{y^2}u$$ which implies using \fref{cnoceooehoe}, \fref{cndkohoheohe}, 
\be
\label{cneokoehoefh}
\int_{y\leq 1}\frac{|\pa^2_yu|^2}{y^4(1+|\log y|^2)}\lesssim \int_{y\le 1}\frac{|Hu|^2}{y^4(1+|\log y|^2)}+\int_{y\le 1}\frac{|u|^2}{y^4(1+|\log y|^2)},
\ee
\bea
\label{vnovheohohohr}
\nonumber \int_{y\leq 1}\frac{|\pa^3_yu|^2}{y^2(1+|\log y|^2)}&\lesssim& \int_{y\le 1}\frac{|\pa_yHu|^2}{y^2(1+|\log y|^2)}+ \int_{y\le 1}\frac{|Hu|^2}{y^4(1+|\log y|^2)}\\
& + & \int_{y\le 1}\frac{|\pa_yu|^2}{y^2(1+|\log y|^2)}+\int_{y\le 1}\frac{|u|^2}{y^4(1+|\log y|^2)}.
\eea
Let now $\zeta=(1-\chi)^{\frac 12}$, then $\zeta u$ satisfies \fref{assumption} from \fref{hypeohj}, \fref{cnoceooehoe}, \fref{cndkohoheohe}, and we thus obtain from \eqref{keyestnatappendix}:
\bea
\label{cnoveoeoe}
\nonumber && \int \zeta^2 \frac{|H u|^2}{y^4(1+|\log y|)^2}\ge  \int \zeta^2 {|H u|^2}\gtrsim \int {|H\zeta u|^2} -
C\int_{1\le y\le 2} (|\pa_y u|^2+|u|^2)\\
\nonumber  &\gtrsim & \int\frac{|\pa_y (\zeta u)|^2}{y^2(1+|\log y|)^2}+ \int \frac{|\zeta u|^2}{y^4(1+|\log y|)^2}-C\int \frac{|\zeta u|^2}{1+y^5}\\
  &\gtrsim& \int_{y\leq 1}\frac{|\pa_y u|^2}{y^2(1+|\log y|)^2}+ \int_{y\le 1}  \frac{|u|^2}{y^4(1+|\log y|)^2}-C
\int_{y\le 2} \frac{|u|^2}{1+y^5}.
\eea
Injecting \fref{cneokoehoefh}, \fref{vnovheohohohr} into \fref{cnoveoeoe} yields the expected control at the origin. This concludes the proof of \fref{neoieoeuo} and Lemma \ref{lem:h3}.\\

We combine the results of Lemma \ref{subcolemmaa} and Lemma \ref{lem:h3} and obtain:
\begin{lemma}[Sub-coercivity for $H^2$]
Let $u$ be a radially symmetric function with
 \bea
 \label{conditionu}
 &&\int {|\pa_y A H u|^2}+\int \frac{|AHu|^2}{y^2(1+y^2)}+\int \frac{|H u|^2}{y^4(1+|\log y|)^2}\\
 \nonumber & + &  \int\frac{|u|^2}{y^4(1+y^4)(1+|\log y|^2)}+\int \frac{(\pa_y u)^2}{y^2(1+y^4)(1+|\log y|)^2}<\infty
\eea
then
\bea
\label{vnoeoijeoj}
&&\int {|H^2 u|^2}\gtrsim  \int \frac{|H u|^2}{y^4(1+|\log y|)^2}+\frac{|\pa_y H u|^2}{y^2(1+|\log y|)^2}+ \int \frac{|\pa^3_y u|^2}{y^2(1+|\log y|)^2}\\
\nonumber & +& \int \frac{|\pa^2_y u|^2}{y^4(1+|\log y|)^2}+
\int \frac{|\pa_y u|^2}{y^2(1+y^4(1+|\log y|)^2}
+ \int \frac{|u|^2}{y^4(1+y^4)(1+|\log y|)^2}\\ 
\nonumber &-& C\left[\int \frac{|H u|^2}{1+y^5}+ \int \frac{(\pa_y u)^2}{1+y^8}+
 \int\frac{|u|^2}{1+y^{10}} \right]
\eea
\end{lemma}


\subsection{Coercivity of $H^2$}


We are now in position to derive the fundamental coercivity property of $H^2$ at the heart of our analysis:

\begin{lemma}[Coercivity of $H^2$]
\label{lemmehardyhsquare}
Let $M\geq 1$ be  a large enough universal constant. Let $\Phi_M$ be given by \fref{defphim}
. Then there exists a universal constant $C(M)>0$ such that for all radially symmetric function $u$ satisfiyng \fref{conditionu} and the orthogonality conditions
$$
(u,\Phi_M)=0,\qquad (Hu,\Phi_M)=0,
$$
there holds:
\bea
&&\int \frac{|H u|^2}{y^4(1+|\log y|)^2}+ \frac{|\pa_y H u|^2}{y^2(1+|\log y|)^2}+\int
\frac{|\pa^4_y u|^2}{(1+|\log y|)^2}+\int \frac{|\pa^3_y u|^2}{y^2(1+|\log y|)^2}\nonumber\\
&+&\int \frac{|\pa_y^2 u|^2}{y^4(1+|\log y|)^2}+\int \frac{|\pa_y u|^2}{y^2(1+y^4)(1+|\log y|)^2}
+ \int \frac{|u|^2}{y^4(1+y^4)(1+|\log y|)^2}\nonumber\\
 &\le & C(M) \int |H^2u|^2.
 \label{H2b}
\eea
\end{lemma}

{\it Proof of Lemma \ref{lemmehardyhsquare}}. We argue by contradiction. Let $M>0$ fixed and consider a normalized sequence $u_n$  
\bea
\label{normlaization}
&&\int \frac{|H u_n|^2}{y^4(1+|\log y|)^2}+ \frac{|\pa_y H u_n|^2}{y^2(1+|\log y|)^2}+\int \frac{|\pa^3_y u_n|^2}{y^2(1+|\log y|)^2}\\
\nonumber &+&\int \frac{|\pa_y^2 u_n|^2}{y^4(1+|\log y|)^2}+\int \frac{|\pa_y u_n|^2}{y^2(1+y^4)(1+|\log y|)^2}
+ \int \frac{|u_n|^2}{y^4(1+y^4)(1+|\log y|)^2}=1, 
\eea
satisfying the orthogonality conditions
\be\label{eq:orthog} 
(u_n,\Phi_M)=0, \qquad (u_n,H\Phi_M)=0,
\ee and 
\be
\label{seeumporgpo}
 \int |H^2u_n|^2\approx  \int\frac{|A Hu_n|^2}{y^2(1+y^2)}+\int|\nabla (A H u_n)|^2\leq \frac{1}{n},
  \ee
The normalization condition  implies that the sequence $u_n$ is uniformly
bounded in $H^1_{loc}$. Moreover, as follows from Lemma \ref{lem:h3}, 
for any smooth cut-off function $\zeta$ vanishing in a neighborhood of  $y=0$ 
the sequence $\zeta u_n$ is uniformly bounded in $H^3_{loc}$. As a consequence, we can assume that $u_n$ and 
$\zeta u_n$ weakly converge in $H^1_{loc}$ and $H^3_{loc}$ to $u_\infty$ and $\zeta u_\infty$ respectively. 
Moreover, $u_\infty$ satisfies the equation
$$
AHu_\infty=0
$$
away from $y=0$. Integrating the ODE 
$$
AHu_\infty=-\Lambda \phi\pa_y\left(\frac{Hu_\infty}{\Lambda\phi}\right)=0
$$
we obtain that $Hu_\infty(y)=\alpha \Lambda \phi(y)$ away from $y=0$. The function $u_\infty$ can be written in the form
\begin{align*}
u_\infty (y)&=\alpha \Gamma(y)\int_0^y  \Lambda \phi(x) \Lambda\phi(x) xdx - \alpha \Lambda\phi(y)\int_1^y
 \Lambda \phi(x) \Gamma(x) xdx \\ &+ \beta \Lambda \phi(y) + \gamma \Gamma(y)= \alpha T_1(y)+ 
 \beta \Lambda \phi(y) + \gamma \Gamma(y)
\end{align*}
Using the condition $u_\infty\in H^1_{loc}$ we can conclude that $\gamma=0$. Passing to the limit in the orthogonality conditions, using that $u_n$ converges to $u_\infty$ weakly in 
$H^1_{loc}$, we conclude that $u_\infty$ satisfies 
$$
(u_\infty,\Phi_M)=0,\qquad (u_\infty,H \Phi_M)=0.
$$ 
We may therefore determine the constants $\alpha, \beta$ using \fref{defphim}, \fref{estunphim} which yield $\alpha=\beta=0$ and thus $u_\infty=0.$
 
 The sub-coercitivity bound \fref{vnoeoijeoj} together with \fref{seeumporgpo} ensures:
 \bee
  \frac 1n&\gtrsim  & \int (H^2u_n)^2\gtrsim \int\frac{|\pa_y H u_n|^2}{y^2(1+|\log y|)^2}+ \int \frac{|H u_n|^2}{y^4(1+|\log y|)^2}+\int \frac{|\pa^3_y u_n|^2}{y^2(1+|\log y|)^2}\\
&+&\int \frac{|\pa_y^2 u_n|^2}{y^4(1+|\log y|)^2}+\int \frac{|\pa_y u_n|^2}{y^2(1+y^4)(1+|\log y|)^2}
+ \int \frac{|u_n|^2}{y^4(1+y^4)(1+|\log y|)^2}\\
& -&  C\left [\int \frac{|H u|^2}{1+y^5}+\int \frac{(\pa_y u_n)^2}{1+y^8}+\int\frac{|u_n|^2}{1+y^{10}}\right] 
 \eee
 Coupling this 
  with the normalization condition we obtain that  
$$
\int \zeta \left [\frac{|H u_n|^2}{1+y^5}+\int \frac{(\pa_y u_n)^2}{1+y^8}+
 \int\frac{|u_n|^2}{1+y^{10}}\right]\ge c
 $$
 for some positive constant $c>0$ and a smooth cut-off function $\zeta$ vanishing 
 for $y<\epsilon$ and $y>\epsilon^{-1}$. The size of $\epsilon$ depends only on the universal constant $C$. 
 Since $u_n$ weakly converges to $u_\infty$ in $H^3$ on any compact subinterval of $y\in (0,\infty)$ we can 
 pass to the limit 
  to conclude 
  $$
\int  \zeta \left [\frac{|H u_\infty|^2}{1+y^5}+\int \frac{(\pa_y u_\infty)^2}{(1+y^8)}
+ \int \frac{|u_\infty|^2}{(1+y^{10})}\right]\ge c.
 $$
 This contradicts the established identity $u_\infty\equiv 0$ and concludes the proof of Lemma \ref{lemmehardyhsquare}
 modulo the additional bound for 
 $$
 \int \frac{|\pa^4_y u|^2}{(1+|\log y|^2)}
 $$ 
 claimed in \eqref{H2b}. To control
 this term we simply observe that   for $y\ge 1$ 
$$
H^2 u =H(-\Delta +\frac {V}{y^2}) u = \pa_y^4 u +
O\left (\sum_{i=0}^3 \frac {|\pa_y^i u|}{y^{4-i}}\right).
$$
and the desired estimate easily follows from the already established bounds for lower derivatives.
For $y\le 1$ we write
$$
H^2 u =H(-\Delta +\frac {V}{y^2}) u = \pa_y^4 u + \pa_y^2(\frac 1y \pa_y - \frac V{y^2})u+ O\left(\frac {|Hu|}{y^2}\right) +
 O\left(\frac {|\pa_y Hu|}{y}\right).
$$
We further note that 
\begin{align*}
 \pa_y^2(\frac 1y \pa_y - \frac V{y^2})u &= \pa^2_y \left (\frac 1 y  Au\right) + \pa_y^2 \left (O(1) u\right) =
 \frac 1y \pa_y^2 (Au) -\frac 2y \pa_y (Au) +\frac 4{y^3} Au  + \pa_y^2 \left (O(1) u\right) \\ &= 
 O\left(\frac {|\pa_y Hu|}{y}\right)+ O\left(\frac {|Hu|}{y^2}\right)+ O\left(\frac {|Au|}{y^3}\right)+ \pa_y^2 \left (O(1) u\right).
\end{align*}
The estimate for $y\le 1$ now follows from the bounds for $\pa_y Hu$, $Hu$ and the coercivity estimate  \eqref{cnoceooehoe}
for $Au$.

\subsection{Coercivity of $H$}
We complement the coercitivity property of the operator $H^2$, established in the previous section,  
by the corresponding statement for the operator $H$, which follows from standard compactness argument. 
A complete proof\footnote{In \cite{RaphRod} the argument was carried out with the orthogonality condition 
$(u,\chi_M \Lambda\phi)=0$. Here we require that $(u,\Phi_M)=0$, which is sufficient in view of \eqref{estunphim},
according to which
$(\chi_{2M} \Lambda\phi, \Phi_M)=(\chi_M\Lambda \phi,\Lambda\phi)=4\log M(1+o_{M\to +\infty}(1))$.} is given in \cite{RaphRod}:
\begin{lemma}[Coercivity of $H$]
\label{lemmahardy1}
Let $M\geq 1$ fixed. Then there exists $c(M)>0$ such that the following holds true. Let $u\in H^1$ with $$(u,\Phi_M)=0$$  and 
\be
\label{ionoeghe}
\int \frac{|u|^2}{y^4(1+|\log y|)^2}+\int |\pa_y (Au)|^2<+\infty,
\ee
then:
\bea
\label{estdeux}
& & \nonumber 
\int_{|y|\ge 1} \frac {\pa_y^2 u}{1+|\log y|^2}+\int \frac{|\pa_yu|^2}{y^2(1+|\log y|)^2}+ \int \frac{|u|^2}{y^4(1+|\log y|)^2} \\
\nonumber & \leq & c(M)\left[\int\frac{|Au|^2}{y^2(1+y^2)}+\int|\pa_y (Au)|^2\right]\\
& \lesssim & c(M)\int|H u|^2.
\eea
\end{lemma}


\section*{Appendix B: Interpolation estimates}


We derive interpolation bounds in the bootstrap regime of Proposition \ref{propboot} and in the regime of parameters described by Remark \ref{remarkparmaeters}. We recall the notation: $$\bw=\left|\begin{array}{lll}\alpha\\\beta\\\gamma\end{array}\right .=\bw^\perp+\gamma e_z$$
and the norms $\mathcal E_1, \mathcal E_2, \mathcal E_4$, introduced in \eqref{defefour}, together with their bootstrap  bounds 
\eqref{initialboundddirichletbootinit}, \eqref{htwobootinit}, \eqref{controlinitbootinit} :
\begin{align*}
&\mathcal E_1=\int (|\nabla \bw|^2+\frac{|\bw|^2}{y^2})\leq {K}\delta(b^*),\\
&\mathcal E_2=\int |R_z\H\bw^\perp|^2=\int (|H\alpha|^2+|H\beta|^2)\le K b^2(t)|\log b(t)|^5,\\ 
&\mathcal E_4=\int|(R_z\H)^2 \bw^\perp|^2=\int (|H^2\alpha|^2+|H^2\beta|^2)\leq K\frac{b^4(t)}{|\log b(t)|^2}.
\end{align*}


\subsection{Regularity at the origin}
The use of the coercivity bounds in Appendix A and the interpolation estimates below requires establishing 
a priori regularity of the Schr\"odinger map $u$, expressed in Frenet basis, at the origin. We will show that 
the smoothness of the map $u(t): {\Bbb R}^2\to {\Bbb S}^2\subset {\Bbb R}^3$ implies boundedness of the
quantities
$$
\left |\frac{\bw} y\right |, \ \ \left |{\pa_y \bw} \right |,\ \ \left |\frac{A\bw^\perp} y\right |, \ \ \left |\frac{H \bw^\perp} y\right |, \ \
\left |\frac{AH \bw^\perp} y\right |, \ \ \left |{\pa_y A H\bw^\perp} \right |
$$  
We first consider the expression $\nabla u = (\pa_y u, \frac 1y \pa_\theta u)$ which, as long as $u$ is smooth, is 
bounded at the origin. Using Lemma \ref{lemmafrenet} we compute this in terms of $(\alphah,\betah,\gammah)$ 
Frenet coordinates of $u$:
$$
\left( (\pa_y \alphah+(1+Z) (1+\gammah)) e_r+\pa_y \betah e_\tau+(\pa_y\gammah-(1+Z)\alphah) Q, -\frac Zy \betah e_r+
(\frac Zy \alphah+\frac {\Lambda\phi} y (1+\gammah)) e_\tau - \frac {\Lambda\phi}y \betah Q\right).
$$
This immediately implies boundedness of 
$$
\left |\frac{\bw} y\right |, \ \ \left |{\pa_y \bw} \right |
$$
Similarly, computing the expression $ \Delta u+|\nabla Q|^2u$ from \eqref{ecomptuaiowquat}
\begin{align*}
 \Delta u+|\nabla Q|^2v & =  \left(-H\alphah+2(1+Z)\pa_y\gammah\right)e_r+(-H\betah)\et\\
\nonumber& + \left(-\frac{2Z(1+Z)}{y}\alphah+\Delta \gammah-2(1+Z)\pa_y\alphah\right) Q
\end{align*}
gives us the boundedness of $|H\bw^\perp|$ and $|\Delta\gamma|$.

Using cartesian coordinates $x=(x_1,x_2)$ on ${\Bbb R}^2$ we examine the expressions
$$
\lim_{x_2=0, x_1\to 0} \nabla_{x_1} u=\lim_{x_1=0, x_2\to 0} \nabla_{x_1} u.
$$
Computing this in polar coordinates and relative to the Frenet basis we immediately obtain
the relations
$$
\pa_y\alphah=\frac 1y \alphah,\quad \pa_y\betah=\frac 1y \betah
$$
in the limit as $y\to 0$. This immediately gives the boundedness of 
$$
\left |\frac{A\alpha} y\right |, \ \ \left |\frac{A\beta}y \right |.
$$
The remaining bounds can be shown by similar arguments. We omit the details.


\subsection{Interpolation bounds for $\bw$}


We now turn to the proof of interpolation estimates for $\bw$ in the bootstrap regimes which are used all along the proof of Proposition \ref{propboot}.

\begin{lemma}[Interpolation estimates for $\bw^\perp$]
\label{interpolationw}
There holds: 
\bea
\label{estun}
& & \int \frac{|\bw^\perp|^2}{y^4(1+y^4)(1+|\log y|^2)}+\int \frac{|\pa^i_y\bw^\perp|^2}{y^2(1+y^{6-2i})(1+|\log y|^2)}\\
& & \lesssim C(M)\mathcal E_4,\ \ 1\leq i\leq 3,
\eea
\be
\label{estund}
\int_{|y|\ge 1} \frac{|\pa^i_y\bw^\perp|^2}{(1+y^{4-2i})(1+|\log y|^2)} \lesssim C(M)\mathcal E_2,\ \ 1\leq i\leq 2,
\ee
\be
\label{lossyboundwperp}
\int_{y\geq 1} \frac{1+|\log y|^C}{y^2(1+|\log y|^2)(1+y^{6-2i})}|\pa_y^i\bw^\perp|^2 \lesssim b^4|\log b|^{C_1(C)}, \ \ 0\leq i\leq 3,
\ee
\be
\label{inteproloatedbound}
\int_{y\geq 1} \frac{1+|\log y|^C}{y^2(1+|\log y|^2)(1+y^{4-2i})}|\pa_y^i\bw^\perp|^2 \lesssim b^3|\log b|^{C_1(C)}, \ \ 0\leq i\leq 2,
\ee
\be \label{3rd}
\int_{|y|\geq 1} |\pa_y \H\bw^\perp|^2\lesssim b^3 |\log b|^6. 
\ee
\be
\label{estperplinfty}
\|\bw^\perp\|_{L^{\infty}}\lesssim \delta(b^*),
\ee
\be
\label{estinterm}
\|\A\bw^\perp\|^2_{L^{\infty}}\lesssim b^2|\log b|^9,
\ee
\be\label{eq:anno}
\int_{y\le 1} \frac{|\A\bw^\perp|^2}{y^6(1+|\log y|^2)} \lesssim C(M) \mathcal E_4,
\ee
\be
\label{estoriiginagian}
\left\|\frac{\A \bw^\perp}{y^2(1+|\log y|)}\right\|_{L^{\infty}(y\leq 1)}^2+\left\|\frac{\Delta \A \bw^\perp}{1+|\log y|}\right\|^2_{L^{\infty}(y\leq 1)}+\left\|\frac{\H \bw^\perp}{y(1+|\log y|)}\right\|_{L^{\infty}(y\leq 1)}^2\lesssim b^4,
\ee
\be
\label{choeouefeouie}
\left\|\frac{|H \alpha|+|H\beta|}{y(1+|\log y|)}\right\|_{L^{\infty}(y\leq 1)}^2\lesssim b^4,
\ee
\be
\label{estlinftydeux}
\|\frac{\bw^\perp}{y}\|^2_{L^{\infty}(y\leq 1)}+\|{\pa_y\bw^{\perp}}\|^2_{L^{\infty}(y\leq 1)}\lesssim b^4,
\ee
\be
\label{estlinftydeuxfa}
\|\frac{\bw^\perp}{y}\|^2_{L^{\infty}(y\geq 1)}+\|\pa_y\bw^{\perp}\|^2_{L^{\infty}(y\geq 1)}\lesssim b^2|\log b|^8,
\ee
\be
\label{estlinftysecond}
\|\frac{\bw^\perp}{1+y^2}\|^2_{L^{\infty}}+\|\frac{\pa_y\bw^{\perp}}{1+y}\|^2_{L^{\infty}}+\|\pa_{yy}\bw^\perp\|^2_{L^{\infty}(y\geq 1)}\lesssim C(M)b^3|\log b|^2.
\ee

\end{lemma}

{\it Proof of Lemma \ref{interpolationw}}: The estimate \fref{estperplinfty} follows from the $\mathcal E_1$ bound: 
$$\|\bw^\perp\|_{L^{\infty}}^2\lesssim \|\pa_y \bw^\perp\|^2_{L^2}+\|\frac{\bw^\perp}{y}\|^2_{L^2}\lesssim \delta(b^*).$$ 

The estimate \fref{estun} follows from $H^2$ coercivity of Lemma \ref{lemmehardyhsquare} and the $\mathcal E_4$ bound.
For $i=0,1,2$  \fref{lossyboundwperp} is easily implied by Lemma \ref{lemmahardy1} and the $\mathcal E_2$ bound. For
$i=3$ 
we split the integral at $y= B_0^{20}$ and estimate the inner contribution using  Lemma \ref{lemmehardyhsquare} and the 
$\mathcal E_4$ bound, and the outer by interpolating between the $\mathcal E_4$ and $\mathcal E_2$ bounds:
 \bee
&&\int_{y\geq 1} \frac{(1+|\log y|^C)|\pa_y^3 \bw^\perp|^2}{y^2(1+y^4)(1+|\log y|^2)}\\
 &\leq & \int_{1\leq y\leq B_0^{20}} \frac{(1+|\log y|^C)|\pa_y^3\bw^\perp|^2}{y^2(1+y^4)(1+|\log y|^2)}+
 \int_{y\geq B_0^{20}}\frac{(1+|\log y|^C)|\pa_y^3 \bw^\perp|^2}{y^2(1+y^4)(1+|\log y|^2)}\\
& \lesssim & |\log b|^{C+2}\mathcal E_4+B_0^{-20} (\log B_0)^C \mathcal E_2^{\frac 12} \mathcal E_4^{\frac 12}.
\eee 
Similarly, the estimate \fref{inteproloatedbound} follows by splitting the integral at $y=B_0$ 
and using the $\mathcal E_4$ and $\mathcal E_2$ bounds for the inner and outer regions respectively.\\
To obtain \eqref{3rd} we write 
$$
\pa_y H\alpha = -A H\alpha + \frac {Z} y H\alpha,
$$
which, using the $\mathcal E_2$ and $\mathcal E_4$ bounds,  implies with $B\in [B_0,2B_0]$ 
\bee
&&\int_{|y|\ge 1} |\pa_y H\alpha|^2 \lesssim \int_{1\le |y|\le B} |\pa_y H\alpha|^2+  \int_{|y|\ge B} \left (|AH\alpha|^2+
\frac {|H\alpha|^2}{y^2} \right)\\ 
& \lesssim & b^3 |\log b|^2 +   \int_{|y|\ge B} |A^*AH\alpha\, H\alpha| + |y| |AH\alpha\, H\alpha|_{|y|=B}
+b^3 |\log b|^6\\ 
& \lesssim & b^3 |\log b|^6+  \int_{|y|\ge B_0} |H^2\alpha\, H\alpha| + \frac 1{B_0}
\int_{B_0}^{2B_0}|AH\alpha\, H\alpha|\\
& \lesssim & b^3 |\log b|^6+  \left (\int_{|y|\ge B_0} |H^2\alpha|^2\right)^{\frac 12}\left(\int_{|y|\ge B_0} |H\alpha|^2\right)^{\frac 12}+ 
\left(\int_{B_0}^{2B_0} \frac{|AH\alpha|^2}{y^2}\right)^{\frac 12} \left(\int_{B_0}^{2B_0}|H\alpha|^2\right)^{\frac 12}\\ 
& \lesssim & b^3 |\log b|^6.
\eee
To prove \eqref{estlinftydeux} and \eqref{estlinftydeuxfa} 
let $a\in [1,2]$ such that $$|\pa_y\bw^\perp(a)|^2\lesssim \int_{1\leq y\leq 2}|\pa_y\bw|^2\le
C(M) \mathcal E_4,$$ then for $y\leq 1$: $$|\pa_y\bw^\perp|\lesssim |\pa_y\bw^\perp(a)|+\int_y^a|\pa_{yy}\bw^\perp|dy\le
C(M) \sqrt{\mathcal E_4}\lesssim b^2,$$ 
and for $y\geq 1$:
 \bee
\left\|\pa_y\bw^\perp\right\|^2_{L^\infty(y\geq 1)} & \lesssim & \left(\int_{y\geq 1} \frac{|\pa_y\bw^\perp |^2}{y^2}\right)^\frac12\left(\int_{y\geq 1} |\pa_{yy}\bw^{\perp}|^2\right)^{\frac12}\\
& \lesssim & \int_{y\geq 1} \frac{|\pa_y\bw^\perp |^2}{y^2}+\int_{y\geq 1} \frac{|\bw^\perp |^2}{y^4}+\int |A^*A\bw^\perp|^2\\
& \lesssim & b^2 |\log b|^8,
\eee
where in the last step we split the integral at $y=B_0^2$ and use $\mathcal E_2$ bound for the inner and 
$\mathcal E_1$ for the outer parts.
Next, we have from $\bw^\perp(0)=0$:
$$\left\|\frac{\bw^\perp}{y}\right\|_{L^{\infty}(y\leq 1)}\lesssim \|{\pa_y\bw^\perp}\|_{L^{\infty}(y\leq 1)}$$ and 
\be
\label{cnoohoeekmvodjvpd}
\left\|\frac{\bw^\perp}{y}\right\|^2_{L^{\infty}(y\geq 1)}\lesssim 
\int_{y\geq 1} \frac{|\pa_y\bw^\perp |^2}{y^2}+\int_{y\geq 1} \frac{|\bw^\perp |^2}{y^4}.
\ee
The estimates \eqref{estlinftydeux} and \eqref{estlinftydeuxfa} now easily follow.\\
The estimate \eqref{eq:anno} follows directly from \eqref{cnoceooehoe}.\\
For \fref{estinterm}: 
$$\|A\alpha\|_{L^{\infty}}^2\lesssim \left(\int \frac{|A\alpha|^2}{y^2}\right)^{\frac12}\left(\int|\pa_yA\alpha|^2\right)^{\frac12}
\lesssim b^2|\log b|^9,$$ 
where in the last step we used the coercivity of $H$ of Lemma \ref{lemmahardy1}, split the first integral at $y=B_0^2$
and used the $\mathcal E_2$ bound for inner and the $\mathcal E_1$ for the outer parts.\\
For \fref{estoriiginagian}, we recall from \fref{inversona}:
$$H\alpha=A^*A\alpha \ \ \mbox{and thus}\ \ A\alpha=\frac{1}{y\Lambda\phi}\int_0^y z \Lambda\phi (H\alpha) dz.$$ This yields for $y\leq 1$: 
$$|A\alpha(y)|\lesssim \frac{1}{y^2}\left(\int \frac{|H\alpha|^2}{y^4(1+|\log y|^2)}\right)^{\frac12}\left(\int_0^y(1+|\log z|^2)z^3z^4dz\right)^{\frac12}\lesssim y^2(1+|\log y|)\sqrt{\mathcal E_4}.$$ Similarily, $$AH\alpha=\frac{1}{y\Lambda\phi}\int_0^y z\Lambda\phi (H^2\alpha) dz$$ yields for $y\leq 1$: 
\be
\label{estahhahf}
|\tilde{H}A\alpha(y)|=|AH\alpha(y)|\lesssim \frac{1}{y^2}\left(\int |H^2(\alpha)|^2\right)^{\frac 12}\left(\int_0^y \frac{z^4}{z}dz\right)^{\frac 12}\lesssim \sqrt{\mathcal E_4}.
\ee This implies for $y\leq 1$: 
$$|\Delta A(\alpha)(y)|\lesssim |\tilde{H}(A\alpha)(y)|+\frac{|A\alpha(y)|}{y^2}\lesssim \sqrt{\mathcal E_4}(1+|\log y|).$$ 
Similar estimates can be shown for $\beta$ and the last term in \fref{estoriiginagian}. 
Let now $a\in [1,2]$ be such that $$|H\alpha|^2(a)\lesssim \int_{1\leq y\leq 2}|H\alpha|^2\lesssim \mathcal E_4,$$ then from 
\be
\label{integrationfrinuala}
Af=-\Lambda \phi\partial_y\left(\frac{f}{\Lambda \phi}\right), 
\ee
 there holds for $y\leq 1$: $$|H\alpha(y)|\leq\Lambda\phi(y)\left[ \frac {|H\alpha|(a)}{\Lambda\phi(a)}+\int_y^a\frac{|AH\alpha|}{\Lambda\phi}dz\right]\lesssim |y|(1+|\log y|)\sqrt{\mathcal E_4},$$ 
 where in the last step we used the coercivity of $A^*$.
 The bound  \fref{choeouefeouie} follows.\\
For \eqref{estlinftysecond}:
\bee
\left\|\frac{\bw^\perp}{1+y^2}\right\|^2_{L^{\infty}} &\lesssim & \left(\int \frac{|\bw^\perp|^2}{y^2(1+y^4)}\right)^{\frac 12}\left(\int \left|\pa_y\left(\frac{\bw^\perp}{1+y^2}\right)\right|^2\right)^{\frac 12}\\
& \lesssim & \int \frac{|\bw^\perp|^2}{y^2(1+y^4)}+\int \frac{|\pa_y\bw^\perp|^2}{1+y^4}
\eee
Now 
\bee
\int \frac{|\bw^\perp|^2}{y^2(1+y^4)}& \lesssim & \frac{|\log b|^2}{b}\int_{y\leq B_0} \frac{|\bw^\perp|^2}{y^4(1+y^4)(1+|\log y|^2)}+
\|{\bw^\perp}\|^2_{L^{\infty}}\int_{y\geq B_0}\frac{1}{1+y^8}\\
& \lesssim & b^3|\log b|^2.
\eee
The argument for the other terms is similar, and \fref{estlinftysecond} is proved.

\begin{lemma}[Interpolation bound for $\gamma=w^z$]
\label{interpolationwz}
There holds:
\bea
\label{cneohoheo}
& & \int \frac{|\gamma|^2}{y^6(1+y^2)(1+|\log y|^2)}+\int \frac{|\pa_y\gamma|^2}{y^4(1+y^{4-2i})(1+|\log y|^2)}+\int \frac{|\pa^i_y\gamma|^2}{y^2(1+y^{6-2i})(1+|\log y|^2)}\nonumber\\
& & \lesssim \delta(b^*)\left(\frac{b^4}{|\log b|^2}+\mathcal E_4\right),\ \ 2\leq i\leq 3, \label{estungamma}
\eea
\be
\label{lossyboundwperpgamma}
\int_{y\geq 1} \frac{1+|\log y|^C}{y^4(1+|\log y|^2)(1+y^{4-2i})}|\pa_y^i\gamma|^2 \lesssim b^4|\log b|^{C_1(C)}, \ \ 0\leq i\leq 2,
\ee
\be
\label{lossyboundwperpgammainteprolated}
\int_{y\geq 1} \frac{1+|\log y|^C}{y^{6-2i}(1+|\log y|^2)}|\pa_y^i\gamma|^2 \lesssim b^3|\log b|^{C_1(C)}, \ \ 0\leq i\leq 2,
\ee
\be\label{eq:Ay} 
\int\frac{|A\pa_y\gamma|^2}{y^4(1+|\log y|^2)}\le \delta(b^*)\left(\frac {b^4}{|\log b|^2} +\mathcal E_4\right),
\ee
\be
\label{estperplinftygamma}
\|\gamma \frac {1+|y|}{|y|}\|_{L^{\infty}}\lesssim \delta(b^*),
\ee
\be
\label{estlinftydeuxgamma}
\|\frac{(1+|y|)\gamma}{y^2}\|^2_{L^{\infty}}+\|\pa_y\gamma\|^2_{L^{\infty}}\lesssim b^2|\log b|^8,
\ee
\be
\label{estlinftysecondgamma}
\|\frac{\gamma}{|y|(1+|y|)}\|^2_{L^{\infty}}+\|\frac{\pa_y\gamma}{|y|}\|^2_{L^{\infty}}\lesssim C(M)b^3|\log b|^2,
\ee
\be
\label{estgammalapla}
\int|\Delta \gamma|^2\lesssim\delta(b^*)\mathcal E_2+b^2|\log b|^2
\ee
\be
\label{estgammalinft}
\|\Delta \gamma\|^2_{L^{\infty}(y\geq 1)}\lesssim b^3|\log b|^8,
\ee
\be
\label{interpolationbylapocia}
\int|\bw^\perp|^2|\Delta^2\gamma|^2+\int_{y\ge 1}
|\Delta^2\gamma|^2\lesssim \delta(b^*)\left(\mathcal E_4+\frac{b^4}{|\log b|^2}\right).
\ee
\end{lemma}

{\bf Proof of Lemma \ref{interpolationwz}}
Recall the normalization relation: $$(1+\gammah)^2+\alphah^2+\betah^2=1.$$ We expand to get:
\be
\label{fromulagamma}
2\gamma=-\left[(2\tgamma_0+\tbeta_0^2)+\talpha_0^2+\alpha(\alpha+2\talpha_0)+\beta(\beta+\talpha_0)\right]-(\gamma+\tgamma_0)^2.
\ee Note also by construction \fref{defgammaab} that the leading order $(2\tgamma_0+\tbeta_0^2)$ contribution to $\gamma$ is cancelled on $y\leq 2B_1$,\\
 and we obtain:
\bee
\int \frac{|\gamma|^2}{y^6(1+y^2)(1+|\log y|^2)} & \lesssim &\left(\left |{\bw^\perp}\frac {1+|y|}{|y|}\right |^2_{L^{\infty}}+
\left |{\tbw_0^\perp}\frac {1+|y|}{|y|}\right |^2_{L^\infty}\right )\int \frac{|\bw^\perp|^2}{y^4(1+y^4)(1+|\log y|^2)}\\ &+&O\left(\frac{b^4}{|\log b|^4}\right)
\lesssim \delta(b^*)\left(\frac{b^4}{|\log b|^2}+\mathcal E_4\right)
\eee
from \eqref{estperplinfty} and \eqref{estlinftydeux} for $\bw^\perp$ and the corresponding bounds for $\tbw_0$.
The remaining estimates in \eqref{estungamma},  \eqref{lossyboundwperpgamma}, \eqref{lossyboundwperpgammainteprolated},
\eqref{estperplinftygamma}, \eqref{estlinftydeuxgamma} and \eqref{estlinftysecondgamma}
are obtained similarily from \fref{fromulagamma} and the corresponding statements for $\bw^\perp$. We omit the details.\\
The estimate \eqref{eq:Ay} for $y\geq 1$ follows from \fref{interpolationwz} and the challenge here is the behavior of the integrand at the origin. To treat we write
$$
A \pa_y \gamma = -\pa_y^2 \gamma + \frac Zy \pa_y \gamma.
$$
Therefore for $y\leq 1$:
\begin{align*}
|A\pa_y\gamma|&\lesssim |A \pa_y(\tilde\alpha^2_0)| + |\bw^\perp| (|\pa_y^2\bw^\perp|+|\pa_y^2\tbw_0|) +
  |\pa_y^2\bw^\perp| (|\bw^\perp|+|\tbw_0|)\\& +\left( |\pa_y \bw^\perp| + \frac {|\bw^\perp|}{|y|}\right)\left(|\pa_y \tbw_0|+
  \frac{|\tbw_0|}{|y|}\right) + |\pa_y \bw^\perp|\, |A\bw^\perp|. 
\end{align*}
The most important aspect of this formula is the last term containing $A\bw^\perp$ and providing the necessary 
vanishing at the origin from Lemma \ref{interpolationw}:
\bee
\int_{y\leq 1}\frac{|A\pa_y\gamma|^2}{y^4(1+|\log y|^2))} & \le &  \frac {b^4}{|\log b|^4} + \delta(b^*) \mathcal E_4+
\int_{y\leq 1}\frac{|\pa_y \bw^\perp|^2\, |A\bw^\perp|^2}{y^4(1+y^2)(1+|\log y|^2)}\\
& \lesssim & \frac{b^4}{|\log b|^2}.
\eee
We now turn to the proof of \fref{estgammalapla}. We first estimate in brute force from \fref{roughboundone}, \fref{roughboundonebis} and the relation $HT_1=\Lambda Q$:
\bea
\label{estimateinterm}
\nonumber \int|\H \tbw_0|^2 & \lesssim & \int\left[|H\alphat_0|^2+|H\betat_0|^2+\int|\Delta \gammat_0|^2+\frac{|\pa_y\gammat_0|^2+|\pa_y\alphat_0|^2+\frac{|\alphat_0|^2}{y^2}}{1+y^4}\right]\\
& \lesssim & b^2+\int_{y\leq 2B_1}\left|\frac{b}{1+y}+\frac{b^3 y^3}{(1+y^2)|\log b|}\right|^2\lesssim b^2|\log b|^2,
\eea
and similarily:
$$ \int\frac{| \tbw_0|^2}{1+y^4}  \lesssim  b^2+\int_{y\leq 2B_1}\left|\frac{b|\log y|}{1+y}+\frac{b^3 y^3}{(1+y^2)|\log b|}\right|^2\lesssim b^2|\log b|^3.$$
We now use this and the estimates of Lemma \ref{interpolationw} and the properties of the profile $\tbw_0$ to estimate:\\
\bee
&&\int |\Delta \gamma|^2  \lesssim b^2|\log b|^2\\
& + & \int|\bw^\perp|^2(|\Delta \bw^\perp|^2+|\Delta \tbw_0|^2)+\int|\Delta \bw^\perp|^2(|\bw^\perp|^2+|\tbw_0|^2)+\int |\nabla \bw^\perp|^2(|\nabla\bw^\perp|^2+|\nabla\tbw_0|^2)
\eee
We then estimate from \fref{estlinftydeuxfa}, \fref{estimateinterm}:
\bee
\int|\bw^\perp|^2|\Delta \tbw_0|^2 &\lesssim& b^2|\log b|^2+\int|\bw^\perp|^2\frac{|\tbw_0|^2}{y^4}\lesssim b^2|\log b|^2+\left\|\frac{\bw^\perp}{1+y}\right\|_{L^{\infty}}^2\int\frac{|\tbw_0|^2}{1+y^2}\\
& \lesssim & b^2|\log b|^2,
\eee
$$\int |\nabla \bw^\perp|^2|\nabla\tbw_0|^2\lesssim \|\pa_y\bw^\perp\|_{L^{\infty}}^2\sqrt{b}|\log b|^C\lesssim b^2|\log b|^2,$$
and using \fref{estund}:
$$
\int|\Delta\bw^\perp|^2|\tbw_0|^2 \lesssim \sqrt{b}b^2|\log b|^C\lesssim b^2|\log b|^2,
$$
and thus arrive to the bound:
$$\int |\Delta \gamma|^2  \lesssim b^2|\log b|^2+\int|\Delta \bw^\perp|^2|\bw^\perp|^2+ \int |\nabla \bw^\perp|^4.$$
From \eqref{estun} we easily see that 
$$
\int_{y\le 1}|\Delta \bw^\perp|^2|\bw^\perp|^2+ \int_{y\le 1} |\nabla \bw^\perp|^4\le C(M) \mathcal E_4\le
b^4.
$$
For $y\ge 1$ we write
$$
|\Delta\bw^\perp|^2\lesssim |H\alpha|^2+|H\beta|^2 + \frac {|\bw^\perp|^2}{y^4} ,\qquad  |\pa_y \bw^\perp|^4\lesssim |A\bw^\perp|^4+
 \frac {|\bw^\perp|^4}{y^4}
$$
and estimate 
\bea
\label{cneofheoheohe}
\nonumber \int_{y\ge 1}|\Delta \bw^\perp|^2|\bw^\perp|^2+ \int_{y\le 1} |\nabla \bw^\perp|^4&\lesssim &
\|\bw^\perp\|_{L^\infty(y\ge 1)}^2 \int_{y\ge 1} |H\bw^\perp|^2 + \int_{y\ge 1}  |A\bw^\perp|^4 + \int_{y\ge 1}
\frac {|\bw^\perp|^4}{y^4}\\ 
 &\lesssim & \delta(b^*)\mathcal E_2+ \int_{y\ge 1}\frac {|\bw^\perp|^4}{y^4}
\eea
where we used \eqref{estperplinfty}  and the Gagliardo-Nirenberg inequality
$$
 \int  |A\bw^\perp|^4 \lesssim  \int |\pa_y A\bw^\perp|^2 \int  |A\bw^\perp|^2\lesssim  \delta(b^*)\left[(\tilde{H}A\alpha,\alpha)+(\tilde{H}A\beta,\beta)\right]=\delta(b^*)\mathcal E_2.
$$
Hence \fref{estgammalapla}  follows from:
\be
\label{newesitmate}
\int_{y\geq 1} \frac{|\bw^\perp|^4}{y^4}\lesssim b^2+\delta(b^*)\mathcal E_2.
\ee
Indeed, let a cut-off function with $\psi(y)=0$ for $y\leq 1$ and $\psi(y)=1$ for $y\geq 2$. We compute: 
\bee
\int \psi \frac{\alpha^4}{y^4}& = & -\frac{1}{2}\int \psi |\alpha|^4\partial_y\left(\frac{1}{y^2}\right) dy=\frac{1}{2}\int\frac{1}{y^3}\left[|\alpha|^4\partial_y\psi+4\psi\alpha^3\partial_y\alpha\right]\\
& \leq & C\int_{1\leq y\leq 2}\frac{(\alpha)^4}{y^4}+2\int \psi\frac{(\alpha)^3}{y^3}\left[\frac{Z}{y}\alpha-A\alpha\right]\\
& \leq & b^4+\int\psi\frac{|\alpha|^2}{y^6}-2\int\psi\frac{(\alpha)^3}{y^3}\left[\frac{1}{y}\alpha+A\alpha\right]\\
& \leq & b^2-2\int\psi\frac{(\alpha)^3}{y^3}\left[\frac{1}{y}\alpha+A\alpha\right]
\eee
where we used that $|Z(y)+1|\lesssim \frac{1}{y^2}$ for $y\geq 1$ and \fref{inteproloatedbound}. We now use H\"older and Sobolev inequalities to derive:
\bee
\label{hoheioh}
\nonumber 3\int \psi\frac{\alpha^4}{y^4}&  \lesssim &b^2+\int \psi\frac{\alpha^4}{y^4}+C\int|A\alpha|^4\lesssim  b^2+ \int \psi\frac{\alpha^4}{y^4}+\int \psi\frac{\alpha^4}{y^4}+\|A\alpha\|_{L^2}^{2}\|\pa_yA\alpha\|_{L^2}^2\\
& \lesssim & b^2+\delta(b^*)\|H\alpha\|_{L^2}^2+\int \psi\frac{\alpha^4}{y^4}
\eee
and \fref{newesitmate} follows.\\
The $L^{\infty}$ bound \fref{estgammalinft} follows from: $$|\Delta \gamma|^2_{L^{\infty}(y\geq 1)}\lesssim \left(\int_{y\geq 1}\frac{|\pa_y \Delta \gamma|^2}{y^2}\right)^{\frac 12}\left(\int_{y\geq 1} |\Delta \gamma|^2\right)^{\frac12}\lesssim b^3|\log b|^8.$$ 
It remains to prove \fref{interpolationbylapocia}. Using \fref{fromulagamma}, we treat the most delicate quadratic term, other terms are treated similarily and are easier to handle. We claim:
\be
\label{cnoieoeuerio}
\int|\alpha|^2|\Delta^2(\alpha^2)|^2\lesssim  \delta(b^*)\left(\mathcal E_4+\frac{b^4}{|\log b|^2}\right).
\ee
To treat the singularity at the origin, we write 
\bee
\Delta(\alpha^2) & = & 2|\pa_y\alpha|^2+2\alpha\Delta \alpha=2|\pa_y\alpha|^2-2\alpha H\alpha+2V\left(\frac{\alpha}{y}\right)^2\\
& = & 2|\pa_y\alpha|^2-2\alpha H\alpha+2V\left(A\alpha+\pa_y\alpha+\frac{1-Z}{y}\alpha\right)^2.
\eee
Taking another Laplacian and multiplying by $\alpha$ now yields sufficient vanishing at the origin to close the $L^2$ estimate \fref{cnoieoeuerio} near the origin using the estimates of Lemma \ref{interpolationw}. Far out, we write $$\Delta(\alpha^2)  =  2|\pa_y\alpha|^2-2\alpha H\alpha+2\left(\frac{\alpha}{y}\right)^2+2(V-1)\left(\frac{\alpha}{y}\right)^2.$$ The contribution of the last term is easily estimated using the extra decay $|V-1|\lesssim \frac{1}{1+y^2}$. We next compute: 
\bee
\int_{y\geq 1} &&|\alpha|^2|\Delta(|\pa_y\alpha|^2)|^2  \lesssim  \|\alpha\|_{L^{\infty}}^2\left[\int_{y\geq 1}\left(|\pa^2_{y}\alpha|^2+|\pa_y\alpha||\Delta \pa_y\alpha|\right)^2\right]\\
& \lesssim & \|\alpha\|_{L^{\infty}}^2\left[\int_{y\geq 1}\left(|H\alpha|^2+|\pa_y\alpha|| \pa_y H\alpha|+ 
\frac 1y |\pa_y\alpha|| H\alpha|+
\sum_{i,k=0}^1 \frac 1{y^{4-i-k} }|\pa_y^i\alpha| \, |\pa_y^k\alpha|
\right)^2\right]\\
& \lesssim & \left (\|\frac {\pa_y\alpha}y\|^2_{L^{\infty}(y\geq 1)}+\|H \alpha\|^2_{L^{\infty}(y\geq 1)}\right) b^2|\log b|^6
+\left(  \|\frac{\alpha}y\|^2_{L^{\infty}(y\geq 1)} + \|\pa_y\alpha\|_{L^{\infty}(y\geq 1)}^2\right) b^3|\log b|^{10}\\
& \lesssim & \delta(b^*)\frac{b^4}{|\log b|^2}.
\eee
Next, we write:
\bee
\Delta(\alpha H\alpha)& = & \Delta\alpha H\alpha+2\pa_y\alpha\pa_yH\alpha+\alpha\left(-H^2\alpha+\frac{V}{y^2}H\alpha\right)\\
& = &  \Delta\alpha H\alpha+2\pa_y\alpha\pa_yH\alpha-\alpha H^2\alpha +H\alpha(\Delta \alpha+H\alpha).
\eee The contribution of all other terms can be treated in a fashion similar to the previous argument.

This concludes the proof of Lemma \ref{interpolationwz}.


\subsection{Interpolation bounds for $\bw_2$}


We recall that $$\bw_2=\hJ\H\bw, \ \ \hJ=(e_z+\hbw)\wedge$$ and the decomposition from \fref{defwtwozero}: 
\be
\label{decompmpe}
\bw_2=\bw_2^0+\bw_2^1, \ \ \bw_2^0=R_z\H \bw^\perp, \ \ \bw_3=R_z\A\bw_2^0.
\ee
From the explicit definition \fref{vectorialhamiltonian} of $\H$, we have the formula:
\be
\label{wtoperp}
\bw^1_2=\bw_2^2+\hbw\wedge \H\bw, \ \ \bw_2^2=R_z\left|\begin{array}{lll}-2(1+Z)\pa_y\gamma\\0\\0\end{array}\right..
\ee

\begin{lemma}[Interpolation bounds for $\bw_2$]
\label{lemmeestimwtwo}
There holds:
\be
\label{htwobound}
\int|\hJ\H \bw|^2=\int|\bw_2|^2=\mathcal E_2+O(b^2|\log b|^2+\delta(b^*)\mathcal E_2),
\ee
\be
\label{vndoioehoe}
\int|\H \bw|^2\lesssim \mathcal E_2+b^2|\log b|^2,
\ee
\be\label{cenouyeoye}
\int \frac{|\H \bw|^2}{(1+y^4)(1+|\log y|^2)}\lesssim C(M)\mathcal E_4,
\ee
\be
\label{memewtwo}
\int \frac{|\bw^0_2|^2}{(1+y^4)(1+|\log y|^2)}\lesssim C(M)\mathcal E_4,
\ee
\be
\label{estlocal}
\int \frac{|\bw^1_2|^2}{(1+y^4)(1+|\log y|^2)}\lesssim \delta(b^*)\left(\frac{b^4}{|\log b|^2}+\mathcal E_4\right),
\ee
\be
\label{betterhtwoglobal}
\int|\H \bw_2|^2 \lesssim  C(M)\left(\mathcal E_4+\frac{b^4}{|\log b|^2}\right),
\ee
\be
\label{betterhtwoglobalbis}
\int|\hJ\H\bw_2|^2\lesssim\mathcal E_4+\frac{b^4}{|\log b|^2},
\ee
\be
\label{betterhtwo}
\int|\H \bw_2^1|^2+\int|R_z\H(R_z^2\bw_2^1)|^2\lesssim \delta(b^*)
\left(\mathcal E_4+\frac{b^4}{|\log b|^2}\right),
\ee
\be
\label{estprofile}
\int|\H \hJ\bw_2|^2\lesssim C(M)\frac{b^4}{|\log b|^2}.
\ee
\end{lemma}

{\bf Proof of Lemma \ref{lemmeestimwtwo}}:\\

{\bf step 1} Estimates for $\bw_2$.\\

 Oberve that 
 \be
 \label{cnoeneofe}
 \mathcal E_2=\int |\bw_0^2|^2, \ \ \mathcal E_4=\int |R_z\H \bw_2^0|^2=\int |\A^*\bw_3|^2.
 \ee We then estimate from \fref{estund}, \fref{estlinftysecondgamma}, \fref{inteproloatedbound}, \fref{lossyboundwperpgammainteprolated}:
 \bee
 \int| \H \bw|^2&\lesssim & \mathcal E_2+\int|\Delta \gamma|^2+\int\frac{1}{1+y^4}\left[|\pa_y\gamma|^2+|\pa_y\alpha|^2+\frac{|\alpha|^2}{y^2}\right]\\
 & \lesssim & \mathcal E_2+b^2|\log b|^2
 \eee
 which yields \fref{vndoioehoe}. Now from \fref{lossyboundwperpgammainteprolated}:
 $$
 \int|\bw_2^1|^2 \lesssim \int \frac{|\pa_y\gamma|^2}{1+y^4}+\delta(b^*)\|\H\bw\|_{L^2}^2\lesssim b^2|\log b|^2+\delta(b^*)\mathcal E_2$$
 which together with \fref{cnoeneofe} concludes the proof of \fref{htwobound}.\\
From the explicit definition of $\H$:
\bee
\int \frac{|\H \bw|^2}{(1+y^4)(1+|\log y|^2)}& \lesssim&  \int \frac{|H\alpha|^2+|H\beta|^2+|\Delta\gamma|^2}{(1+y^4)(1+|\log y|^2)} +\int\frac{1}{1+y^8}\left[|\pa_y\gamma|^2+|\pa_y\alpha|^2+\frac{|\alpha|^2}{y^2}\right]\\
& \lesssim & C(M)\mathcal E_4
\eee
and \fref{cenouyeoye} follows from \eqref{estun} and \eqref{estungamma}. For \eqref{memewtwo} we simply observe 
$$
|\bw^0_2|=|R_z\H\bw|\le |\H \bw|.
$$ 
On the other hand,
$$
|\bw_2^1|\le  |\hbw\wedge H\bw|\lesssim |\hbw|\, |\H \bw|
$$
and \eqref{estlocal} follows from the bound $\|\hbw\|_{L^{\infty}}\le \delta(b^*)$, see \eqref{estperplinfty} and 
\eqref{estperplinftygamma}.\\

{\bf step 2} Estimates for $\H\bw_2$.\\

We now turn to the $H^4$ estimates \eqref{betterhtwoglobal} and \fref{betterhtwo}. Recalling 
the decomposition \fref{decompmpe}
we first examine the quantity $\H \bw_2^0$:
$$
|\H \bw_2^0|\le |H^2\alpha|+|H^2\beta|+\frac 1{1+y^2} (|\pa_y H\beta|+\frac {|H\beta|}y)
$$
The estimate
$$
 \int |\H\bw_2^0|^2\lesssim C(M) \mathcal E_4
$$
now easily follow from the definition of $\mathcal E_4$, the coercivity bounds of Lemma \ref{lemmehardyhsquare}.
To obtain the improved bound for $\hJ \H \bw_2^0$ we simply note that 
$$
|R_z \H \bw_2^0|\le |H^2\alpha|+|H^2\beta|
$$
and 
$$
|\hat\bw\wedge \H \bw_2^0|\le \delta(b^*)  |H \bw_2^0|
$$
from which $$\int|\hJ\H\bw^0_2|^2\lesssim\mathcal E_4+\frac{b^4}{|\log b|^2}$$ and thus \fref{betterhtwoglobalbis} follows from the first part of \fref{betterhtwo}.\\
We now claim: 
\be
\label{iestouetoiure}
 \int |\H\bw_2^2|^2\lesssim \delta(b^*)\left(\frac{b^4}{|\log b|^2}+\mathcal E_4\right).
 \ee
Indeed, we compute:
$$\H \bw_2^2=-2\left[H((1+Z)\pa_y\gamma)\right]e_y
$$
and split the integral: $$\int |\H\bw_2^2|^2\leq  \int_{y\leq 1} |\H\bw_2^2|^2+\int_{y\geq 1} |\H\bw_2^2|^2.$$ The outer integral is easily estimated using the extra decay $|1+Z|\lesssim \frac{1}{1+y^2}$ and the estimates \eqref{estungamma} of Lemma \ref{interpolationwz}. For the inner integral, we use $|Z-1|\lesssim |y|^2$ to estimate: 
$$\int_{y\leq 1} |\H\bw_2^2|\lesssim \int_{y\leq 1} |H\pa_y\gamma|^2+ \delta(b^*)\left(\frac{b^4}{|\log b|^2}+\mathcal E_4\right).
$$ 
We then observe near the origin that for any function f: 
$$
Hf=-\pa_{yy}f-\frac{\pa_y}{y}f+\frac{1}{y^2}f+\frac{V-1}{y^2}f=-\pa_{yy}f+\frac{Af}{y}+\frac{V-1+Z-1}{y^2}f.
$$ 
Since
$|V-1|+|Z-1|\lesssim y^2$, the estimates of Lemma \ref{interpolationwz} imply: 
$$ 
\int_{y\leq 1} |\H\bw_2^2|^2\lesssim \int_{y\leq 1} \frac{|A\pa_y\gamma|^2}{y^2}+ 
\delta(b^*)\left(\frac{b^4}{|\log b|^2}+\mathcal E_4\right).
$$ 
The bound for the remaining term is given by \eqref{eq:Ay}.\\

{\bf step 3} Quadratic term.\\

 We claim:
 \be
 \label{toberppoveius}
  \int |\H(\hbw\wedge \H \bw)|^2\lesssim \delta(b^*)\left(\frac{b^4}{|\log b|^2}+\mathcal E_4\right).
 \ee
Indeed, first compute:
\be
\label{formualawege}
\hbw\wedge \H \bw=\left|\begin{array}{lll} \betah\left[-\Delta \gamma+2(1+Z)(\pa_y\alpha+\frac{Z}{y}\alpha)\right]-\gammah H\beta\\
-\alphah\left[-\Delta \gamma+2(1+Z)(\pa_y\alpha+\frac{Z}{y}\alpha)\right]+\gammah\left[H\alpha-2(1+Z)\pa_y\gamma\right]\\ \alphah H\beta-\betah(H\alpha-2(1+Z)\pa_y\gamma)\end{array}\right.
\ee
We now apply the $\H$ operator again and estimate all terms. 
We sketch the proof for the most delicate terms. 
\be
\label{epouoitbbsiuv}
\int |H(\betah\Delta \gamma)|^2\lesssim \int |\Delta \gamma|^2|H\betah|^2+\int |\pa_y\Delta \gamma|^2|\pa_y\betah|^2+\int |\betah|^2|\Delta^2\gamma|^2.
\ee The last integral is the most delicate term estimated from \fref{interpolationbylapocia}. For the other terms, we estimate using \eqref{choeouefeouie}, \eqref{estgammalapla} and \eqref{estgammalinft}:
\bee
\int |\Delta \gamma|^2|H\betah|^2&\lesssim & \|\Delta\gamma\|^2_{L^{\infty}(y\geq 1)}\int |H\betah|^2+
 \|H\betah\|^2_{L^{\infty}(y\leq 1)}\int_{y\leq 1} |\Delta\gamma|^2\\
 & \lesssim & b^2 |\log b|^6 b^3 |\log b|^8+ b^4 b^2|\log b|^8\lesssim  b^5.
\eee  
On the other hand, since
$$
\pa_y \Delta \gamma = \pa_y^3 \gamma - \frac 1y A\pa_y\gamma+\frac {Z-1}y \pa_y\gamma
$$
and $|Z-1|\lesssim y^2$ we can estimate from \eqref{estlinftydeux}, \eqref{estungamma} and \eqref{eq:Ay}
$$
\int_{y\le 1} |\pa_y\Delta \gamma|^2|\pa_y\betah|^2\lesssim \|\pa_y\betah\|_{L^\infty(y\le 1)} 
\int_{y\le 1} |\pa_y\Delta \gamma|^2\le b^6.
$$
For $y\ge 1$ we can interpolate between \eqref{estgammalinft} and \eqref{interpolationbylapocia}
to obtain 
$$
\int_{y\ge 1} |\pa_y\Delta \gamma|^2|\pa_y\betah|^2\lesssim b^5||\log b|^C+b^{\frac{10}3} \|\pa_y\beta\|_{L^\infty(y\ge 1)} 
\le b^5.
$$
For the term involving the last coordinate in \fref{formualawege}, we compute:
\bee
\int |\Delta(\alphah H\beta)|^2 & \lesssim & \int |\Delta \alphah|^2|H\beta|^2+\int |\pa_y\alphah|^2|\pa_yH\beta|^2+\int|\alphah|^2|\Delta H\beta|^2\\
& \lesssim & \int |\Delta \alphah|^2|H\beta|^2+\int |\pa_y\alphah|^2|\pa_yH\beta|^2+\int|\alphah|^2\frac{|H\beta|^2}{y^4}+\delta(b^*)\mathcal E_4.
\eee
Terms near the origin are easily estimated using Lemma \ref{interpolationwz}. Far out, the first two terms are easily treated and for the third one, we estimate from \eqref{estlinftysecond}: 
$$
\int_{y\geq 1} |\alphah|^2\frac{|H\beta|^2}{y^4}\lesssim b^4|\log b|^C\sqrt{b}+\|\frac {\alpha}{1+y^2}\|_{L^\infty}^2
 \int |H\beta|^2\lesssim  \delta(b^*) \frac{b^4}{|\log b|^2}
 $$
This concludes the proof of \fref{toberppoveius}. The second part of \fref{betterhtwo} and \fref{estprofile} can be obtained in a similar
fashion.\\
We omit the details.\\
This concludes the proof of Lemma \ref{lemmeestimwtwo}.



\section*{Appendix C: Proof of Lemma \ref{gainderivatives}}


This Appendix is devoted to the proof of Lemma \ref{gainderivatives} which is the key to handle the quasilinear stucture of the problem. The proof is mostly algebraic and makes an implicit use of the interpolation estimates of Appendix B.\\

{\bf step 1} Gain of two derivatives.\\

Let
$$a=\alpha e_r+\beta e_\tau+\gamma Q, \ \ \bG=\left|\begin{array}{lll}\alpha\\\beta\\\gamma\end{array}\right .$$ 
be a decomposition of the vector $a$ relative to the Frenet basis of $Q$ with $(\alpha,\beta,\gamma)=
(\alpha(y),\beta(y),\gamma(y))$ functions of the radial variable $y$, and $\alpha^2+\beta^2+(1+\gamma)^2=1$.
Then from \fref{ecomptuaiowquat}:
\be
\label{optuou}
\Delta a+|\nabla Q|^2a=-\Bbb H {\bG}.
\ee
We compute the action of derivatives 
\be
\label{comptuiotdeirvative}
\pa_ya=\pa_y \bG+M\bG, \ \  M\bG:=(1+Z)e_y\wedge {\bG},
\ee
$$
\frac 1y \pa_\theta a = N\bG,\ \ N\bG:=\left(\frac Zy e_z-(1+Z) e_x\right)\wedge {\bG}.
$$
We recall the double wedge formula: $$a\wedge(b\wedge c)=(a\cdot c)b-(a\cdot b)c.$$

Let also $$u=\hbw+e_z, \ \ |u|^2=1.$$ 
be a unit vector.
The proof of Lemma \ref{gainderivatives} is based on two computations.\\
 The first one relies on the action of the Laplace operator in the Frenet basis:
$$
\int a\cdot u\wedge \Delta a  =  \int \Delta a\cdot(a\wedge u)=\int a\cdot\Delta(a\wedge u)=\int a\cdot\left[\Delta a\wedge u+2\nab a\wedge \nab u\right]
$$
 and thus 
 \be
 \label{cnofeof}\int a\cdot u\wedge \Delta a =\int a\cdot(\nab a\wedge \nab u).
 \ee We now compute from \fref{optuou}:
 $$a\cdot\left[u\wedge \Delta a\right]=a\cdot\left[u\wedge (\Delta a+|\nabla Q|^2a)\right]=-\bG\cdot\left[(e_z+\hbw)\wedge \Bbb H\bG\right]=-\bG\cdot \hJ\Bbb H \bG.$$ Hence from \fref{comptuiotdeirvative}, \fref{cnofeof}:
\bee
& = & \int \bG\cdot\left[(\pa_y\bG+M\bG)\wedge(\pa_y\hbw+M(e_z+\hbw))+N\bG\wedge N(e_z+\hbw)\right]\\
& = & \int \bG\cdot\left[\pa_y\bG\wedge(\pa_y\hbw+M(e_z+\hbw))\right]+\int \bG\cdot\left[M\bG\wedge(\pa_y\hbw+M(e_z+\hbw))\right]\\
&+&\int \bG\cdot\left[N\bG\wedge N(e_z+\hbw)\right].
\eee
The second computation uses the normalization of $\hbw$:
\bee
\int \hJ\bG\cdot\left[\pa_y(\hJ\bG)\wedge\pa_y \hbw\right] & =  &\int \hJ\bG\cdot\left[(\pa_y \hbw\wedge\bG+(e_z+\hbw)\wedge \pa_y \bG)\wedge\pa_y \hbw\right]\\
& = &- \int\left[(\pa_y\hbw\cdot \bG)(\pa_y \hbw\cdot \hJ\bG)\right]
\eee
 and the structure of the operator $M$:
 \bee
 &&\int \hJ\bG\cdot\left[\pa_y(\hJ\bG)\wedge\left( M(e_z+\hbw)\right)\right]\\
& = & \int \hJ\bG\cdot\left[(\pa_y \hbw\wedge \bG+(e_z+\hbw)\wedge \pa_y \bG)\wedge(M(e_z+\hbw))\right]\\
& = & \int \hJ\bG\cdot\left[(\pa_y \hbw\wedge \bG)\wedge (M(e_z+\hbw))\right].
\eee
This generates a two derivatives gain:
  \bee
  \label{gaintwo}
&& \int \hJ\H \bG\cdot \hJ\H(\hJ\H \bG)=-\int \hJ\H \bG\cdot\left[\pa_y(\hJ\H \bG)\wedge(\pa_y \hbw+M(e_z+\hbw)\right]\\
 & - & \int \hJ\H \bA\cdot\left[M\hJ\H \bG\wedge(\pa_y\hbw+M(e_z+\hbw)\right]-\int \hJ\H \bG\cdot\left[N\hJ\H \bG\wedge N(e_z+\hbw)\right]\\
   & = & \int (\pa_y \hbw\cdot \H \bG)(\pa_y\hbw\cdot \hJ \H \bG)-\int \hJ\H \bG\cdot\left[(\pa_y \hbw\wedge \H \bG)\wedge(M(e_z+\hbw))\right]\\
& - &    \int \hJ\H \bG\cdot\left[M\hJ\H \bG\wedge(\pa_y\hbw+M(e_z+\hbw))\right] -\int \hJ\H \bG\cdot\left[N \hJ\H \bG\wedge N(e_z+\hbw)\right].
   \eee
We now observe from $$Me_z=(1+Z)e_x, \ \ Ne_z=(1+Z)e_y$$ the cancellation:
\bee
& - & \int \hJ\H \bG\cdot\left[M\hJ\H \bG\wedge Me_z\right]-\int \hJ\H \bG\cdot\left[N \hJ\H \bG\wedge Ne_z\right] \\
 &=&  -\int \hJ\H \bG\cdot\left[(1+Z)^2(e_y\wedge \hJ\H \bG)\wedge e_x+\left((\frac Zy e_z-(1+Z)e_x)\wedge \hJ\H \bG\right)\wedge (1+Z)e_y\right]\\ 
 & = & \int \hJ\H \bG\cdot\left[(1+Z)^2(e_x\cdot\hJ\H\bG)e_y+(1+Z)\left((e_y\cdot\hJ\H\bG)\frac Zye_z-(1+Z)(e_y\cdot\hJ\H\bG)e_x\right)\right]\\
 & = & \int \frac{Z(1+Z)}{y}(\hJ\H\bG\cdot e_y)(\hJ\H\bG\cdot e_z).
\eee
We have thus arrived at the formula:
\bea
\label{coancoufoeri}
&& \int \hJ\H \bG\cdot \hJ\H(\hJ\H \bG)\\
\nonumber & = & \int (\pa_y \hbw\cdot \H \bG)(\pa_y\hbw
\cdot \hJ \H \bG)-\int \hJ\H \bG\cdot\left[(\pa_y \hbW\wedge \H \bG)\wedge(M(e_z+\hbw))\right]\\
\nonumber & - &    \int \hJ\H \bG\cdot\left[M\hJ\H \bG\wedge(\pa_y\hbw+M\hbw)\right]\\
\nonumber & + & \int \frac{Z(1+Z)}{y}(\hJ\H\bG\cdot e_y)(\hJ\H\bG\cdot e_z)- \int \hJ\H \bG\cdot\left[N \hJ\H \bG\wedge N\hbw\right].
\eea

{\bf step 2} Leading order $b$ term.\\

Let us write $$\tbw=\tbw_0+\tbw_1, \ \ \tbw_0=b\tt_1 e_y.$$ We compute:
$$M\tbw_0=0,\ \ N\tbw_0=-b\tt_1\left(\frac Zy e_x+(1+Z)e_z\right).$$ This yields in particular the cancellation:
$$N\tbw_0\cdot\left(-(1+Z)e_x+\frac Zye_z\right)=0.$$
We now compute the leading order contribution of $\tbw_0$ to \fref{coancoufoeri}. First,
\bee
&&\int \frac{Z(1+Z)}{y}(\hJ\H\bG\cdot e_y)(\hJ\H\bG\cdot e_z)=\int \frac{Z(1+Z)}{y}(\hJ\H\bG\cdot e_y)\left[((\tbw_0+\tbw_1+\bw)\wedge\H\bG)\cdot e_z\right]\\
& = & \int  \frac{Z(1+Z)}{y}(\hJ\H\bG\cdot e_y)[-b\tt_1(\H \bG\cdot e_x)]+O\left(\left\|\frac{\tbw_1+\bw}{y(1+y^2)}\right\|_{L^{\infty}}\|\H\bG\|_{L^2}^2\right)\\
& = &-b\int \frac{Z(1+Z)}{y}\tt_1(\hJ\H \bG\cdot e_y)^2+O\left(\left[\left\|\frac{\tbw_1+\bw}{y(1+y^2)}\right\|_{L^{\infty}}+\left\|\frac{b\tt_1|\hbw|}{y(1+y)}\right\|_{L^{\infty}}
\right]\|\H\bG\|_{L^2}^2\right)\\
& = & -b\int \frac{Z(1+Z)}{y}\tt_1(\hJ\H \bG\cdot e_y)^2+O\left(b\delta(b^*)\|\H\bG\|_{L^2}^2\right)
\eee
where we used the estimates of Lemma \ref{interpolationw}.
Next:
\bee
& - &    \int \hJ\H \bG\cdot\left[M\hJ\H \bG\wedge\pa_y\tbw_0+N \hJ\H \bG\wedge N\tbw_0\right]\\
& = & b\int \hJ\H \bG\cdot\left[(1+Z)\pa_y\tt_1e_y\wedge(e_y\wedge \hJ\H \bG)+N\tbw_0\wedge\left((-(1+Z)e_x+\frac Zye_z)\wedge\hJ\H \bG\right)\right]\\
& = & b\int (1+Z)\pa_y\tt_1\left[(\hJ\H \bG\cdot e_y)^2-\|\hJ\H \bG\|^2\right]\\
& + & b\int\tt_1\left[(-(1+Z)(\hJ\H \bG\cdot e_x)+\frac Zy(\hJ\H \bG\cdot e_z))(-\frac{Z}{y}(\hJ\H \bG\cdot e_x)-(1+Z)(\hJ\H \bG\cdot e_z))\right]\\
& = & b\int (1+Z)A\tt_1(\hJ\H \bG\cdot e_x)^2\\
& + & b\int(\hJ\H \bA\cdot e_z)\left[-(1+Z)(\pa_y\tt_1+\frac{Z}{y}\tt_1)(\hJ\H \bG\cdot e_z)+\tt_1((1+Z)^2-\frac{Z^2}{y^2})(\hJ\H \bG\cdot e_x)\right]\\
& = & b\int (1+Z)A\tt_1(\hJ\H \bG\cdot e_x)^2+O\left(b\left\|(\frac{|\pa_y\tt_1}{1+y^2}+\frac{|\tt_1|}{y^2})|\bw|\right\|_{L^\infty}\|\H\bG\|_{L^2}^2\right)\\
& = & b\int (1+Z)A\tt_1(\hJ\H \bG\cdot e_x)^2+O\left(b\delta(b^*)\|\H\bG\|_{L^2}^2\right),
\eee
thanks to Lemma \ref{interpolationw}. Next,
\bee
& - & \int \hJ\H \bG\cdot\left[(\pa_y \tbw_0\wedge \H \bG)\wedge Me_z\right]= -b\int (1+Z)\pa_y\tt_1 \hJ\H \bG\cdot\left[(e_y \wedge \H \bG)\wedge e_x\right]\\ 
& = & b\int (1+Z)\pa_y\tt_1 (\hJ\H \bG\cdot e_y)(\H \bG\cdot e_x)\\ 
& = & b\int (1+Z)\pa_y\tt_1 (\hJ\H \bG\cdot e_y)\left((\hJ\H \bG\cdot e_y)-(\hbw\wedge \hJ\H\bG)\cdot e_y\right)\\
& = & b\int (1+Z)\pa_y\tt_1 (\hJ\H \bG\cdot e_y)^2+O\left(\|b(1+Z)\pa_y\tt_1w\|_{L^{\infty}}\|\H\bG\|_{L^2}^2\right)\\
& = & b\int (1+Z)\pa_y\tt_1 (\hJ\H \bG\cdot e_y)^2+O\left(b\delta(b^*)\|\H\bG\|_{L^2}^2\right).
\eee
Therefore we obtained
\bea
\label{coancoufoerifinal}
\nonumber && \int \hJ\H \bG\cdot \hJ\H(\hJ\H \bG)=  b\int (1+Z)A\tt_1\left[(\hJ\H\bG\cdot e_x)^2-(\hJ\H\bG\cdot e_y)^2\right]\\
\nonumber & + & \int (\pa_y \hbw\cdot \H \bG)(\pa_y\hbw\cdot \hJ \H \bG)-\int \hJ\H \bG\cdot\left[(\pa_y (\tbw_1+\bw)\wedge \H \bG)\wedge(M(e_z+\hbw))\right]\\
\nonumber & - &    \int \hJ\H \bG\cdot\left[M\hJ\H \bG\wedge(\pa_y\tilde{\bw}_+\pa_y\bw+M(\tilde{\bw}_1+\bw)\right]\\
\nonumber & - & \int \hJ\H \bG\cdot\left[N \hJ\H \bG\wedge N(\tilde{\bw}_1+\bw)\right]\\
& + & O\left(b\delta(b^*)\|\H\bG\|_{L^2}^2\right).
\eea

{\bf step 3} Upper bound on the quadratic term.\\

We now claim that 
\be
\label{estcurcieliop}
\forall y\geq 0, \ \ 0\leq(1+Z)A\tilde{T_1}\leq 1-d_1
\ee
for some universal constant $$0<d_1<1.$$ We prove the inequality for $T_1$, the claim for $\tt_1$ follows immediately. From  $$\frac{(\Lambda \phi)'}{\Lambda \phi}=\frac{Z}{y},$$ there holds: $$A^*(AT_1)=\frac{1}{y\Lambda \phi}\frac{\partial}{\partial y}\left(y\Lambda \phi AT_1\right)=\Lambda \phi$$ and thus, $$AT_1=\frac{1}{y\Lambda \phi}\int_0^y \tau (\Lambda\phi)^2d\tau.$$ Now $$1+Z=\frac{2}{1+y^2}=\frac{\Lambda\phi}{y}.$$
Therefore,
\be
\label{defj}
J(y)=(1+Z)A(T_1)=\frac{1}{y^2}\int_0^y\tau(\Lambda \phi)^2d\tau\geq 0.
\ee Let now $$f(y)=\int_0^y\tau(\Lambda \phi)^2d\tau-y^2,$$ then $$f'(y)=y(\Lambda \phi)^2-2y=\frac{4y^3}{(1+y^2)^2}-2y=\frac{1}{y}\left[4y^3-2y(1+2y^2+y^4)\right]\leq 0$$ and thus $$f(y)<f(0)=0 \ \ \mbox{for} \ \ y>0.$$ Hence $J(y)<1$ for $y>0$. Now $J(y)\to 0$ as $y\to +\infty$. It therefore attains its maximum at some $y_0\geq 0$ with $J(y_0)<1$, unless $y_0=0$ which is ruled out 
since  $J(0)=0$, and \fref{estcurcieliop} is proved.\\

{\bf step 4} Conclusion.\\

To control the remaining nonlinear terms in \fref{coancoufoerifinal} we use the estimates of Lemma \ref{interpolationw}. The first three terms are easily controlled:
$$ \left|\int (\pa_y \hbw\cdot \H \bG)(\pa_y\hbw\cdot \hJ \H \bG)\right|\lesssim \|\pa_y\hbw\|_{L^{\infty}}^2 \|\H\bG\|_{L^2}^2\lesssim b\delta(b^*)\|\H\bG\|_{L^2}^2,$$
\bee
\left|\int \hJ\H \bG\cdot\left[(\pa_y (\tbw_1+\bw)\wedge \H \bG)\wedge(M(e_z+\hbw))\right]\right|&\lesssim &\left\|\frac{|\pa_y \tbw_1|+|\pa_y\bw|}{1+y^2}\right\|_{L^{\infty}}\|\H\bG\|_{L^2}^2\\
& \lesssim & b\delta(b^*)\|\H\bG\|_{L^2}^2,
\eee
\bee
&&\left| \int \hJ\H \bG\cdot\left[M\hJ\H \bG\wedge(\pa_y\tilde{\bw}_+\pa_y\bw+M(\tilde{\bw}_1+\bw)\right]\right|\\
& \lesssim & \left[\left\|\frac{|\pa_y \tbw_1|+|\pa_y\bw|}{1+y^2}\right\|_{L^{\infty}}+\left\|\frac{|\tbw_1|+|\bw|}{1+y^4}\right\|_{L^{\infty}}\right]\|\H\bG\|_{L^2}^2\\
& \lesssim & b\delta(b^*)\|\H\bG\|_{L^2}^2.
\eee
The last term requires an additional cancellation to handle a singularity at the origin. Indeed,
\bee
& - & \int \hJ\H \bG\cdot\left[N \hJ\H \bG\wedge N(\tilde{\bw}_1+\bw)\right]=\int N\hJ\H \bG\cdot\left[ \hJ\H \bG\wedge N(\tilde{\bw}_1+\bw)\right]\\
& = & \int(\frac{Z}{y}e_z-(1+Z)e_x)\wedge\hJ\H \bG\cdot\left[\hJ\H \bG\wedge\left[(\frac{Z}{y}e_z-(1+Z)e_x)\wedge(\tilde{\bw}_1+\bw)\right]\right]\\
& = & \int \frac{Z^2}{y^2}e_z\wedge \hJ\H \bG\cdot\left[\hJ\H \bG\wedge\left[e_z\wedge(\tilde{\bw}_1+\bw)\right]\right]+O\left(\left\|\frac{|\bw_1|+|\bw|}{y(1+y^2)}\right\|_{L^{\infty}}\|\H\bG\|_{L^2}^2\right)\\
& = &- \int \frac{Z^2}{y^2}e_z\wedge \hJ\H \bG\cdot \left[(\hJ\H\bA\cdot e_z)(\tilde{\bw}_1+\bw)\right]+O\left(b\delta(b^*)\|\H\bG\|_{L^2}^2\right)\\
& = & \left(\left\|\frac{|\bw|(|\bw_1|+|\bw|)}{y^2}\right\|_{L^{\infty}}\|\H\bG\|_{L^2}^2\right)+O\left(b\delta(b^*)\|\H\bG\|_{L^2}^2\right)\\
& \lesssim & O\left(b\delta(b^*)\|\H\bG\|_{L^2}^2\right).
\eee
Combining these estimates together with \fref{estcurcieliop} and \fref{coancoufoerifinal}  concludes the proof of \fref{defquadra} and of Lemma \ref{gainderivatives}.

\end{document}